\documentclass[a4paper,10pt]{amsart}

\usepackage[usenames]{color}
\usepackage[utf8]{inputenc} % to use accents, etc.: å è ï ô ú
\usepackage[english]{babel} % determines the language of the document: 
\usepackage{amssymb, amsmath, amsthm}
\usepackage{amsfonts,mathtools}
\usepackage{thmtools}
\usepackage{geometry,graphicx}
\usepackage{enumerate}
\usepackage{multicol}

\usepackage{url}
\usepackage{enumitem}
\usepackage[hidelinks]{hyperref}

\usepackage{fancyhdr} % headings style
\usepackage{hyperref}
\usepackage[all]{xy}
\usepackage{tikz}
\usepackage{tikz}
\usepackage{tikz-cd}
\usetikzlibrary{cd}
\usetikzlibrary{arrows.meta}
\usetikzlibrary{arrows,calc,through,backgrounds,matrix,decorations.pathmorphing}
\usetikzlibrary{matrix}
\usetikzlibrary{arrows}
\newtheoremstyle{exampstyle}
  {3pt} % Space above
  {3pt} % Space below
  {\itshape} % Body font
  {} % Indent amount
  {\bfseries} % Theorem head font
  {. } % Punctuation after theorem head
  {.3em} % Space after theorem head
  {} % Theorem head spec (can be left empty, meaning `normal')
  
  \newtheoremstyle{exampstyle1}
  {3pt} % Space above
  {3pt} % Space below
  {} % Body font
  {} % Indent amount
  {\bfseries} % Theorem head font
  {. } % Punctuation after theorem head
  {.3em} % Space after theorem head
  {} % Theorem head spec (can be left empty, meaning `normal')

\declaretheoremstyle[%
  spaceabove=3pt,%
  spacebelow=3pt,%
  headfont=\normalfont\itshape,%
  postheadspace=.3em,%
  qed=\qedsymbol%
]{mystyle} 
\declaretheorem[name={Proof},style=mystyle,unnumbered,
]{dem}

\theoremstyle{exampstyle} % with Italic

\newtheorem{proposition}{Proposition}[section] % one counter for all 
\newtheorem{lemma}[proposition]{Lemma}
\newtheorem{question}[proposition]{Question}
\newtheorem{claim}[proposition]{Claim}

\newtheorem{theorem}[proposition]{Theorem} 
\newtheorem{corollary}[proposition]{Corollary}
\newtheorem{conjecture}[proposition]{Conjecture}

\theoremstyle{exampstyle1} %without Italic

\newtheorem{example}[proposition]{Example}
\newtheorem{examples}[proposition]{Examples}
\newtheorem{hecho}[proposition]{Fact}
\newtheorem{definition}[proposition]{Definition}
\newtheorem{notation}[proposition]{Notation}
\newtheorem{remark}[proposition]{Remark}

\numberwithin{equation}{section} % equations are numbered for each section
\pgfmathsetmacro{\myxlow}{-2}
\pgfmathsetmacro{\myxhigh}{2}
\pgfmathsetmacro{\myiterations}{6}

\DeclareMathOperator{\Aut}{Aut}

\DeclareMathOperator{\gl}{gl}

\DeclareMathOperator{\id}{id}

\DeclareMathOperator{\Id}{Id}
\DeclareMathOperator{\glu}{gl}
\DeclareMathOperator{\rec}{rec}
\DeclareMathOperator{\rk}{rk}

\DeclareMathOperator{\coh}{Coh}
\DeclareMathOperator{\Coh}{Coh}
\DeclareMathOperator{\Hom}{Hom}

\DeclareMathOperator{\Stab}{Stab}

\DeclareMathOperator{\Slice}{Slice}

\DeclareMathOperator{\Tr}{Tr}

\DeclareMathOperator{\TCoh}{TCoh}

\newcommand{\GL}{\widetilde{\textnormal{GL}}^{+}(2,\mathbb{R})}

\newcommand{\dtri}[3]{\xymatrix{#1 \ar[r] & #2 \ar[r] & #3 \ar[r] & #1[1]}}

\newcommand{\maps}[5]{#1 \colon \xymatrix@C=1.5em{ #2 \ar[r] & #3,} \xymatrix@C=1.5em{#4 \ar@{|->}[r] & #5;}}

\newcommand{\recoll}[9]{
\begin{tikzpicture}[->,>=stealth',shorten >=1pt,auto,node distance=2cm, thick]
  \node (1) {$#1$};
  \node (2) [right of=1] {$#2$};
  \node (3) [right of=2] {$#3$};
 
  \path
    (1) edge node [above, pos=0.5] {$#4$} (2)
    (2) edge [out=110,in=65] node[above, pos=0.5] {$#5$} (1)
    		edge [out=240,in=295] node[above, pos=0.5] {$#6$} (1)
    (2) edge node [above, pos=0.5] {$#7$} (3)
    (3) edge [out=110,in=65] node[above, pos=0.5] {$#8$} (2)
    		edge [out=240,in=295] node[above, pos=0.5] {$#9$} (2);
\end{tikzpicture}}

\newcommand{\la}{\rightarrow}
\newcommand{\lai}{\hookrightarrow}

\newcommand{\bcla}{\begin{claim}}
\newcommand{\ecla}{\end{claim}}
\newcommand{\bp}{\begin{proposition}}
\newcommand{\ep}{\end{proposition}}
\newcommand{\brem}{\begin{remark}}
\newcommand{\erem}{\end{remark}}
\newcommand{\bd}{\begin{definition}}
\newcommand{\ed}{\end{definition}}
\newcommand{\bl}{\begin{lemma}}
\newcommand{\el}{\end{lemma}}
\newcommand{\bh}{\begin{hecho}}
\newcommand{\eh}{\end{hecho}}
\newcommand{\bq}{\begin{question}}
\newcommand{\eq}{\end{question}}
\newcommand{\bo}{\begin{obs}}
\newcommand{\eo}{\end{obs}}
\newcommand{\bc}{\begin{corollary}}
\newcommand{\ec}{\end{corollary}}
\newcommand{\bcon}{\begin{conjecture}}
\newcommand{\econ}{\end{conjecture}}
\newcommand{\bnot}{\begin{notation}}
\newcommand{\enot}{\end{notation}}
\newcommand{\bdem}{\begin{dem}}
\newcommand{\edem}{\end{dem}}
\newcommand{\benum}{\begin{enumerate}}
\newcommand{\eenum}{\end{enumerate}}
\newcommand{\bitem}{\begin{itemize}}
\newcommand{\eitem}{\end{itemize}}
\newcommand{\bes}{\begin{examples}}
\newcommand{\ees}{\begin{examples}}

\newcommand{\be}{\begin{example}}
\newcommand{\ee}{\end{example}}
\newcommand{\bt}{\begin{theorem}}
\newcommand{\et}{\end{theorem}}

\newcommand{\R}{\mathbb{R}}
\newcommand{\Ot}{\mathcal{O}}

\newcommand{\N}{\mathcal{N}}

\newcommand{\T}{\mathcal{T}}
\newcommand{\A}{\mathcal{A}}
\newcommand{\Ser}{\mathcal{S}}

\newcommand{\D}{\mathcal{D}}

\newcommand{\Z}{\mathbb{Z}}

\newcommand{\Co}{\mathbb{C}}

\newcommand{\C}{\mathcal{C}}
\newcommand{\HH}{\mathcal{H}}

\newcommand{\Li}{\mathcal{L}}
\newcommand{\VV}{\mathcal{V}}
\newcommand{\B}{\mathcal{B}}
\newcommand{\DD}{\mathbb{D}}

\newcommand{\FF}{\mathcal{F}}

\newcommand{\HHom}{\mathcal{H}om}
\newcommand{\RHHom}{R\mathcal{H}om}

\newcommand{\lin}{\langle}
\newcommand{\rin}{\rangle}

\newcommand{\Pd}{\mathcal{P}}

\DeclarePairedDelimiter\abs{\lvert}{\rvert}%

\DeclareMathOperator{\Ker}{Ker}

\DeclareMathOperator{\Coker}{Coker}
\DeclareMathOperator{\Img}{Img}

\DeclareMathOperator{\Kom}{Kom}

\DeclareMathOperator{\Mor}{Mor}

\DeclareMathOperator{\rank}{rank}

\DeclareMathOperator{\SL}{SL}
\DeclareMathOperator{\Vect}{Vect}

\newcommand{\mapsi}[5]{#1 \colon \xymatrixrowsep{0.25pc}\xymatrix{ #2 \ar[r] & #3\\
#4 \ar@{|->}[r] & #5
}}

\newcommand\restr[2]{{% we make the whole thing an ordinary symbol
  \left.\kern-\nulldelimiterspace % automatically resize the bar with \right
  #1 % the function
  \vphantom{\big|} % pretend it's a little taller at normal size
  \right|_{#2} % this is the delimiter
  }}

\newcommand{\sur}{\twoheadrightarrow}

\makeatletter
\newsavebox\myboxA
\newsavebox\myboxB
\newlength\mylenA
\newcommand*\xoverline[2][0.75]{%
    \sbox{\myboxA}{$\m@th#2$}%
    \setbox\myboxB\null% Phantom box
    \ht\myboxB=\ht\myboxA%
    \dp\myboxB=\dp\myboxA%
    \wd\myboxB=#1\wd\myboxA% Scale phantom
    \sbox\myboxB{$\m@th\overline{\copy\myboxB}$}%  Overlined phantom
    \setlength\mylenA{\the\wd\myboxA}%   calc width diff
    \addtolength\mylenA{-\the\wd\myboxB}%
    \ifdim\wd\myboxB<\wd\myboxA%
       \rlap{\hskip 0.5\mylenA\usebox\myboxB}{\usebox\myboxA}%
    \else
        \hskip -0.5\mylenA\rlap{\usebox\myboxA}{\hskip 0.5\mylenA\usebox\myboxB}%
    \fi}
  \makeatother
  \makeatletter
\renewcommand{\tocsection}[3]{%
  \indentlabel{\@ifnotempty{#2}{\bfseries\ignorespaces#1 #2\quad}}\bfseries#3}
\renewcommand{\tocsubsection}[3]{%
  \indentlabel{\@ifnotempty{#2}{\ignorespaces#1 #2\quad}}#3}
\newcommand\@dotsep{4.5}
\def\@tocline#1#2#3#4#5#6#7{\relax
  \ifnum #1>\c@tocdepth % then omit
  \else
    \par \addpenalty\@secpenalty\addvspace{#2}%
    \begingroup \hyphenpenalty\@M
    \@ifempty{#4}{%
      \@tempdima\csname r@tocindent\number#1\endcsname\relax
    }{%
      \@tempdima#4\relax
    }%
    \parindent\z@ \leftskip#3\relax \advance\leftskip\@tempdima\relax
    \rightskip\@pnumwidth plus1em \parfillskip-\@pnumwidth
    #5\leavevmode\hskip-\@tempdima{#6}\nobreak
    \leaders\hbox{$\m@th\mkern \@dotsep mu\hbox{.}\mkern \@dotsep mu$}\hfill
    \nobreak
    \hbox to\@pnumwidth{\@tocpagenum{\ifnum#1=1\bfseries\fi#7}}\par% <-- \bfseries for \section page
    \nobreak
    \endgroup
  \fi}
\AtBeginDocument{%
\expandafter\renewcommand\csname r@tocindent0\endcsname{0pt}
}

\makeatother

\makeatletter
\renewcommand\subsubsection{\@startsection{subsubsection}{3}%
  \z@{.5\linespacing\@plus.7\linespacing}{-.5em}%
  {\normalfont\bfseries}}
\makeatother

\title[BRIDGELAND STABILITY CONDITIONS ON THE CATEGORY OF HOLOMORPHIC TRIPLES]{BRIDGELAND STABILITY CONDITIONS ON THE CATEGORY OF HOLOMORPHIC TRIPLES OVER CURVES}

\author{EVA MART\'{I}NEZ-ROMERO}
\address{Freie Univeristät Berlin, Arnimallee 3, 14195 Berlin, Germany}
\email{3v4.mr6@gmail.com}

\author{ALEJANDRA RINC\'{O}N-HIDALGO}
\address{ICTP.
Leonardo Da Vinci Building,
Strada Costiera 11,
34151 Trieste, 
Italy}
\email{arincon@ictp.it}
\urladdr{http://users.ictp.it/~arincon/}
\author{ARNE R\"{U}FFER}
\address{Mary Immaculate College, University of Limerick, Limerick, Ireland}
\email{arne.rueffer@mic.ul.ie}

\begin{document}
\maketitle
\begin{abstract}
We give a complete description of the Bridgeland stability manifold for the bounded derived category of holomorphic triples over a smooth projective curve of genus 1 as a connected, four dimensional complex manifold.
\end{abstract}
\setcounter{section}{0}

\tableofcontents

%------------------------------------------------
\section{Introduction}

The concept of a holomorphic triple was introduced by Bradlow and Garc\'ia-Prada in \cite{GP2} and \cite{GP1}. The abelian category of holomorphic triples $\TCoh(C)$ consists of objects of the form $\varphi\colon E_1 \rightarrow E_2$, where $E_1,E_2$ are coherent sheaves on a nonsingular projective curve $C$ together with a morphism $\varphi$ between them. It was shown in \cite{GP2} that moduli spaces of semistable holomorphic triples of vector bundles exist and are projective. This category has also played an important role in the study of Higgs bundles \cite{bradlow2003surface}. Assuming $C$ to have genus 1, we will provide a complete description of the Bridgeland stability manifold of the triangulated category $\T_C\coloneqq D^b(\TCoh(C))$. 

Stability conditions on triangulated categories have been introduced by Bridgeland in \cite{BS1} as a formalisation of Douglas' work in \cite{BSD1.2} and \cite{BSD2}. The main result in \cite{BS1} asserts that the set of stability conditions has the structure of a complex manifold (usually referred to as stability space). Stability manifolds have several applications in algebraic geometry as for example they serve as an important aid for the understanding of derived categories or as a tool in birational geometry (\cite{huizenga17}, \cite{BS6}). The description of stability spaces is not an easy endeavour in geometric situations. For nonsingular projective curves it is well understood (\cite{BS1}, \cite{KB}, \cite{M1}, \cite{O1}) and our strategy is to use this as a building stone in the case of holomorphic triples. 

Our first finding is that $D^b(\Coh(C))$ can be embedded as a strictly full subcategory into $\mathcal{T}_{C}$ in three different ways. These are 
\begin{alignat}{8}\label{eq:introembeddings}
\mathcal{C}_1 &\coloneqq   \{& X  &\rightarrow&  0 &\colon & X  \in D^b(\Coh(C)) \} & \textnormal{ } \subset \T_C ,\\ \nonumber 
\mathcal{C}_2 &\coloneqq \{& 0  &\rightarrow&  X &\colon & X \in D^b(\Coh(C)) \} & \textnormal{ } \subset  \T_C, \\ \nonumber
\mathcal{C}_3 &\coloneqq  \{ & X &\xrightarrow{\id}& X& \colon& X   \in D^b(\Coh(C)) \} &\textnormal{ }\subset \T_C 
\end{alignat}
and subsequent pairing of two of the resulting subcategories, respectively, leads to three semiorthogonal decompositions
\begin{align}\label{eq:intro:sod}
\mathcal{T}_{C}  =\langle \mathcal{C}_1, \mathcal{C}_2 \rangle  \textnormal{, } \T_C= \langle \mathcal{C}_2, \mathcal{C}_3 \rangle \textnormal{ and } \T_C=\langle \mathcal{C}_3, \mathcal{C}_1 \rangle.
\end{align}
Following \cite{BK}, we can prove the existence of the Serre functor on $\mathcal{T}_{C}.$ Additionally, the reiterated application of the Serre functor to any of the semiorthogonal decompositions \eqref{eq:intro:sod} leads to permuting through all three of them.

We construct discrete pre-stability conditions on $\T_C$ in two different ways. The first approach is to construct pre-stability conditions on $\T_C$ by using the semiorthogonal decompositions \eqref{eq:intro:sod} to glue stability conditions from $\Stab(D^b(\Coh(C))$ following \cite{CP}. Moreover, in the Appendix we compare this gluing procedure with the well-known recollement from \cite{BBD}, which provides the same hearts when the conditions for gluing as in \cite{CP} (we refer to them as CP-glued hearts) are satisfied and we prove that in the case of $\T_C$ we cannot obtain stability conditions from hearts constructed by recollement using hearts that are not CP-glued hearts. The second approach is to tilt the standard heart in $\T_C$ with respect to a certain torsion pair following \cite{HRS}. 

The next step is then to study the structure of the stability manifold $\Stab(\T_C)$. We  first prove that for each pre-stability condition at least two of the three embeddings of $D^b(C)$ into $\mathcal{T}_{C}$ have the property that they map all line bundles and all skyscrapers to stable objects in Theorem \ref{TEO1}. This finding is a crucial step in the process of describing the entire stability space as a connected complex manifold.

It turns out, however, that there are stability conditions that are not in a $\GL$-orbit with a stability condition obtained by CP-gluing. In Proposition \ref{P01} we obtain then, that up to the $\GL$-action, the stability space of $\mathcal{T}_{C}$ is given by the stability conditions obtained via either CP-gluing via the semiorthogonal decompositions \eqref{eq:intro:sod} or by tilting with regard to the torsion pair from Lemma \ref{TPNG}.

In order to prove our main result, we need to verify that tilting and CP-gluing actually produces stability conditions in our situation. This includes the verification of the support property. To prove that the support property is fulfilled for CP-gluing pre-stability conditions we use a generalisation of some of the equations of \cite{GP2} to arbitrary stability conditions. For the non-gluing case, under the condition $g(C)=1,$ we use the Euler form as a Bogomolov-type inequality. We conjecture that the support property will hold for genus $g(C)>1$ as well. We will now extend the  Harder--Narashiman-property using Bridgeland's deformation result to the non-discrete case as in \cite{BS8} and \cite{BMLP}. Hence, we obtain the main theorem of this paper. 

\begin{theorem}[Theorem \ref{TEO2}]
If
\begin{enumerate}
\item $g(C) = 1$ or
\item $g(C) > 1$ and all non-gluing pre-stability conditions with negative discriminant satisfy the support property,
\end{enumerate}
then $\Stab(\mathcal{T}_{C})$ is a connected 4-dimensional complex manifold.
\end{theorem}

Finally, we conjecture that if $g(C)=1,$ then $\mathcal{T}_{C}$ is a fractional Calabi--Yau of fractional dimension $\frac{4}{3}$.

As far as the organisation of this paper is concerned, we proceed as follows: after providing the necessary framework in Section 2, we go on to introduce and subsequently make use of the necessary techniques we require to construct pre-stability conditions on $\mathcal{T}_{C}$ in Section 3. In Section 4 we study the stability of skyscraper sheaves and line bundles under the embeddings \eqref{eq:introembeddings} and we prove that every stability condition on $\T_C$ is obtained by the techniques introduced in Section 3. Next, the support property is proved in Section 5, while a topological description of the stability space is provided in Section 6. Finally, in the Appendix we compare the two procedures to glue hearts from semiorthogonal decompositions and justify why we restricted the first part of our construction to CP-glued hearts.

\subsubsection*{Acknowledgments}We thank Bernd Kreussler and Alexander Schmitt for their invaluable advice on the research necessary to produce this article. We also wish to thank Kotaro Kawatani for pointing out a mistake in the proof of Lemma \ref{GKRT} in the preliminary version of this paper. We thank the Freie Universit\"{a}t Berlin, the Berlin Mathematical School, the IRTG GRK 1800 and the Mathematics department of Mary Immaculate College, University of Limerick for the financial support provided to the authors, as well as making possible to meet in order to develop this work. 

%------------------------------------------------
\section{Preliminaries}

\subsection{Review of t-structures and semiorthogonal decompositions.}

Let $\mathcal{D}$ be a $\Co$-linear triangulated category  of finite type. For the general theory of t-structures we suggest \cite[Chapter 1]{BBD}.

\bd A \emph{t-structure} on a triangulated category $\D$ consists of a pair of full additive subcategories  $(\D^{\leq 0},\D^{\geq 0})$, with $\D^{\leq i} \coloneqq \D^{\leq 0}[-i]$ and $\D^{\geq i} \coloneqq \D^{\geq 0}[-i]$ for $i \in \Z$, such that:
\begin{enumerate}[leftmargin=0.5cm]
\item $\Hom_\D(\D^{\leq 0}, \D^{\geq 1})=0$.
\item For all $E \in \D$, there is a distinguished triangle $T\la E\la F\la T[1]$
 with $T \in \D^{\leq 0}$ and $F \in \D^{\geq 1}$.
\item $\D^{\leq 0} \subset \D^{\leq 1}$ and $\D^{\geq 0} \supset \D^{\geq 1}$.
\end{enumerate}
A t-structure is \emph{bounded} if every $E\in \D$ is contained in $D^{\leq n} \cap D^{\geq -n}$ for some $n>0.$
\ed
\bd \label{HEARTBT} The \emph{heart} of a bounded t-structure $(\D^{\leq 0},\D^{\geq 0})$ is defined as $\A\coloneqq \D^{\leq 0}\cap \D^{\geq 0}.$
\ed
\brem \href{https://arxiv.org/pdf/1607.01262.pdf}{\cite[Lemma \ 5.2]{MS1}} The heart of a bounded t-structure $\A\subseteq \D$ is an abelian category whose short exact sequences are precisely the exact triangles in $\D$ with objects in $\A.$ A morphism $A\la B$ between two objects in $\A$ is defined to be an inclusion if its cone is also in $\A,$ and it is defined to be a surjection if the cone is in $\A[1].$ 
\erem
\be \cite[Chapter 1]{BBD} \begin{enumerate}[leftmargin=0.5cm] \item If $(\D^{\leq 0},\D^{\geq 0})$ is a t-structure, then the pair $(\D^{\leq n},\D^{\geq n})$ is also a t-structure. 
 \item Let $\D=D^b(\A),$ where $\A$ is an abelian category. The standard bounded t-structure of $\D$ is given by $\D^{\leq 0}\coloneqq \{E\mid H^i(E)=0\textnormal{ , }i>0\}$ and $\D^{\geq 0}\coloneqq\{E\mid H^i(E)=0 \textnormal{ , }i<0\}.$ Note that $\A$ is the heart of the standard bounded t-structure.  
\end{enumerate}
\ee

%\bd \label{t-exact} Given triangulated categories $\D$ and $\D'$ endowed with t-structures \linebreak
%$(\D^{\leq 0},\D^{\geq 0})$ and $(\D'^{\leq 0},\D'^{\geq 0})$, a functor $F\colon  \D\la \D'$  is called (\emph{left})\emph{right
%t-exact} if \linebreak ($F(\D^{\leq 0})\subset \D'^{\geq 0}$) $F(\D^{\geq 0})\subset \D'^{\geq 0}$. We say that $F$ is t-exact if it is left and right t-exact.
%\ed
\bl \label{heartVStstructure} \textup{\href{http://annals.math.princeton.edu/wp-content/uploads/annals-v166-n2-p01.pdf}{\cite[Lemma \ 3.2]{BS1}}},  \textup{\href{https://arxiv.org/pdf/1111.1745.pdf}{\cite[Rem. 1.16]{HUYSC}}}
Let $\A \subset \D$ be a full additive subcategory of a triangulated category $\D$. Then $\A$ is the heart of a bounded t-structure if and only if
\begin{enumerate}[leftmargin=0.5cm]
\item $\Hom_{\D}(\A[k_1],\A[k_2])=0$ for $k_1 > k_2$.
\item For every nonzero $E \in \D$ there exists a finite sequence of integers
$$k_1 > k_2 > \cdots > k_m$$
and a collection of distinguished triangles
$$\xymatrix@C=1em{
	0=E_0 \ar[rr] & & E_1 \ar[dl] \ar[rr] & & E_2 \ar[dl] \ar[r] & \cdots \ar[r] & E_{m-1} \ar[rr] & & E_m = E \ar[dl]\\
	& A_1  \ar@{-->}[ul] & & A_2  \ar@{-->}[ul] & & & & A_m  \ar@{-->}[ul] &
}$$
with $A_j \in \A[k_j]$ for all $j$.
\end{enumerate}
\el
\brem Let $\A\subseteq \D$ be the heart of a bounded t-structure. For every object $E\in \D,$ the objects $A_j\in \A[k_j]$ are its \emph{cohomological objects} with respect to $\A.$ They are denoted by $A_i=H_{\A}^{-k_i}(E).$ Moreover, they induce a cohomological functor i.e.\ they are functorial and induce a long exact sequence of cohomology for any exact triangle. Note that 
a t-structure is uniquely determined by its heart (see \href{https://arxiv.org/pdf/1111.1745.pdf}{\cite[Def. 1.13]{HUYSC}}).
\erem

Hearts of bounded t-structures play an important role in the definition of Bridgeland stability conditions. Therefore, we will discuss different ways of giving hearts of bounded t-structures. 

\subsubsection*{Torsion pairs and tilting}

\bd Let $\A$ be an abelian category. A \emph{torsion pair} for $\A$ consists of a pair $(\mathcal{T}, \mathcal{F})$ of full subcategories such that
\begin{enumerate}[leftmargin=0.5cm]
\item $\Hom_\A(\mathcal{T}, \mathcal{F})=0$.
\item For all $E \in \A$, there is a short exact sequence
$ 0 \rightarrow T \rightarrow E \rightarrow F \rightarrow 0$
with $T \in \mathcal{T}$ and $F \in \mathcal{F}$.
\end{enumerate}
\ed

\bp \label{Tilting} \textup{\cite[Prop.\ 2.1]{HRS}} and  \textup{\cite[Lemma \ 6.3]{MS1}} Let $\A$ be the heart of a bounded t-structure on $\mathcal{D}.$ Given a torsion pair $(\T,\FF)$ in $\A,$ then the full subcategory
$$\A^{\sharp}=\{E\in \D\mid H_\A^{i}(E)=0 \textnormal{ for } i\notin \{-1,0\}\textnormal{, } H_{\A}^{-1}(E)\in \FF \textnormal{ and } H_{\A}^{0}(E)\in \T \}$$ is the heart of a bounded t-structure on $\D.$ We call $\A^{\sharp}$  the \emph{tilt} of $\A$ with respect to $(\T,\FF).$
Moreover, the torsion pair $(\T,\FF)$ gives rise to the torsion pair $(\FF[1],\T)$ for the tilt $\A^{\sharp}.$  
\ep
\brem Note that $\A^{\sharp}[-1]=(\FF,\T[-1])$ is also the heart of a bounded t-structure. The heart $\A^{\sharp}[-1]=(\FF,\T[-1])$ is called the \emph{right tilt} of $\A$ with respect to $(\T,\FF).$
\erem

\subsubsection*{Admissible subcategories and semiorthogonal decompositions}

In this section we introduce concepts to study a triangulated category $\D,$ namely semiorthogonal decompositions and Serre functors. We refer to \cite{HUYFM} and \cite{BK}. If $X$ is a smooth projective variety we refer to $D^b(\Coh(X))$ just as $D^b(X).$ 

\begin{definition}
Let  $\mathcal{D}$ be a triangulated category. Let $\mathcal{A}\subseteq \mathcal{D}$ be a strictly full triangulated subcategory. The category $\mathcal{A}$ is \emph{(left) right-admissible} if the inclusion functor $i\colon \mathcal{A}\rightarrow \mathcal{D}$ has a (left) right adjoint ($ i^*\colon \mathcal{D}\rightarrow \mathcal{A}$) $i^!\colon \mathcal{D} \rightarrow\mathcal{A}.$ It is called \emph{admissible} if it is right and left admissible.
\end{definition}

\begin{definition}
Let $\mathcal{D}$ be a triangulated category. A \textit{semiorthogonal decomposition} of $\mathcal{D}$ consists of a collection $\mathcal{A}_1,\cdots, \mathcal{A}_n$ of full triangulated subcategories such that
\begin{enumerate}[leftmargin=0.5 cm]
\item $\Hom_{\mathcal{D}}(\mathcal{A}_i, \mathcal{A}_j)=0$ for every $1 \leq j < i \leq n$.
\item $\mathcal{D}$ is generated by the $\mathcal{A}_i$.
\end{enumerate}
We write $\mathcal{D}=\langle \mathcal{A}_1,\cdots,\mathcal{A}_n \rangle$.
\end{definition}

\begin{lemma}[{\cite{BK}}]\label{lemma: sodecomp}
Let $\mathcal{D}$ be a triangulated category. Let $\mathcal{A}$ and $\mathcal{B}$ be strictly full triangulated subcategories of $\mathcal{D}$. Assume that $\Hom_{\mathcal{D}}(\mathcal{B},\mathcal{A})=0$. Then, the following are equivalent:
\begin{enumerate}[leftmargin=0.5 cm]
\item The category $\mathcal{D}$ is generated by $\mathcal{A}$ and $\mathcal{B}$ i.e.\ for each $X \in \mathcal{D}$, there exists a distinguished triangle $B\la X \la A \la B[1]$
with $A \in \mathcal{A}$ and $B \in \mathcal{B}$ i.e.\ $\mathcal{D}=\langle \mathcal{A}, \mathcal{B} \rangle.$
\item $\mathcal{B}= \prescript{\perp}{}{\mathcal{A}} \coloneqq \{D \in \mathcal{D} \mid \Hom_{\mathcal{D}}(D, A)=0 \textnormal{ for all } A\in \mathcal{A} \}$ and there exists a functor $i^* \colon \mathcal{D} \rightarrow \mathcal{A}$ which is left adjoint to the inclusion $i \colon \mathcal{A} \hookrightarrow \mathcal{D}$ i.e.\ $\mathcal{A}$ is left admissible. 
\item $\mathcal{A}= \mathcal{B}^\perp  \coloneqq \{D \in \mathcal{D} \mid \Hom_{\mathcal{D}}(B,D)=0 \textnormal{ for all } B\in \mathcal{B} \}$ and there exists a functor $j^! \colon \mathcal{D} \rightarrow \mathcal{B}$ which is right adjoint to the inclusion $j \colon \mathcal{B} \hookrightarrow \mathcal{D}$ i.e.\ $\mathcal{B}$ is right admissible.  
\end{enumerate} 
\end{lemma}

\begin{remark} Let $F\colon\mathcal{D}\rightarrow\mathcal{D}$ be an autoequivalence. If there is a semiorthogonal decomposition $\mathcal{D}=\langle \mathcal{A}, \mathcal{B} \rangle,$ then $\mathcal{D}=\langle F(\mathcal{A}), F(\mathcal{B}) \rangle.$
\end{remark}

%\begin{definition}
%Let $\mathcal{D}$ be a triangulated category. Let $\mathcal{A}$ and $\mathcal{B}$ be full triangulated subcategories of $\mathcal{D}$ with $\Hom_{\mathcal{D}}(\mathcal{B},\mathcal{A})=0$. If any of the equivalent statements stated in Lemma \ref{lemma: sodecomp} is satisfied,  we say that $\mathcal{A}$ is \textit{left admissible}, $\mathcal{B}$ is \textit{right admissible}. We say that a full subcategory of $\mathcal{D}$ is \textit{admissible} if it is both left and right admissible.
%\end{definition}

%\begin{lemma}[{\cite{BK}}]\label{lem:OrthogonalCatsEquivalence}
%Let $\mathcal{T}$ be a triangulated category and let $\mathcal{A}$ be full triangulated subcategory of $\mathcal{T}$. If $\mathcal{A}$ is admissible, we find an equivalence of triangulated categories $\prescript{\perp}{}{\mathcal{A}} \cong \mathcal{A}^\perp$.
%\end{lemma}

\begin{definition}
Let $\mathcal{D}$ be a $k$-linear triangulated category of finite type, where $k$ is a field. We say that $\mathcal{D}$ is \textit{right (resp.\  left) saturated} if every contravariant (resp.\ covariant) cohomology functor $\mathcal{D} \rightarrow \Vect_{k}$ of finite type is representable.
\end{definition}

\begin{proposition}[{\cite[Proposition 2.6]{BK}}]\label{prop:Satimpliesad}
Let $\mathcal{A}$ be right (resp.\ left) saturated. Suppose $\mathcal{A}$ is embedded in a triangulated category $\mathcal{D}$ as a full triangulated subcategory. Then $\mathcal{A}$ is right (resp.\ left) admissible.
\end{proposition}

\begin{theorem}[{\cite[Theorem 2.14]{BK}}]
Let $X$ be a smooth projective variety. Then, $D^b(X)$ is right and left saturated.
\end{theorem}
\subsubsection*{Serre functor}
Let $\mathcal{D}$ be a $\mathbb{C}$-linear triangulated category.
\begin{definition}A \textit{Serre functor} on $\mathcal{D}$ is an exact autoequivalence $S \colon \mathcal{D} \rightarrow \mathcal{D}$ such that for any $E,F \in \mathcal{D}$, there is an isomorphism
$\eta_{E,F} \colon \Hom_\mathcal{D}(E,F) \rightarrow \Hom_\mathcal{D}(F,S(E))^*$
(of $\mathbb{C}$-vector spaces) which is functorial in $E$ and $F$.
\end{definition}

\begin{remark}
For $\mathcal{D}$ of finite type, a Serre functor, if it exists, is unique up to isomorphism. Moreover, it commutes  with equivalences, i.e.\ for $F:\mathcal{D} \rightarrow \mathcal{D}'$ an equivalence, $S_{\mathcal{D}'}\circ F \cong F \circ S_\mathcal{D}$. Furthermore, given an admissible subcategory $\mathcal{X} \subset \mathcal{D},$ by Serre duality $S_{\mathcal{D}}$ sends $\prescript{\perp}{}{\mathcal{X}}$ to $\mathcal{X}^\perp.$ 
\end{remark}

\begin{example}\label{Serre}
Let $X$ be a smooth projective variety defined over $\mathbb{C}.$ The autoequivalence of $D^b(X)$ given by $\Ser_{X}(E)= E\otimes \omega_{X}[\dim X],$ where $\omega_X$ is the dualizing line bundle, is the Serre functor on $D^b(X)$. In particular, if $X$ is a Calabi--Yau variety, its Serre functor is simply the shift functor $S_X=[\dim X]$.
\end{example}

\begin{definition}
Let $n \in \mathbb{Z}$. A triangulated category $\mathcal{D}$ is a $n$-\textit{Calabi-Yau category} if it has a Serre functor $S_\mathcal{D}$ and $S_\mathcal{D} \cong [n]$. The integer $n$ is called the \textit{CY-dimension} of $\mathcal{D}$.
\end{definition}

\begin{definition}
A triangulated category $\mathcal{D}$ is a \textit{fractional Calabi--Yau category} if it has a Serre functor $S_\mathcal{D}$ and there are integers $p$ and $q\neq 0$ such that  $S^{q}_\mathcal{D} \cong [p]$. In this case we say that $\mathcal{D}$ has (\textit{CY}-)\textit{fractional dimension} $p/q$.
\end{definition}
The following result plays an important role in proving the existence of the Serre functor on  $\T_C.$
\begin{proposition}[{\cite[Proposition 3.8]{BK}}]\label{BKSERRE}
Let $\mathcal{D}$ be a triangulated category and $\mathcal{B} \subset \mathcal{D}$ an admissible full triangulated subcategory with $\mathcal{C}\coloneqq \mathcal{B}^\perp$ admissible. If $\mathcal{B}$ and $\mathcal{C}$ have Serre functors, then there exists a Serre functor on $\mathcal{D}$.
\end{proposition}

\subsection{Review of Bridgeland stability conditions}
\label{BRIPRE}

In this section we define Bridgeland stability conditions on a $\Co$-linear triangulated category $\mathcal{D}$ of finite type. We follow Bridgeland's papers \href{http://annals.math.princeton.edu/wp-content/uploads/annals-v166-n2-p01.pdf}{\cite{BS1}} and \cite{BS8}. We also recommend  the following lecture notes on Bridgeland stability theory \href{https://arxiv.org/pdf/1607.01262.pdf}{\cite{MS1}}, \href{http://www1.phys.vt.edu/mp10/lect-notes/bayer/dc-lecture-notes.pdf}{\cite{BS1.1}}, \href{https://arxiv.org/pdf/1111.1745.pdf}{\cite{HUYSC}} and 
\href{https://arxiv.org/pdf/0912.0043.pdf}{\cite[Appendix B.]{BMLP}}. 
  
\bd The \emph{Grothendieck group} $K(\D)$ of a triangulated category $\D$ is the abelian group generated by the isomorphism classes of objects in $\D$ subject to the relation $[A]=[C]+[B],$ where {$C\la A\la B\la C[1]$} is an exact triangle. 
We consider the Euler  bilinear form given by $\chi(E,F)=\sum_{i}(-1)^i\Hom_{\D}(E,F[i]).$ We define the \emph{numerical Grothendieck group} $\N(\D)$ as the quotient $K (\D)/K (\D)^{\perp},$ where $K(\D)^\perp $ denotes the right orthogonal with respect to the Euler  form.  Moreover, if $\N(\D)$ has finite rank then $\D$ is called \emph{numerically finite}.
\ed

\be \begin{enumerate}[leftmargin=0.5cm]
\item If $X$ is a smooth projective variety over $\Co$, then  $D^b(X)$ is numerically finite.
\item If $\D=\lin D_1, \dots, D_n\rin,$ then $K(\D)=\oplus_i K(D_i).$

\end{enumerate}
\ee

Throughout this entire section, we assume that $\D$ is numerically finite.
\bd  If $\A$ is an abelian category, we define the Grothendieck group $K(\A)$ as the abelian group generated by isomorphism classes of objects of $\A$ subject to the relation 
$[A]=[C]+[B], $ where $0\la C\la A\la B\la 0$ is a short exact sequence in  $\A.$ If $\A\subseteq \D$ is a heart of a bounded t-structure then $K(\D)\cong K(\A).$
\ed
We now fix a finite rank $\mathbb{Z}$-lattice $\Lambda$ and a surjective homomorphism $v\colon K(\D)\sur \Lambda.$
Since $\D$ is numerically finite, then $\N(\D)$ is a finite rank $\Z$-lattice. We often choose $\Lambda=\N(\D)$ and $v$ as the natural projection.

The definition of a Bridgeland stability condition has two main components: the heart of a bounded t-structure and a stability function.  

\bd Let $\A$ be an abelian category. We say that a group homomorphism $Z\colon K (\A)\la \Co$ is a \emph{weak stability function} on $\A$ if, for all $E \in \A$, we have $\Im(Z([E])) \geq 0$, with $\Im(Z([E])) = 0$ implying $\Re(Z([E])) \leq 0.$
If additionally, for $E \neq 0,$ we have that if $\Im (Z([E])) = 0$ then $\Re (Z([E])) < 0$, we say that $Z$ is a \emph{stability function} on $\A$. Note that in this case, the image of $Z$ is contained in the semi-closed upper half plane
$\overline{\mathbb{H}}=\{\alpha\in \Co\mid \Im(\alpha)\geq 0 \textnormal{ and if } \Im(\alpha)=0\textnormal{, then } \Re(\alpha)<0 \}.$ 
\ed

We consider a group homomorphism $Z\colon \Lambda \la \Co,$ such that  $ Z\circ v\colon K(\A)\la \Co$ is a stability function on $\A$. We define the slope $\mu_{Z}\colon K(\A) \la \R \cup \infty $ by \begin{equation} \nonumber
\mu_{Z}(E)=\begin{cases} \nonumber
-\frac{\Re(Z(E))}{\Im(Z(E))} & \textnormal{ if } \Im(Z(E))\neq 0 \\
+\infty & \textnormal{otherwise,}\end{cases}
\end{equation}
where $Z(E)\coloneqq Z(v([E])).$
We say that $0\neq E\in \A$ is \emph{Z-semistable (stable)} if for all proper subobjects  $F\subseteq E,$ we have that $\mu_{Z}(F)\leq\mu_{Z}(E) (\mu_{Z}(F)<\mu_{Z}(E)).$ We also define the phase of $0\neq E$ as  $\phi(E)=\arg(Z(E))\frac{1}{\pi}\in (0,1].$  Note that $E$ is $Z$-semistable if and only if for all proper subobjects $F\subseteq E,$ we have that $\phi(F)\leq \phi(E).$  We will constantly use the correspondence between slope and phase given for the complex numbers in the semi-closed upper half plane. 

\bd A stability function $Z\colon K(\A)\la \Co$ satisfies the \emph{Harder--Narasimhan property} (HN-property, for short) on $\A$ if for every $0\neq E\in \A,$ there is a filtration $0=E_0\subseteq  \cdots \subseteq E_{m-1}\subseteq E_m=E$ on $\A,$ such that $E_i/E_{i-1}$ is $Z$-semistable for $i=1,\ldots,m$ and $\phi(E_1/E_0)>\cdots>\phi(E_m/E_{m-1}).$
Moreover, as the HN-filtration is unique, we define the \emph{HN-factors} of $E$ as the quotients $E_i/E_{i-1}.$ 
\ed

\bd \label{DefSC} A \emph{(weak) pre-stability condition} on $\D$ is a pair $\sigma=(Z,\A),$ where $\A\subseteq \D$ is the heart of a bounded t-structure and $Z\colon \Lambda \la \Co$ is a group homomorphism such that $ Z\circ v\colon K(\A)(=K(\D))\la \Co$ is a (weak) stability function on $\A$ satisfying the HN-property. The homomorphism $Z$ is also called a central charge.
\ed

\brem \label{WEAKTILTING}\textup{\href{https://arxiv.org/pdf/1703.10839.pdf}{\cite[Proposition\ 2.9]{BMSL1}}} Let $\sigma=(Z,\A)$ be a weak pre-stability condition, and let $\alpha\in \R.$ We form the following subcategories of $\A$ 
$$\T_{\sigma}^{\alpha}\coloneqq \{E\in \A\mid \textnormal{The HN-factors $F$ of $E$ satisfy }\mu_{Z}(F)>\alpha\},$$
$$ \FF_{\sigma}^{\alpha}\coloneqq\{E\in \A\mid \textnormal{The HN-factors $F$ of $E$ satisfy }\mu_{Z}(F)\leq \alpha\}.$$
By the HN-property, we obtain that $(\T_{\sigma}^{\alpha},\FF_{\sigma}^{\alpha})$ is a torsion pair on $\A.$
Therefore, by \cite[Proposition\ 2.1]{HRS}, after tilting we obtain a heart of a bounded t-structure denoted by  $\A_{\sigma}^{\alpha}$ with a torsion pair given by $(\FF_{\sigma}^{\alpha}[1],\T^{\alpha}_{\sigma}).$
\erem

We will now define a slicing. Intuitively, a heart of a bounded t-structure $\A\subseteq \D$ breaks up every object in $\D$ in terms of its cohomology indexed by $\Z,$ a slicing further refines the heart of a bounded t-structure, which allows us to break up each object into pieces indexed by the  real numbers.

\bd \label{Slicing}{\cite{BS1}} A \emph{slicing} $\mathcal{P}$ on $\D$ is a collection of full subcategories $\mathcal{P}(\phi)$ for all $\phi \in \R$ satisfying:\begin{enumerate}[leftmargin=0.5cm]\item $\mathcal{P}(\phi)[1]=\mathcal{P}(\phi + 1)$, for all $\phi \in \R$.
\item If $\phi_1 > \phi_2$ and $E_i \in \mathcal{P}(\phi_i)$, $i=1,2$, then $\Hom_{\D} (E_1,E_2)=0$.
\item For every nonzero object $E \in \D$ there exists a finite sequence of maps
$$0=E_0\xrightarrow{f_{0}}E_1\xrightarrow{f_1}\cdots \la E_{m-1}\xrightarrow{f_{m-1}}E_m=E$$

%$$\xymatrix{0=E_0 \ar[r]^{f_{1}} & E_1 \ar[r] & \cdots \ar[r] & E_{m-1}\ar[r]^{f_{m}} & E_m=E}$$

and of real numbers $\phi_0 > \cdots > \phi_m$ such that the cone of $f_j$ is in $\mathcal{P}(\phi_j)$ for $j=0,\cdots,{m-1}$.
\end{enumerate}
\ed

For every interval $I\subseteq \R$ we define $\mathcal{P}(I)$ as the extension-closed subcategory generated by the subcategories $\mathcal{P}(\phi)$ with $\phi\in \R.$

\bp[{\cite[Proposition 5.3]{BS1}}] \label{SCSvH}  To give a pre-stability condition $\sigma$ on $\D$ is equivalent to giving a slicing $\mathcal{P}$ and a group homomorphism $Z\colon \Lambda\la \Co$ such that for every $0\neq E\in \mathcal{P}(\phi),$ we have that $Z(E) \in \R_{>0} \cdot e^{i \pi \phi}.$
\ep
\brem The key point of the proof of \cite[Proposition 5.3]{BS1} is to show that if $\mathcal{P}$ is a slicing, then $\A=\mathcal{P}(0,1]$ is the heart of a bounded t-structure satisfying the HN-property. If we have a pre-stability condition $\sigma=(Z,\A),$  we define $\mathcal{P}(\phi),$ for $\phi\in (0,1]$ as the set of all $Z$-semistable objects in $\A.$ 
\erem

\brem Let $\sigma=(Z,\mathcal{P})$ be a pre-stability condition. The objects of $\mathcal{P}(\phi)$ are called \emph{$\sigma$-semistable} and the simple ones \emph{$\sigma$-stable} objects.
By Definition \ref{Slicing}, for every $E\in \D,$ there is a filtration associated to $E$, that we also refer to as the Harder--Narasimhan filtration. The semistable objects in the filtration are called \emph{Harder--Narasimhan factors} (HN-factors, for short). Moreover, we write $\phi^{+}(E),\phi^{-}(E)$ for the largest and the smallest phase appearing in this filtration respectively. If $E$ is $\sigma$-semistable, $\phi^{+}(E)=\phi^{-}(E)=\phi(E).$
\erem

\brem \label{OCGKR}
\begin{enumerate}[leftmargin=0.5cm]
\item Let $\sigma=(Z,\A)$ be a pre-stability condition. By definition if $E\in \D$ is $\sigma$-semistable, then there exists $n\in \Z$ such that $E[n]\in \A.$
\item If $E,A\in \D$ and $\phi^{-}(E)>\phi^{+}(A),$ then $\Hom_{\D}(E,A)=0.$
\item Consider the last triangle $E_{m-1}\la E\la A_{m}\la E_{m-1}[1]$ of the HN-filtration of $E\in \D,$ where $A_{m}$ is the cone of $f_{m-1}.$ We have that $\Hom_{\D}^{\leq 0}(E_{m-1},A_{m})=0.$
\end{enumerate}  
\erem

%\bd The simple objects of $\mathcal{P}(\phi)$ are called \emph{$\sigma$-stable} objects. 
%\ed

Let $\GL$ be the universal covering of $\rm{GL}^{+}(2,\R),$ whose elements are given by pairs $(T,f)$ where $T\in \rm{GL}^{+}(2,\R)$ and $f\colon \R\la \R$ is a continuous increasing function that \linebreak satisfies  $f(x+1)=f(x)+1$ for all $x\in \R$ such that the induced maps of $T$ and $f$ on \linebreak $S^1=\R/2\Z=(\R^2-\{(0,0)\})/\R_{>0}$ coincide. In the next section we study in detail $\GL$ and its action on the set of pre-stability conditions. 

We define a right action of $\GL$ on the set of pre-stability conditions. If $\sigma=(Z,\A)$ is a pre-stability condition and $g=(T,f)\in \GL,$ then we define $\sigma'=\sigma g=(Z',\mathcal{P}')$ as $Z=T^{-1}\circ Z$ and $\mathcal{P}'(\phi)=\mathcal{P}(f(\phi)),$ where $\mathcal{P}$ and $\mathcal{P}'$ are the slicings of $Z$ and $Z'$ respectively.
Note that the $\GL$-action preserves the semistable objects, but relabels their phases. 

Let us consider the group $\Aut_{\Lambda}(\D)$ of autoequivalences $\Phi$ on $\D$ whose induced automorphism $\phi_*$ of $K(\D)$ is compatible with the map $v\colon K(\D)\la \Lambda.$

We define a left action of the group $\Aut_{\Lambda}(\D)$ on the set of pre-stability conditions. For $\Phi\in\Aut_{\Lambda}(\D)$ of $\D.$  We define $\Phi(\sigma)=(Z',\mathcal{P}')$  as  $Z'=Z\circ \phi^{-1}_*$ and $\mathcal{P}'(\phi)=\Phi(\mathcal{P}(\phi)).$ Note that if $E$ is a $\sigma$-semistable object, then $\Phi(E)$ is $\Phi(\sigma)$-semistable.

\bd A pre-stability condition $\sigma$ is \emph{locally finite} if  there is some $\epsilon>0$ such that each category $\mathcal{P}((\phi-\epsilon,\phi+\epsilon)),$ for $\phi\in \R,$ is of finite length. 
\ed

\bd A pre-stability condition is \emph{discrete} if the image of $Z$ is a discrete subgroup of $\Co.$
\ed

\bl[{\cite[Lemma\ 4.5]{BS8}}] \label{DISCRETELF} Suppose that $\sigma=(Z,\mathcal{P})$ is a discrete pre-stability condition and fix $0<\epsilon< \frac{1}{2}.$ Then for each $\phi\in \R$ the category $\mathcal{P}((\phi-\epsilon,\phi+\epsilon))$ is of finite length. In particular $\sigma$ is locally finite. 
\el

\brem[{\cite[Lemma 5.2]{BS1}}] The categories $\mathcal{P}(\phi)$ with $\phi\in \R$ are abelian. If $\sigma$ is locally finite, then $\mathcal{P}(\phi)$ has finite length. Therefore, a $\sigma$-semistable object $E\in \mathcal{P}(\phi)
$ admits a finite Jordan-Hölder filtration, i.e.\ a finite filtration $E_0\subset  \ldots \subset E_{n}=E$ with stable quotient $E_{i+1}/E_{i}\in \mathcal{P}(\phi),$ as the stable objects are the simple objects  in $\mathcal{P}(\phi).$
\erem

We will now recall the support property. It plays an important role in proving   good deformation properties and a well-behaved wall and chamber decomposition. 
We suggest \cite[App. A]{BMS3} to understand better the relation between the support property and effective deformations of Bridgeland stability conditions.

\bd A pre-stability condition $\sigma=(Z,\A)$ satisfies the \emph{support property} if there is a symmetric bilinear form $Q$ on $\Lambda\otimes \R\coloneqq \Lambda_{\R}$
which satisfies
\begin{enumerate}[leftmargin=0.5cm]
\item All $\sigma$-semistable objects $E\in \A,$ satisfy $Q(v(E),v(E))\geq 0.$
\item All nonzero vectors $v\in \Lambda_{\R}$ with $Z(v)=0$ satisfy $Q(v,v)<0.$
\end{enumerate}
\ed

\brem If $\rk(\Lambda)= 2$ and $Z\colon\Lambda\la \Co$ is injective, then every pre-stability condition \linebreak  $\sigma=(Z,\A)$  trivially satisfies the support property with respect to any positive semidefinite quadratic form. %Indeed, take $Q\colon \R^2 \la \R,$ as $\Ker(Z)=0,$ there is nothing to prove. If $E\in \D$ a $\sigma$-semistable, then $Q(v(E))\geq 0.$
\erem

\bd  A pre-stability condition $\sigma=(Z,\A)$ that satisfies the support property is called a \emph{Bridgeland stability condition}. The set of Bridgeland stability conditions with respect to $(\Lambda,v)$ is denoted by $\Stab_{\Lambda}(\D).$ If $\Lambda=\N(\D)$ and $v$ the natural projection, then the set of stability conditions is denoted by $\Stab(\D).$
\ed

%We now cite $\cite{BS1}$  to  on $\Stab_{\Lambda}(\D)$ and to prove that it is a complex manifold. 

\brem There is a generalized metric on the set of slicings $\Slice(\D),$ i.e.\ a metric that does not need to be finite: given two slicings $\mathcal{P}$ and $\mathcal{Q},$ we define
$$d(\mathcal{P},\mathcal{Q})=\sup_{0\neq E\in \D}\{\abs{\phi_{\mathcal{P}}^{+}(E)-\phi_{\mathcal{Q}}^{+}(E)}, \abs{\phi_{\mathcal{P}}^{-}(E)-\phi_{\mathcal{Q}}^{-}(E)}\}\in [0, \infty].$$ 
\erem
%Let $\sigma=(Z,\A)$ be a pre-stability condition. As $\Lambda$ has finite rank, we define the following norm 
%$$\vert\vert.\vert\vert_{\sigma}\colon \Hom_{\Z}(\Lambda,\Co)\la [0,\infty]$$ by sending a group homomorphism $W\colon\Lambda \la \Co$ to 
%$$\vert\vert W \vert\vert_{\sigma}=\sup\{\frac{\abs{W(E)}}{\abs{Z(E)}}\colon E \textrm{ }\sigma-\rm{semistable} \}.$$ 

%Note that the norm on $\Hom_{\Z}(\Lambda,\Co)$ depends on $\sigma,$ however after fixing a connected component $\Sigma$ and $\tau,\sigma\in \Sigma$ then $\vert\vert.\vert\vert_{\sigma}$ and $\vert\vert.\vert\vert_{\tau}$ are equivalent. See \cite[Lemma\ 6.2]{BS1}. 

%We define a topology on $\Stab_{\Lambda}(\D)$ as the coarsest topology such that both forgetful maps \vspace{-\abovedisplayskip} \begin{minipage}[b]{0.5\textwidth}
%\begin{eqnarray} \nonumber
%\Stab_{\Lambda}(\D)& \la & \Slice (\D)\\ \nonumber
%(Z,\mathcal{P})&\la & \mathcal{P}
%\end{eqnarray}
%\end{minipage}
%\begin{minipage}[b]{0.5\textwidth}
%\begin{eqnarray} \nonumber
%\mathcal{Z}\colon\Stab_{\Lambda}(\D)& \la  & \Hom_{\Z}(\Lambda,\Co)\\ \nonumber
%(Z,\mathcal{P})&\la & Z
%\end{eqnarray}
%\end{minipage}\\
%are continuous.

In order to study $\Stab_{\Lambda}(\D),$ we study the projection 
$\mathcal{Z}\colon \Stab_{\Lambda}(\D)\la \Hom_{\Z}(\Lambda,\Co).$ 
%We first show that $\mathcal{Z}$ is locally injective.
%The main theorem in \cite{BS1} shows that $\mathcal{Z}$ is a local homeomorphism. 

\bt [{\cite[Theorem.\ 7.1]{BS1}}]\label{BDT}  Let $\sigma=(Z,\mathcal{P})$ be a stability condition. If $1/8>\epsilon>0,$ then for any group homomorphism $W\colon K(\D)\la \Co $ with $\vert\vert W-Z \vert \vert _{\sigma}<\sin(\pi\epsilon),$ there exists a stability condition $\tau=(W,\mathcal{Q})$ on $\D$ with $d(\mathcal{P},\mathcal{Q})<\epsilon.$ 
\et
\bt[\cite{BS1}] The map  $\mathcal{Z}\colon \Stab_{\Lambda}(\D)\la \Hom(\Lambda, \Co)$  is a local homeomorphism. Particularly, this implies that $\Stab_{\Lambda}(\D)$ is a complex manifold of dimension $\rk(\Lambda).$
\et

For a complete proof see \cite[Sec 5.5]{BS1.1}.

%\bt \label{BAYER} \textup{\href{https://arxiv.org/pdf/1606.02169.pdf}{\cite[Theorem.\ 1.2]{BY1}} and \cite[Prop.\ A.5]{BMS3}} Let $Q$ be a quadratic form on $\Lambda\otimes \R.$ Assume that the stability condition $\sigma=(Z,\mathcal{P})$ satisfies the support property with respect to $Q.$ Then:
%\begin{enumerate}[leftmargin=0.5cm]
%\item There is an open neighbourhood $\sigma\in U_{\sigma}\subseteq \Stab_{\Lambda}(\D)$ such that $\mathcal{Z}\colon U_{\sigma}\la \Hom(\Lambda_{\R},\Co)$
%is a covering of the set of $Z'$ such that $Q$ is negative definite on $\Ker Z'.$
%\item All stability conditions in $U_{\sigma}$ satisfy the support property with respect to $Q.$
%\end{enumerate}
%\et

\bl[{\cite[Lemma\ 4.5]{BS8}, {\cite[Proposition B.4]{BMLP}}}] Let $\sigma$ be a Bridgeland stability condition and fix $0<\epsilon <\frac{1}{2}.$ Then, the quasi-abelian subcategory $\mathcal{P}((\phi-\epsilon, \phi+\epsilon))$ is of finite length, for each $\phi\in \R$. As a consequence, Bridgeland stability conditions are locally-finite.
\el

The main point of the proof of the last lemma is that for a Bridgeland stability condition $\sigma$ there exist discrete stability conditions arbitrarily close to $\sigma.$

%\bt \textup{\href{http://annals.math.princeton.edu/wp-content/uploads/annals-v166-n2-p01.pdf}{\cite[Theorem.\ 1.2]{BS1}}} Let $\D$ be a numerically finite triangulated category. For each connected component $\Sigma\subseteq \Stab_{\Lambda}(\D)$ there are linear subspaces $V(\Sigma)=\{W\in \Hom_{\Z}(\Lambda,\Co)\mid \vert\vert W\vert \vert _{\sigma} <\infty \}\subseteq $ and a local homeomorphism $\mathcal{Z}\colon \Sigma\la V(\Sigma)$ which maps a stability condition to its central charge Z. In particular $\Sigma$ is a finite-dimensional complex manifold.
%\et

\brem \label{SPUTACTION} Let $\sigma\in \Stab_{\Lambda}(\D)$ and $g\in \GL,$ then $\sigma g$ is also a Bridgeland stability condition. Indeed, if $\sigma$ satisfies the support property with respect to $Q,$ then $\sigma g$ also satisfies the support property with respect to $Q.$ If $\Phi\in \Aut_{\Lambda}(\D),$ then $\Phi(\sigma)$ satisfies the support property with respect to $Q\circ \phi^{-1}_*.$ By \cite[Lemma\ 8.2]{BS1}, the right action of $\GL$ and the left action of $\Aut_{\Lambda}(\D)$ on $\Stab_{\Lambda}(\D)$ commute.
\erem
\subsection{Bridgeland stability conditions on curves with \texorpdfstring{$g>0$}{TEXT}}

Let $C$ be a nonsingular projective curve of genus $g(C)>0.$
To describe $\Stab(C)\coloneqq\Stab(D^b(C)),$ we follow closely \cite[Section\ 9]{BS1} and \cite{M1}.
Note that there is an isomorphism $\mathcal{N}(D^b(C))\cong \Z^2$ given by $(\deg(E),\rk(E))$ for $E\in D^b(C).$ 
%Therefore, a stability function is given by a homomorphism $Z: \mathbb{Z}^2 \rightarrow \Co$ 

To describe $\Stab(C),$ one studies the last triangle of their HN-filtrations $E\la X\la G\la E[1],$ where $X$ is either $\Co(x)$ or $\Li.$ Note that by Remark \ref{OCGKR}, we have that $\Hom^{\leq 0 }(E,G)=0.$ The next lemma is a strong consequence of $D^b(C)$ being hereditary, i.e.\ it has homological dimension $1$ and the fact that $g(C)>0.$
\begin{lemma}[{\cite[Lemma\ 7.2]{GKR}}]\label{GKR}

\label{lem:GKRlemma} 
Given a distinguished triangle $E\la X\la G \la E[1]$ in $D^b(C)$,
with $X \in \Coh(C)$ and $\Hom^{\leq 0}_{D^b(C)}(E,G)=0$, then $E,G \in \Coh(C)$.
\end{lemma}

\begin{proposition}[{\cite{M1}}]\label{prop:all-L,k(x)-stable}
For any $\sigma \in \Stab(D^b(C))$, every line bundle $\mathcal{L}$ and skyscraper sheaf $\mathbb{C}(x)$ of a point $x \in C$ are $\sigma$-stable.
\end{proposition}

Moreover, there is a distinguished stability condition given by the standard slope stability: \linebreak 
$\sigma_\mu \coloneqq (Z_\mu,\coh(C))$
where $Z_\mu(r,d)= -d + i r$ with $(r,d)\in \mathbb{Z}^2$ with $r\coloneqq \rk(E)$ and $d \coloneqq \deg(E)$ for $[E]\in \mathcal{N}(D^b(C)).$ The next theorem states that this is in fact the only stability condition in $\Stab(D^b(C))$ up to the $\GL$-action.
\begin{theorem}[{\cite[Theorem.\ 9.1]{BS1},  \cite[Theorem 2.7]{M1}}]\label{thm:stabC}\label{GLC}
The action of $\GL$ on $\Stab(D^b(C))$ is free and transitive. In particular,
\begin{equation}\label{eq: curves iso}
\Stab(D^b(C)) \cong \GL.
\end{equation}
\end{theorem}

\brem \label{HEARTSCURVE} Let us consider $\sigma_{\mu}=(Z_{\mu},\Coh(C))$ and its corresponding slicing $\mathcal{P}_{\mu}.$ Note that $\Coh^{r}(C)\coloneqq \mathcal{P}_{\mu}(r,r+1]$ for $r\in \R$ is a heart of a bounded t-structure. All the hearts appearing in the stability conditions $\sigma=(Z,\A)\in \Stab(C)$ are of this form. Indeed, if $\sigma=\sigma_{\mu}g$ with $g=(T,f)\in \GL,$ we obtain that $\A=\mathcal{P}_{\mu}(f(0),f(1)=f(0)+1].$ If $f(0)=n+\theta,$ with $n\in \Z$ and $\theta\in  [0,1)$ then we also have $\A=\Coh^{\theta}(C)[n].$
\erem

For our purpose it is important to understand the isomorphism \eqref{eq: curves iso}. We define the following matrices in $\SL(2,\R)$
$$K_{\phi}=\begin{bmatrix}
\cos(\phi) & -\sin(\phi) \\
\sin(\phi) & \cos(\phi),
\end{bmatrix} \textnormal{, } A_{a}=\begin{bmatrix}
a&0 \\
0&\frac{1}{a}
\end{bmatrix} \textnormal{ and } N_{x}=\begin{bmatrix}
1& x\\0 & 1
\end{bmatrix},$$

for $\phi\in [0,2\pi),$ $x,a\in\R$ and $a>0.$

\bl[Iwasawa decomposition {\cite[Section\ 16.3]{HJB}}] For every $T\in \rm{GL}^{+}(2,\R),$ there are real numbers $\phi\in [0,2\pi),$ $k,a\in\R_{>0}$ and $x\in \R,$ such that $T=k K_{\phi}A_{a}N_{x}.$ Moreover, this representation is unique.
\el

Let us consider $g=(T,f)\in \GL$ where $T=k K_{\phi}A_{a}N_{x}$ and $f(0)=n+\theta,$ as above. We will now relate $\phi$ and $\theta.$

\brem \label{SCCHR}\label{KANS} 
\begin{enumerate}[leftmargin=0.5cm]
\item Let us consider $g=(T,f)\in \GL$ where $T=k K_{\phi}A_{a}N_{x}$ as above,  then there is $m\in \Z,$ such that $2m+\frac{\phi}{\pi}=f(0).$ Moreover, we have two cases: either
$$n=2m \textnormal{ with } \theta=\frac{\phi}{\pi} \textnormal{ or } n=2m+1 \textnormal{ with } \theta=\frac{\phi}{\pi}-1.$$ 
\item Let $\sigma=(Z,\A)\in \Stab(C).$ If $\sigma=\sigma_{\mu}g$ with  $g=(T,f)\in \GL,$ where $T=k K_{\phi}A_{a}N_{x}$ then $\A=\Coh^{\frac{\phi}{\pi}}(C)[2m].$
\item  Let $\sigma=(Z,\A)\in \Stab(C).$ If $\sigma'=\sigma g$ with $g=(T,f)\in \GL$ where $T=\pm k A_{a}N_{x},$ then there is $l\in \mathbb{Z}$ with $\A'=\A[l]$ where $\sigma'=(Z',\A').$ 
\item Let $\sigma=\sigma_{\mu}g=(Z,\A)$ with $g=(T,f)\in \GL$ and  $T^{-1}=\begin{bmatrix} -A & B \\ -D & C\end{bmatrix}$. We obtain that \linebreak $Z=T^{-1}Z_{\mu},$ i.e.\ $Z(d,r)=Ad+Br+i(Cr+Dd).$ Moreover, if $T=k K_{\phi}A_{a}N_{x},$ then \linebreak  $\arg(C+Di)=\phi.$ Indeed, as $C=\cos(\phi)\frac{a}{k}\textnormal{ and } D=\sin(\phi)\frac{a}{k},$ then $$\arg(C+Di)=\arg(\cos(\phi)+\sin(\phi)i)=\phi.$$
\item Let $\theta\in (0,1).$ First note, that by the HN-property, if $\T_{\theta}=\mathcal{P}_{\mu}(\theta,1]$ and $\FF_{\theta}=\mathcal{P}_{\mu}(0,\theta],$ then $(\T_{\theta},\FF_{\theta})$ is a torsion pair of $\Coh(C).$ By \cite[Proposition\ 2.1]{HRS} it induces a heart of a bounded t-structure which is precisely $\Coh^{\theta}(C).$ Moreover, if we define $\alpha=-\cot(\pi\theta),$ then by Remark \ref{WEAKTILTING}  we obtain that $\Coh^{\theta}(C)=\Coh(C)^{\alpha}_{{\sigma}_{\mu}}.$ Indeed, it is enough to notice that $\T_{\theta}=\T_{\sigma_{\mu}}^{\alpha}$ and that $\FF_{\theta}=\FF_{\sigma_{\mu}}^{\alpha},$ because
$\mu(E)>\alpha \textnormal{ if and only if } \phi(E)>\theta.$
\end{enumerate} 
\erem
%\bdem As $f$ and $T$ restricted to $S^1$ agree, we compute $$T(1,0)=k a(\cos(\phi),\sin(\phi))\in \R^2-{(0,0)}/\R_{>0}.$$
%It implies that $\restr{f}{[0,2)}(0)=\frac{\phi}{\pi}\in [0,2).$ Therefore, there is $m\in 
%\mathbb{Z}$ with $f(0)=\frac{\phi}{\pi}+2m.$
%\edem

%\bc  Let $\sigma=(Z,\A)\in \Stab(C).$ If $\sigma=\sigma_{\mu}g$ with  $g=(T,f)\in \GL,$ where $T=k K_{\phi}A_{a}N_{x}$ then $\A=\Coh^{\frac{\phi}{\pi}}(C)[2m].$
%\ec

%As a corollary, we obtain that, if we do not rotate $\sigma$ the heart will not change up to a shift. 

%\bdem As $\phi=0$ or $\phi=\pi,$ then $f(0)=2m$ or $f(0)=2m+1.$ Therefore, we  have that $\A'=\mathcal{P}'(0,1]=\mathcal{P}(f(0),f(1)]=\A[l]$ with $l=2m$ or $l=2m+1.$
%\edem

%\item Let $\sigma=\sigma_{\mu}g=(Z,\A)$ with $g=(T,f)\in \GL$ and \linebreak $T^{-1}=\begin{bmatrix} -A & B \\ -D & C\end{bmatrix}$. We obtain that $Z=T^{-1}Z_{\mu},$ i.e.\ $Z(d,r)=Ad+Br+i(Cr+Dd).$ Moreover, if $T=k K_{\phi}A_{a}N_{x},$ then  $\arg(C+Di)=\phi.$ Indeed, as $$C=\cos(\phi)\frac{a}{k}\textnormal{ and } D=\sin(\phi)\frac{a}{k},$$ then $\arg(C+Di)=\arg(\cos(\phi)+\sin(\phi)i)=\phi.$
%\item 
%\end{enumerate}
%\erem

Let $\sigma=\sigma_{\mu}g=(Z,\A)$ with $g\in \GL$ and $\mathcal{P}$ its slicing. The objects $\Co(x)$ and $\Li$ are $\sigma$-stable. By the definition of the $\GL$-action, this implies that for every $\sigma\in \Stab(C)$ we obtain $\phi_0\coloneqq \phi_{\sigma}(\Co(x))$ for $x\in C.$
 
By definition of the $\GL$-action, we obtain that $\mathcal{P}(\phi_0)=\mathcal{P}_{\mu}(f(\phi_0))=\mathcal{P}_{\mu}(1),$ i.e.\ \linebreak $f(\phi_0)=1.$ Analogously, if $\phi_1=\phi_{\sigma}(\Ot_C),$ then $f(\phi_1)=\frac{1}{2}.$ 

\brem \label{IPSC}\label{DATA}  \begin{enumerate}[leftmargin=0.5cm]
    \item We have that $-n<\phi_0=f^{-1}(1)\leq-n+1 \textnormal{ if and only if } f(0)=n+\theta,$ with $n\in \Z$ and $\theta\in [0,1).$ 
    \item There is a homeomorphism
\begin{eqnarray} \nonumber
\rho\colon \Stab(C) &\la & \{(m_0,m_1,\phi_0,\phi_1)\in \R^4 \mid \phi_1<\phi_0<\phi_1+1,\textnormal{ and } m_0,m_1>0\} \\\nonumber
 \sigma=(Z,\A) & \mapsto & (m_0,m_1,\phi_0,\phi_1),
\end{eqnarray}
where  $m_0=|Z([\Co(x)])|, m_1=|Z([\Ot_C])|.$
\end{enumerate} 
\erem
\section{Bridgeland stability conditions on \texorpdfstring{$\mathcal{T}_{C}$}{TEXT}}

Let $C$ denote a smooth projective curve of genus $g>0$. 
\begin{definition}
A \textit{holomorphic triple} $E=(E_1, E_2, \varphi)$ (also denoted by $E=E_1 \overset{\varphi}\rightarrow E_2$) on $C$ consists of two coherent sheaves $E_1, E_2 \in \coh(C)$ and a sheaf morphism $\varphi \colon E_1 \rightarrow E_2$ between them. We denote by $\TCoh(C)$ the category of holomorphic triples on $C.$
\end{definition}

\begin{remark}
The category $\TCoh(C)$ is abelian (see \href{https://pdfs.semanticscholar.org/2f49/9840b84f5aa95e0c89bbdcf2b082e5f79dbd.pdf}{\cite[Theorem\ 1]{MJ1}}).
\end{remark}

\subsection{The triangulated category \texorpdfstring{$\mathcal{T}_{C}$}{TEXT}}
To be studied in this subsection is the category \linebreak $\T_{C}\coloneqq D^b(\TCoh(C)).$  We construct different semiorthogonal decompositions for $\T_{C}.$ We prove the existence of the Serre functor.

We note that we have three different ways to embed $\Coh(C)$ into $\TCoh(C)$:
\begin{multicols}{3}\noindent\begin{eqnarray}\nonumber
i_{*}\colon \Coh(C) &\lai &\TCoh(C)\\
 X&\mapsto & X\rightarrow 0,\nonumber
\end{eqnarray}
\begin{eqnarray} \nonumber
j_{*}\colon \Coh(C) &\lai & \TCoh(C)\\
 X&\mapsto & 0\rightarrow X,\nonumber
\end{eqnarray}\begin{eqnarray} \nonumber
l_{*}\colon \Coh(C) &\lai & \TCoh(C)\\
X&\mapsto & X\xrightarrow{\id} X.\nonumber
\end{eqnarray}

\end{multicols}

%\begin{alignat}{8}
%i_* &\colon \Coh(C)  \{X  &\rightarrow&  0 &\colon & X \in \Coh(C) \} &\subset& \TCoh(C) ,\\ \nonumber 
%j_* (\Coh(C) ) &\coloneqq  \{0  &\rightarrow&  X &\colon & X \in \Coh(C) \} &\subset& \TCoh(C) , \\ \nonumber
%l_* (\Coh(C) ) &\coloneqq  \{X &\xrightarrow{\id}& X &\colon& X \in \Coh(C) \} &\subset & \TCoh(C).  
%\end{alignat}

As these functors are exact, we take their corresponding derived functors and by \href{https://link.springer.com/content/pdf/10.1007/978-3-319-28829-1_4.pdf}{\textup{\cite[Theorem\ 2.4]{BDG1}},} we obtain three different exact embeddings from $D^b(C)$ into $\TCoh(C)$. We adapt the same notation $i_*$, $j_*$, $l_*$. We denote the strictly full subcategories of $\mathcal{T}_{C}$ obtained as images of $D^b(C)$ in $\mathcal{T}_{C}$ under each embedding ($i_*$, $j_*$, $l_*$ respectively) by $\mathcal{C}_i$ for $i=1,2,3$.

Since the subcategories $\mathcal{C}_i $ of $\mathcal{T}_{C}$ are equivalent to $D^b(C),$ they are also saturated. 
\begin{remark}[{\cite{BK}}]\label{rem:Diadmissible}
As a consequence of Proposition \ref{prop:Satimpliesad}, we obtain that the strictly full subcategories $\mathcal{C}_i$ are admissible in $\mathcal{T}_{C}$ for all $i=1,2,3.$
\end{remark}

\begin{proposition}\label{prop:SOD}
The triangulated category $\mathcal{T}_{C}$ admits 3 semiorthogonal decompositions $$\mathcal{T}_{C}  =\langle \mathcal{C}_1, \mathcal{C}_2 \rangle\textnormal{, }\T_C=\langle \mathcal{C}_2, \mathcal{C}_3 \rangle \textnormal { and }\T_C =\langle \mathcal{C}_3, \mathcal{C}_1 \rangle.
$$
\begin{proof}
By Remark \ref{rem:Diadmissible} we already know that the $\mathcal{C}_i$ are admissible in $\mathcal{T}_{C}$ for all $i=1,2,3.$

We define the following functors
\begin{align*}
l_{*} &\colon \mathcal{C}_3 \rightarrow \mathcal{T}_{C} & X  &\mapsto X \overset{\id}{\rightarrow} X \\
l^* &\colon \mathcal{T}_{C} \rightarrow \mathcal{C}_3 & E_1 \overset{\varphi}{\rightarrow} E_2 &\mapsto E_2 \\
l^! &\colon \mathcal{T}_{C} \rightarrow \mathcal{C}_3 & E_1 \overset{\varphi}{\rightarrow} E_2 &\mapsto E_1.
\end{align*}
Left adjointness, $(l^*,l_*)$ and right adjointness, $(l_*,l^!)$ follow directly from the definitions. 

Additionally, we define the functors
\begin{align*}
i_{*} &\colon \mathcal{C}_1 \rightarrow \mathcal{T}_{C}  &X  &\mapsto X \rightarrow 0 \\
i^* &\colon \mathcal{T}_{C} \rightarrow \mathcal{C}_1  &E_1 \overset{\varphi}{\rightarrow} E_2 &\mapsto E_1.
\end{align*}
Left adjointness, $(i^*,i_*)$ follows directly from the definitions. 

Finally, we have to prove that
$\mathcal{C}_{3}^{\perp}=\{E \in \mathcal{T}_{C} \mid \Hom_{\mathcal{T}_{C}}(l_*\mathcal{C},E)=0 \} =\mathcal{C}_2.$ 
Indeed, by adjointness we have $\Hom_{\mathcal{T}_{C}}(l_*(X ), E)=\Hom_{\mathcal{C}}(X, l^!(E))$ for any $E=E_1 \overset{\varphi}{\rightarrow} E_2 \in \mathcal{T}_{C}$ and any $X \in \mathcal{C}$. This means that $\Hom_{\mathcal{C}}(X, E_1)=0$ for all $X \in \mathcal{C}$. This happens if and only if $E_1=0$.

With a similar argument, one can see that 
$\prescript{\perp}{}{\mathcal{C}}_3=\mathcal{C}_1.$ and $\prescript{\perp}{}{\mathcal{C}}_1=\mathcal{C}_2.$
\end{proof}
\end{proposition}

\begin{remark} \label{Triangles} Let $E=E_1\xrightarrow{\varphi} E_2\in \T_C.$ \label{rmk:SOD}\begin{enumerate}[leftmargin=0.5cm]
\item The semiorthogonal decomposition $\mathcal{T}_{C}=\langle \mathcal{C}_1, \mathcal{C}_2\rangle$ induces the  distinguished triangle $j_*(E_2)\la E\la i_*(E_1)\la j_*(E_2)[1].$

\item The semiorthogonal decomposition $\mathcal{T}_{C}=\langle \mathcal{C}_3, \mathcal{C}_1 \rangle$ induces the  distinguished triangle
$i_*i^!(E)\la E \la l_*(E_2)\la i_*i^!(E)[1].$

\item The semiorthogonal decomposition $\mathcal{T}_{C}=\langle \mathcal{C}_2, \mathcal{C}_3 \rangle$ induces the  distinguished triangle $l_*(E_1)\la E\la j_*(j^*(E))\la l_*(E_1)[1].$
\end{enumerate}

By admissibility of the full subcategories $\mathcal{C}_i$ for $i=1,2,3$, and the unicity of the triangles associated to the semiorthogonal decompositions, we have that $l^*=j^!$, where $j^!$ satisfies right adjunction $(j_*,j^!)$, similarly \ $l^!=i^*$.  

In order to describe the precise distinguished triangles associated to the semiorthogonal decompositions of $\mathcal{T}_{C}$ we need to understand the functors $i^!$ and $j^*$ satisfying the adjunction relations $(i_*,i^!)$ and $(j^*, j_*)$ respectively. We will now describe the functors $i^!$ and $j^*$ at the level of objects and some features at the level of morphisms.

%We will define these missing functors after the construction of the Serre functor on $\mathcal{T}_{C}$.
\end{remark}

\begin{lemma}\label{SQ31} 
Let $E=(E_1,E_2,\varphi_{E})$ and  $F=(F_1,F_2,\varphi_{F})\in \T_{C}.$ The functor $i^!$ is given by \begin{eqnarray}
i^!\colon \T_{C}& \la & D^b(C) \\ \nonumber
E &\mapsto & C(\varphi_{E})[-1], 
\end{eqnarray} where $C(-)$ is the cone, 
at the level of objects. If $\psi\colon E\la F$ is a morphism of triples, then the following diagram commutes
\begin{equation} \label{squarei}
 \xymatrix{ C(\varphi_E)[-1]  \ar[r]^{i^!(\psi)} \ar[d]^{{\pi}_E} & C(\varphi_F)[-1] \ar[d]^{{\pi}_F}  \\
 E_1 \ar[r]^{i^*(\psi)} & F_1.}
\end{equation}
\begin{proof}
From the triangle in part (3) from Remark \ref{Triangles} and the fact that $i^*$ is an exact functor, we obtain the triangle 
$i^!(E)\xrightarrow {\pi_E} E_1\xrightarrow{\varphi} E_2\la i^!(E),$
therefore $i^!(E)= C(\varphi_{E})[-1].$
We use the naturality of the adjunction and we obtain $\psi\circ \pi_E=i_*i^!(\psi)\circ \pi_F$ and  by taking $i^*$ in both sides we obtain the square \eqref{squarei}.
%Indeed, we have the following diagram, where the two inner squares commute: 
%\begin{equation} 
 %\xymatrix{\Hom_{D^b(C)}(i^!(E), i^!(E)) \ar[r] \ar[d]_{i^!(\psi)} & \Hom_{\T_{C}}(i_*(i^!(E)),E) \ar[d]_{\psi} \\ \Hom_{D^b(C)}(i^!(E),i^!(F)) \ar[r]^{i^{!}(\psi)} & \Hom_{\T_{C}}(i_*(i^!(E)),F)\\
 %\Hom_{D^b(C)}(i^!(F),i^!(F)) \ar[u]^{i^!(\psi)} \ar[r]& \Hom_{\T_{C}}(i_*(i^!(F)),F).\ar[u]^{i^!(\psi)},} 
%\end{equation}
%note that $i^!(\psi) \circ id_{i^!(E)}=\id_{i^!(F)}\circ i^!(\psi).$
\end{proof}
\end{lemma}

\begin{lemma} 
Let $E=(E_1,E_2,\varphi_{E})$ and $F=(F_1,F_2,\varphi_{F})\in \T_{C}.$ The functor $j^*$ is given by \begin{eqnarray}
j^*\colon \T_{C}& \la & D^b(C) \\ \nonumber
E &\mapsto & C(\varphi_{E}) 
\end{eqnarray}
at the level of objects. If $\psi\colon E\la F$ is a morphism of triples, then the following diagram commutes
\begin{equation} 
 \xymatrix{E_2 \ar[r]^{j^!(\psi)} \ar[d]_{\tau_E} & F_2 \ar[d]^{\tau_F} \\ C(\varphi_E) \ar[r]^{j^*(\psi)} & C(\varphi_F). }
\end{equation}

\begin{proof}
The proof goes along the lines of Lemma \ref{SQ31}.
\end{proof}
\end{lemma}

\begin{remark}
In many cases, we express  a morphism $\Phi\in \Hom_{\T_{\mathcal{C}}}(E,F)$ as two horizontal arrows, but note that they just represent $i^*(\Phi)$ and $j^!(\Phi)$ and they do not characterize the morphism $\Phi.$ 
\end{remark}

We denote by $\mathcal{N}(C)=\mathcal{N}(D^b(C))$ the numerical Grothendieck group of $\coh(C).$ 

\brem Note that $K(\T_C)=K(D^b(C))\oplus K(D^b(C)).$ 

Moreover, the Euler form $\chi(E,E') \coloneqq \sum \limits_{i} (-1)^i \dim_\mathbb{C} \Hom_{\mathcal{T}_{C}}(E,E'[i])$ on $\TCoh(C)$ vanishes if and only if $r_1=r_2=0$ and $d_1=d_2=0$, where $d_i\coloneqq \deg(E_i)$, $r_i\coloneqq  \rk(E_i)$ for $i=1,2$. Hence, $\mathcal{N}(\mathcal{T}_{C})$ is isomorphic to $\mathbb{Z}^4$ by identifying the class $[E] \in \mathcal{N}(\mathcal{T}_{C})$ of a holomorphic triple $E=E_1 \overset{\varphi}\rightarrow E_2$ with the point $(r_1,d_1,r_2,d_2) \in \mathbb{Z}^4.$
\erem 

\subsubsection*{Serre functor}

By Proposition \ref{BKSERRE} and Proposition \ref{prop:SOD}, we have the existence of the Serre functor $S_{\mathcal{T}_{C}}$ on $\mathcal{T}_{C}$. The adjunction properties of the Serre functor provide us with the following lemma.

\begin{lemma}[{\cite[Lemma 2.42]{ARH1}}]\label{lem:SerreTappliedToCohi}
Let $X \in  D^b(C)$. The following equalities hold:
\begin{enumerate}
\item $S_{\mathcal{T}_{C}}(i_*(X))= j_*(S_{C}(X))[1],$ 
\item $S_{\mathcal{T}_{C}}(j_*(X))= l_*(S_{C}(X)),$ 
\item $S_{\mathcal{T}_{C}}(l_*(X))= i_*(S_{C}(X)),$ where $S_C$ is as in Example \ref{Serre}. 
\end{enumerate}
\end{lemma}

\begin{remark}Let $E=E_1\xrightarrow{\varphi} E_2$  be an object in $\T_{\mathcal{C}}.$  As we have $F^!S_{\T_{\mathcal{C}}}=S_{C}F^*$ for $F=i,j$ or $l,$ we directly $i^!(S_{\T_{\mathcal{C}}}(E))  =  S_{C}(E_1) \textnormal{ and }j^!(S_{\T_{\mathcal{C}}}(E))  =  S_{C}(C(\varphi)).$

\end{remark}

\brem [{\cite[Proposition 1.3.3]{BBD}} and {\cite[Corollary 1.1.10]{BBD}}] \label{Uniquemorphism}
Let us consider the triangle $j_*j^!(E)\la E \la i_*i^*(E)\xrightarrow{t_{E}} j_*j^!(E)[1].$ Note that the morphism $t_{E} \colon i_*i^*(E)\rightarrow j_*j^!(E)[1]$ characterizes the triangle, i.e.\ it is the unique morphism representing the isomorphism class of the triangle 
induced by the semiorthogonal decomposition $\lin \mathcal{C}_1, \mathcal{C}_2\rin$ for $E.$
Moreover, note that $C(t_{E})[-1]$ is isomorphic to $E$ up to a nonunique isomorphism in $\T_{C}.$ 
\erem

\begin{remark}\label{GBSF} Let us consider the triangle $j_*j^!(E)\la E \la i_*i^*(E)\xrightarrow{t_{E}} j_*j^!(E)[1].$ After applying the Serre functor we obtain a triangle $$S_{\T_{C}}(j_*j^!(E))\la S_{\T_{C}}(E) \la S_{\T_{C}}(i_*i^*(E))\xrightarrow{S_{\T_{C}}(t_{E})} S_{\T_{C}}(j_*j^!(E)[1].$$ As $S_{\T_{C}}({^\perp}\D)=\D^{\perp}$ for an admissible triangulated subcategory $\D\subseteq \T_C,$ then $S_{\T_{C}}(j_*j^!(E))\in \mathcal{C}_3 \textnormal{ and } S_{\T_{C}}(i_*i^*(E))\in \mathcal{C}_2.$ By the uniqueness of the triangle induced by the semiorthogonal decomposition, we obtain precisely the corresponding triangle of $S_{\T_{C}}(E)$ induced by the semiorthogonal decomposition $\T_{C}=\lin \mathcal{C}_2,\mathcal{C}_3\rin $ up to isomorphism. Moreover $S_{\T_{C}}(t_{E})$ is the unique morphism that characterizes the triangle as mentioned in Remark \ref{Uniquemorphism}. See \href{https://arxiv.org/pdf/1509.07657.pdf}{\cite[Lemma\ 2.3]{AK1}} and \cite[Proposition\ 1.3.3]{BBD}.
\end{remark} 
The following remark connects the triple $E$ with the roof $[\varphi_E]\coloneqq (\id,\varphi_E) $ in $\Mor(D^b(C),D^b(C))$ and it  plays a role in the explicit construction of the Serre functor. 

\brem[{\cite[Lemma 2.49]{ARH1}}]\label{TL12}  $t_{E}=t_{l_*(E_2)}\circ i_*([\varphi]).$
\erem 

We will now describe the Serre functor at the level of objects. 

\bp [{\cite[Lemma 2.51]{ARH1}}] \label{SERRE1}  Let $E=E_1\xrightarrow{\varphi}E_2$ be an object of $\T_{C}.$ If $S_{C}(i_{E})$ is a morphism of complexes, where $E_2\xrightarrow{i_E} C(\varphi)$ is the morphism of complexes given by the injection, then $S_{\T_{C}}(E)$ is isomorphic to the triple $S_{C}(E_2)\xrightarrow{S_{C}(i_{E})}S_{C}(C(\varphi))$ in $\T_{C}.$ 

On the other hand, if $S_{C}^{-1}(p_E)$ is a morphism of complexes, where $C(\varphi)[-1]\xrightarrow{p_E}E_1$ is the projection, then $S^{-1}_{\T_{C}}(E)$ is isomorphic to $S^{-1}_{C}(C(\varphi)(-1))\xrightarrow{S^{-1}_{C}(p_E)}S^{-1}_{C}(E_1).$
\ep

\begin{remark} First of all, note that by Serre duality and the semiorthogonal decomposition described above, it is clear that the homological dimension of $\TCoh(C)$ is $\leq 2$ (see \cite[Proposition 2.7.24]{EM1}). Moreover, given $E_1 \in \Coh(C)$, if we take $i_*(E_1)$ and  $j_*(E_1)$, by Serre duality and Lemma \ref{lem:SerreTappliedToCohi} we get 
\begin{align*}
\Hom_{\mathcal{T}_{C}} (i_*(E_1), j_*(E_1)[2]) &=\Hom_{\mathcal{T}_{C}} (j_*(E_1)[2],S_{\mathcal{T}_{C}}(i_*(E_1)))^* \\
&=\Hom_{D^b(C)} (E_1, E_1 \otimes \omega_C)^*[2]\neq 0.  
\end{align*}
 
This implies that the category $\TCoh(C)$ has
homological dimension 2.
\end{remark}

If $g(C)=1$, we have the following conjecture.
\begin{conjecture}
Let $C$ be an elliptic curve. Then $S_{\mathcal{T}_{C}}^3=[4]$, which implies that $\mathcal{T}_{C}$ is a fractional Calabi--Yau category of fractional dimension $4/3$.
\end{conjecture}

\subsection{Constructing pre-stability conditions}
\subsubsection{CP-gluing} 

In \cite{CP} (and Collins in \cite{Co09}) Collins and Polishchuk introduced a way to construct hearts from semiorthogonal decompositions of a triangulated category $\mathcal{D}$ which later allow us to define stability conditions on $\mathcal{D}$ in a natural way. 

Let $\mathcal{D}$ be a triangulated category equipped with a semiorthogonal decomposition $\mathcal{D}=\langle \mathcal{D}_1, \mathcal{D}_2\rangle$. Let $\rho_2 \colon \mathcal{D} \rightarrow \mathcal{D}_2 $ be the right adjoint functor to the full embedding $i_2 \colon \mathcal{D}_2 \rightarrow \mathcal{D}$ and let $\lambda_1 \colon \mathcal{D} \rightarrow \mathcal{D}_1 $ be the left adjoint functor to the full embedding $i_1 \colon \mathcal{D}_1 \rightarrow \mathcal{D}$. 
\begin{proposition}[{\cite[Lemma 2.1]{CP}}]\label{CP}
With the above notations, assume that we have t-structures $(\mathcal{D}_{i}^{\leq 0},\mathcal{D}_{i}^{\geq 0})$ with hearts $\mathcal{A}_i$ on $\mathcal{D}_i$, for $i=1,2$, such that
\begin{equation}\label{eq:glHom0}
\Hom_{\mathcal{D}}^{\leq 0}(i_1 \mathcal{A}_1,i_2 \mathcal{A}_2)=0.
\end{equation}
Then there is a t-structure on $\mathcal{D}$ with the heart
\begin{equation}\label{eq:gluedA}
\glu(i_1 \mathcal{A}_1,i_2 \mathcal{A}_2)=\{E \in \mathcal{D} \mid \rho_2 E \in \mathcal{A}_2, \lambda_1 E \in \mathcal{A}_1 \}.
\end{equation}
Moreover, $i_k\mathcal{A}_k \subset \mathcal{A} \coloneqq \glu(i_1 \mathcal{A}_1,i_2 \mathcal{A}_2)$ for $k=1,2$.
\end{proposition}

\begin{definition}
We will refer to hearts of the form \eqref{eq:gluedA} as hearts obtained by \textit{CP-gluing}.
\end{definition}

\begin{notation}
In the case of the semiorthogonal decompositions $\mathcal{T}_{C}=\langle \mathcal{C}_i, \mathcal{C}_j \rangle$, with \linebreak $ij \in \{12, 23, 31\}$, we will denote the heart obtained by CP-gluing of hearts $\mathcal{A}_i$ on $ \mathcal{C}_i$ and $\mathcal{A}_j$ on $\mathcal{C}_j$ by $ \glu_{ij}(\mathcal{A}_i,\mathcal{A}_j)$. Recall that for $r\in \R,$ we define $\Coh^r(C)$ as in Remark \ref{HEARTSCURVE}.
\end{notation}

\begin{proposition}\label{prop: gluingcond}
We distinguish 3 cases according to the semiorthogonal decomposition of $\mathcal{T}_{C}$.
\begin{enumerate}[leftmargin=0.5cm]
\item If $\mathcal{T}_{C}=\langle \mathcal{C}_1, \mathcal{C}_2 \rangle$, we have
$\Hom_{\mathcal{T}_{C}}^{\leq 0}(\Coh_1^{r_1}(C) ,\Coh_2^{r_2}(C) )=0$ if and only if $r_1 \geq r_2$.
\item If $\mathcal{T}_{C}= \langle \mathcal{C}_3, \mathcal{C}_1 \rangle$, we have
$\Hom_{\mathcal{T}_{C}}^{\leq 0}(\Coh_3^{r_3}(C),\Coh_1^{r_1}(C))=0$ if and only if $r_3 \geq r_1 +1$.
\item If $\mathcal{T}_{C}= \langle \mathcal{C}_2, \mathcal{C}_3 \rangle$, we have
$\Hom_{\mathcal{T}_{C}}^{\leq 0}(\Coh_2^{r_2}(C) ,\Coh_3^{r_3}(C))=0$ if and only if $r_2 \geq r_3 +1$.
\end{enumerate}
\begin{proof}
As in Remark \ref{SCCHR}, we write each $r_i = n_i + \theta_i$ for unique $n_i \in \mathbb{Z}$ and $\theta_i\in [0,1)$.
Firstly, we want to determine under which conditions
$\Hom_{\mathcal{T}_{C}}(\Coh_1^{\theta_1}(C) [n_1],\Coh_2^{\theta_2}(C)[n_2+i])=0$
for every $i \leq 0$. Assume $n_2=0$, up to shifting by $-n_2$. For a fixed $i$, take $0 \rightarrow E_2 \in \Coh_2^{\theta_2}(C)$ and $E_1 \rightarrow 0 \in \Coh_1^{\theta_1}(C)$. By Serre duality we have
\begin{equation}\label{eq:serredualityingluingprop}
\Hom_{\mathcal{T}_{C}}(E_1 [n_1] \rightarrow 0,0 \rightarrow E_2 [i])=\Hom_{\mathcal{T}_{C}}(0 \rightarrow E_2,S_{\mathcal{T}_{C}}(E_1 \rightarrow 0) [n_1 - i])^*.
\end{equation}
By Lemma \ref{lem:SerreTappliedToCohi}, $S_{\mathcal{T}_{C}}(E_1 \rightarrow 0) [n_1] \in j_*(S_{C}(\Coh^{\theta_1}(C)[n_1]))[1]$, so \eqref{eq:serredualityingluingprop} vanishes for all $i \leq 0$ if and only if $n_1 \geq 0$ and if $n_1 =0$, we see that we need that $ \theta_1 \geq \theta_2$. Indeed, if $n_1=0$, recall that each heart $\Coh^{\theta_i}(C)$ was defined by tilting $\Coh^{\theta_i}(C)=\langle F_{\theta_i}[1], T_{\theta_i} \rangle$, for $i=1,2$. By the previous argument with the Serre functor, the only restriction appears when $E_1 \rightarrow 0 \in T_{\theta_1}$ and $0 \rightarrow E_2 [1]\in F_{\theta_2}[1]$. Here, if $ \theta_1 \geq \theta_2$, we have $T_{\theta_1} \subset T_{\theta_2}$. Use Serre duality for $D^b(C)$, such that
$$\Hom_{D^b(C)}(E_2[1],E_1 \otimes \omega_{C}[2-i])=\Hom_{D^b(C)}(E_1  [1-i], E_2 [1])^*.$$
This vanishes for all $i \leq 0$, since $(T_{\theta_2}, F_{\theta_2})$ is a torsion pair.

Case 2 (resp.\ 3) follows applying the Serre functor $S_{\mathcal{T}_{C}}$ (resp.\ $S^{-1}_{\mathcal{T}_{C}}$) to case 1.
\end{proof}
\end{proposition}

\begin{remark}
By Theorem \ref{thm:stabC}, the CP-gluing conditions in Proposition \ref{prop: gluingcond} can be re-stated in terms of the corresponding elements $(T_i, f_i) \in \GL$ for $i = 1,2,3$ by considering $r_i = f_i(0)$.
\end{remark}

Trivial examples of hearts in $\T_C$ defined by CP-gluing can be given by considering $\coh(C)$ embedded in each triangulated subcategory $\mathcal{C}_i$ for $i=1,2,3$ as follows.

\begin{definition}[{\cite{CP}}]
Let $\sigma_i=(Z_i,\mathcal{A}_i)$ be stability conditions on $\mathcal{D}_i$ for $i=1,2$, such that the hearts $\mathcal{A}_i$ satisfy \eqref{eq:glHom0}. Then we say that a pair $\sigma=(Z,\mathcal{A})$ on $\mathcal{D}=\lin \mathcal{D}_1,\mathcal{D}_2 \rin$ is a \textit{CP-glued pair} (from $\sigma_1$ and $\sigma_2$) if the heart $\mathcal{A}$ is given by \eqref{eq:gluedA} and $Z \colon K(\mathcal{A}) \rightarrow \mathbb{C}$ is given by
\begin{equation}\label{eq:gluedZ}
Z = Z_1 \circ \lambda_1 + Z_2 \circ \rho_2.
\end{equation}
\end{definition}

\begin{remark}
Note that this CP-glued pair is uniquely determined by $\sigma_1$ and $\sigma_2$. It is a pre-stability condition if and only if the Harder--Narasimhan property holds for the stability function $Z$ on the glued heart $\mathcal{A}$. We will check this property separately later.
\end{remark}

\begin{notation}
In the case of the semiorthogonal decompositions $\mathcal{T}_{C}=\langle \mathcal{C}_i, \mathcal{C}_j \rangle$ for $ij \in \{12, 23, 31\}$ we will denote by $\glu_{ij}(\sigma_i,\sigma_j)$ the CP-glued pair obtained by CP-gluing of stability conditions $\sigma_i$ on $ \mathcal{C}_i$ and $\sigma_j$ on $\mathcal{C}_j$.
\end{notation}
\begin{definition}[Standard hearts]\label{hearts123} 
Consider $\T_C=\lin \mathcal{C}_1,\mathcal{C}_2 \rin.$ By definition we have that $$\TCoh(C)=\gl_{12}(\Coh_1(C),\Coh_2(C)).$$ 

Consider $\T_C=\lin \mathcal{C}_2,\mathcal{C}_3 \rin,$ then we define $$\HH_{23}\coloneqq \gl_{23}(\Coh_2(C)[1],\Coh_3(C)).$$ 

Consider $\T_C=\lin \mathcal{C}_3,\mathcal{C}_1 \rin,$ then we define $$\HH_{31}\coloneqq \gl_{31}(\Coh_3(C)[1],\Coh_1(C)).$$
\end{definition}
\begin{lemma}[{\cite[Proposition 2.2]{CP}}]\label{lem:CPgluProps}
\begin{enumerate}[leftmargin=0.5cm]
\item A pre-stability condition $\sigma=(Z,\mathcal{A})$ on $\mathcal{D}$ is CP-glued from $\sigma_1=(Z_1,\mathcal{A}_1)$ on $\mathcal{D}_1$ and $\sigma_2=(Z_2,\mathcal{A}_2)$ on $\mathcal{D}_2$ if and only if $Z_i= Z|_{\mathcal{D}_i}$ for $i=1,2$, \linebreak $\Hom_{\mathcal{D}}^{\leq 0}(\mathcal{A}_1,\mathcal{A}_2)=0$ and $\mathcal{A}_i \subset \mathcal{A}$ for $i=1,2$.
\item Let $\sigma=(Z,\mathcal{A})$ be a pre-stability condition on $\mathcal{D}$. Assume that the heart $\mathcal{A}$ is glued from hearts $\mathcal{A}_1 \subset \mathcal{D}_1$ and $\mathcal{A}_2 \subset \mathcal{D}_2$, with $\Hom_{\mathcal{D}}^{\leq 0}(\mathcal{A}_1,\mathcal{A}_2)=0$, such that \eqref{eq:gluedA} holds. Then, there exists a pre-stability condition $\sigma_i=(Z_i=Z|_{\mathcal{D}_i},\mathcal{A}_i)$ on $\mathcal{D}_i$, for $i=1,2$, such that $\sigma$ is glued from $\sigma_1$ and $\sigma_2$.
\item If $\sigma=(Z,\mathcal{P})$ is CP-glued from the pre-stability conditions $\sigma_1=(Z_1,\mathcal{P}_1)$ and $\sigma_2=(Z_2,\mathcal{P}_2)$, then $\mathcal{P}_1(\phi) \subset \mathcal{P}(\phi)$ and  $\mathcal{P}_2(\phi) \subset \mathcal{P}(\phi)$ for every $\phi \in \mathbb{R}$.
\end{enumerate}
\end{lemma}

We will now explain the behaviour of the CP-glued pre-stability conditions under the $\GL$-action. 
\begin{lemma}\label{GBCP} 
Let $\sigma_i=(Z_i,\mathcal{A}_i)$ be pre-stability conditions on $\mathcal{D}_i$ for $i=1,2$ and let  $\sigma$ be a CP-glued pre-stability condition from $\sigma_1$ and $\sigma_2$ on a triangulated category $\mathcal{D}$ with respect to the semiorthogonal decomposition $\mathcal{D}=\lin \mathcal{D}_1, \mathcal{D}_2 \rin.$ Let $g=(T,f)\in \GL,$ then $\sigma g=(W,\B)$ satisfies the following conditions:
If $\sigma_i g=(W_i,\B_i),$ for $i=1,2$ then 
\begin{enumerate}[leftmargin=0.5cm]
\item $i_1\B_1\subseteq \B$ and $i_2\B_2 \subseteq \B$
\item $\restr{W}{\mathcal{D}_i}=W_i,$ for i=1,2. 
\end{enumerate}
Moreover, if $\Hom^{\leq 0}_\mathcal{D}(i_1\B_1,i_2\B_2)$ then $\sigma g=\gl_{12}(\sigma_1 g, \sigma_2 g)$.

\begin{proof} Let us consider the slicing $\mathcal{P}_i$ of $\sigma_i,$ for $i=1,2,$ and $\mathcal{P}$ the slicing of $\sigma.$ By definition of the $\GL$-action we have $\B_i=\mathcal{P}_i(f(0),f(1)]$ and $\B=\mathcal{P}(f(0),f(1)].$ By the third part of Proposition \ref{lem:CPgluProps} we obtain directly that $i_1\B_1\subseteq \B$ and $i_2\B_2\subseteq \B.$ 

Next, recall that by definition we have that $W=T^{-1}\circ Z.$ Let $E_i\in \mathcal{D}_i$ for $i=1,2$. We note that $\lambda_1(E_1)=E_1$ and $\rho_2(E_1)=0$ (resp.\ $\lambda_1(E_2)=0$ and $\rho_2(E_2)=E_2$) which implies $W([E_i])=T^{-1}\circ Z_i([E_i])=W_i$. Finally, if we also assume that $\Hom^{\leq 0}_{\mathcal{D}}(i_1\B_1,i_2\B_2)=0,$ by the first part of Proposition \ref{lem:CPgluProps} we obtain that $\sigma g=\gl_{12}(\sigma_1 g,\sigma_2 g).$
\end{proof}
\end{lemma}

It is important to remark that the gluing condition may not be preserved after applying the $\GL$-action, as we can see in the following example. 

\begin{example}\label{ex:gl-action}
Consider the CP-glued pair $\sigma=\gl_{12}(\sigma_1,\sigma_\mu)$ with $\sigma_1=\sigma_\mu g$ and  $g=(N_{1},f)$ in $\GL.$
Note that $f(0)=0$ and that $\mathcal{P}_1(t)=\mathcal{P}_{\mu}(f(t)).$

We have $1>t > f(t)$ for all $t \in (0,1)$. Indeed, let $E \in \mathcal{P}_{\mu}(t)$ be an object of phase $t \in (0,1)$. In particular, the rank $r$ of $E$ is strictly positive. Now, since $E \in \mathcal{P}_{\mu}(t) = \mathcal{P}_1(f^{-1}(t))$, together with the fact that $\Im Z_1 (E) =\Im Z_\mu(E)=r$ and $\Re Z_1(E)= -d+r > \Re Z_{\mu}(E)$, we have \linebreak $0<f(t)<t<1$. The inequalities follow because $f$ is strictly increasing. 
Let us consider $g'=(K_{\phi \pi},f_{\phi\pi})$  for $\phi\in (0,1).$ We will now study $\sigma_1'=\sigma_1 g'$
and $\sigma_2'=\sigma_\mu g',$ which are also given by $(T'_i,f'_i)\in \GL$ for $i=1,2$ respectively. Note that $f'_1(0)=f\circ g'(0)=f(\phi)$ and \linebreak $f'_2(0)=g'(0)=\phi.$ Therefore, we obtain  $f'_1(0)<f'_2(0)$ and from Proposition \ref{prop: gluingcond} it follows that $\sigma'_1$ and $\sigma'_2$ do not satisfy gluing conditions with respect to the semiorthogonal decomposition $\T_C=\lin \C_1,\C_2 \rin.$ See Subsection \ref{CSC12} for the general picture. 

\end{example}

Now that we have hearts in $\mathcal{T}_{C}$ with the corresponding stability functions, we have to check that they satisfy the Harder--Narasimhan property and the support property. We begin with the Harder--Narasimhan property, arguing along the lines of \cite{CP}.

\begin{theorem}[{\cite[Theorem 3.6]{CP}}]\label{thm:HNbig}
Let $(\mathcal{D}_1,\mathcal{D}_2)$ be a semiorthogonal decomposition of a triangulated category $\mathcal{D}$. Suppose $\sigma_i =(Z_i,\mathcal{P}_i)$ is a stability condition on $\mathcal{D}_i$ for $i=1,2$ and let $a \in (0,1)$ be a real number. Assume the following conditions hold:
\begin{enumerate}
\item $\Hom_{\mathcal{D}}^{\leq 0}(i_1 \mathcal{P}_1(0,1],i_2 \mathcal{P}_2(0,1])=0$ 
\item $\Hom_{\mathcal{D}}^{\leq 0}(i_1 \mathcal{P}_1(a,a+1],i_2 \mathcal{P}_2(a,a+1])=0$. 
\end{enumerate}
Then, there exists a locally finite pre-stability condition $\sigma$ glued from $\sigma_1$ and $\sigma_2$.
\end{theorem}

\begin{definition}
For a real number $a \in (0,1)$, we define the subset $S(a)$ as the subset of pairs of stability conditions $(\sigma_1, \sigma_2) \in  \Stab(\mathcal{D}_1) \times \Stab(\mathcal{D}_2)$ satisfying the conditions in Theorem \ref{thm:HNbig}.
\end{definition}

\begin{theorem}[{\cite[Theorem 4.3]{CP}}]\label{thm:gluecontinous}
Let $\textnormal{gl} \colon S(a) \rightarrow \Stab(\mathcal{D})$ be the map associating to \linebreak  $(\sigma_1, \sigma_2) \in  \Stab(\mathcal{D}_1) \times \Stab(\mathcal{D}_2)$ the corresponding glued pre-stability condition $\sigma$ on $\mathcal{D}$ (defined by Theorem \ref{thm:HNbig}). Then, the map $\rm{gl}$ is continuous on $S(a)$.
\end{theorem}

For $a \in (0,1)$, we have a precise description of the sets $S(a)$ for the semiorthogonal decomposition $\mathcal{T}_{C} = \langle \mathcal{C}_1, \mathcal{C}_2 \rangle$.

\begin{proposition}\label{prop:S(a)forTCOH}
Consider the semiorthogonal decomposition $\mathcal{T}_{C} = \langle \mathcal{C}_1, \mathcal{C}_2 \rangle$. For $a \in (0,1)$, we have $S(a)\cong \left\{ ((T_1, f_1),(T_2, f_2)) \in \GL \times \GL \colon f_1(0)\geq f_2(0) \textnormal{ and } f_1(a)\geq f_2(a)\right\}.$
\begin{proof}
Suppose $\sigma_i =(Z_i,\coh_i^{\theta_i}(C)[n_i])$ is a stability condition on $\mathcal{C}_i$  with $\theta_i \in [0,1)$ and $n_i \in \mathbb{Z}$, for $i=1,2$. Assume that these stability conditions satisfy the gluing condition, i.e.\ $n_1 + \theta_1 \geq n_2 + \theta_2$. Let $(T_i,f_i)$ be the elements in $\GL$ corresponding to $\sigma_i$ under the equivalence in Theorem \ref{thm:stabC} (for $i=1,2$). Note that $f_i(0)=n_i + \theta_i$ (for $i=1,2$) so the condition 1 in Theorem \ref{thm:HNbig} is equivalent to $f_1(0)\geq f_2(0)$. We end the proof by showing that condition 2 is equivalent to $f_1(a)\geq f_2(a)$. Indeed, we will have 
$\Hom_{\mathcal{D}}^{\leq 0}(i_* \mathcal{P}_1(a,a+1],j_* \mathcal{P}_2(a,a+1])=0$
if and only if the stability condition $\sigma'$ obtained from $\sigma$, acting by rotation of angle $a$ satisfies the gluing property. Hence, if we denote $\mathcal{P}(0,1]= \Coh(C)$ the standard heart associated to slope-stability $Z_\mu$, then $\mathcal{P}_i(a,a+1] = \mathcal{P}(f_i(a), f_i(a)+1]
=\coh^{f_i(a)}(C),$ for $i=1,2$ and they will satisfy the gluing condition if and only if $f_1(a)\geq f_2(a)$.
\end{proof}
\end{proposition}

\begin{example}\label{ex:condHNthm}
Consider the semiorthogonal decomposition $\mathcal{T}_{C} = \langle \mathcal{C}_1, \mathcal{C}_2 \rangle$. Let $\sigma_i =(Z_i,\coh_i^{r_i}(C))$ be a stability conditions on $\mathcal{C}_i$  with $r_i \in \mathbb{R}$, for $i=1,2$. 
\begin{enumerate}
\item If $r_1 > r_2$, then $\gl_{12}(\sigma_1,\sigma_2)$ is a locally finite pre-stability condition on $\T_C$
\item The CP-glued pair $\sigma=\gl_{12}(\sigma_1,\sigma_\mu)$ with $\sigma_1=\sigma_\mu g$ and  $g=(N_{1},f)\in \GL$, that we considered in Example \ref{ex:gl-action}, does not satisfy the conditions in Theorem \ref{thm:HNbig}.
\end{enumerate}
See \cite[Example 2.9.11]{EM1} for more details on these and more examples.
\end{example}

For small hearts that fulfill the condition 1 in Theorem \ref{thm:HNbig} but not 2, we may need a different strategy. 

\begin{proposition}[HN-property for $\mathbb{Q}$]\label{prop:HNrational}
Let 
$$\glu_{12}(\sigma_1,\sigma_2)=(Z_{12}=Z_1\circ i^* + Z_2 \circ j^!, \mathcal{A}_{12}=\glu_{12}(\mathcal{A}_1, \mathcal{A}_2))$$ 
be a CP-glued pair in $\mathcal{T}_{C}$ obtained by CP-gluing the stability conditions $\sigma_i=(Z_i,\mathcal{A}_i)$ on $\mathcal{C}_i$ with $\mathcal{A}_i = \Coh_i^{\theta_i}(C)$ for $\theta_i \in [0,1)$  for $i=1,2$.
If $\tan(\pi \theta_i)\in \mathbb{Q}$, for all $i$, then $Z_{12}$ has the HN-property.
\begin{proof}
It is easy to see that $0$ is an isolated point of $\Im Z_i(\mathcal{A}_i) \subset \mathbb{R}_{\geq 0}$ for $i=1,2$. Then, it follows from {\cite[Proposition 3.5]{CP}} that $Z_{12}$ has the HN-property on $\mathcal{A}_{12}$.
\end{proof}
\end{proposition}

\begin{corollary}\label{CP1}
Let 
$\glu_{12}(\sigma_1,\sigma_2)=(Z_{12}=Z_1\circ i^* + Z_2 \circ j^!, \mathcal{A}_{12}=\glu_{12}(\mathcal{A}_1, \mathcal{A}_2))$
be a CP-glued pair in $\mathcal{T}_{C}$ obtained by CP-gluing the stability conditions $\sigma_i=(Z_i,\mathcal{A}_i)$ on $\mathcal{C}_i$ with $\mathcal{A}_i = \Coh_i^{\theta_i}(C)$ for $\theta_i \in [0,1)$  for $i=1,2$, such that $\theta_1 \geq \theta_2$ but it does not belong to $S(a)$ for any $a \in (0,1)$. If $\tan(\pi \theta_i)\in \mathbb{Q}$, for all $i$, then $\glu_{12}(\sigma_1,\sigma_2)$ is a pre-stability condition.
\end{corollary}

\begin{example}[classic $\alpha$-stability as CP-gluing]\label{EXGP94}
Note that by {\cite[Theorem 5.3.11]{CDK1}}, pre-stability conditions on $\mathcal{T}_{C}$ with heart $\TCoh(C)$ are obtained as CP-gluing of two stability conditions on $D^b(C)$ with heart $\Coh(C)$ and the stability function is given as
$$Z(r_1,d_1,r_2,d_2)\coloneqq -A_1 d_1 -A_2 d_2 +B_1 r_1 + B_2 r_2 + i(C_1 r_1 +C_2 r_2)$$
where $A_i,B_i,C_i \in \mathbb{R}$ such that $A_i, C_i >0$, for $i=1,2$. In particular, we can recover the classic notion of $\alpha$-stability for holomorphic triples of vector bundles of Garc\'{i}a-Prada et al.\ in \cite{GP94} and \cite{GP2} by taking $\sigma_{\alpha}=(Z_{\alpha},\TCoh(C))$ where $Z_\alpha(r_1,d_1,r_2,d_2)\coloneqq -d_1 -d_2 -\alpha r_1 + i( r_1 + r_2)$ with $\alpha \in \mathbb{R}.$
\end{example}

\begin{remark}
There is a way to construct hearts from semiorthogonal decompositions which agrees with CP-gluing when the gluing condition \eqref{eq:glHom0} is satisfied. This is the well-known recollement introduced in \cite{BBD}. See Appendix A. and Lemma \ref{lem:jealousy} for examples of hearts that do not accept a stability function.
\end{remark}

 \subsubsection{Constructing pre-stability conditions via tilting}\label{CSCTILT}
In this section we construct pre-stability conditions in $\T_C$ whose hearts are not given by Proposition \ref{CP}. We follow the steps of \cite[Lemma\ 6.1]{BS8}, i.e.\ we use  weak stability functions on $\TCoh(C)$ to obtain torsion pairs on $\TCoh(C)$ via truncation of the HN-filtrations. After tilting in the sense of \cite[Proposition\ 2.1]{HRS}, we obtain hearts that admit Bridgeland stability functions. 

\brem The intuition of this construction comes from Proposition \ref{BH}. This proposition gives us a description of torsion pairs of $\TCoh(C),$ which after tilting will give us a heart of a pre-stability condition.
\erem

We define the following homomorphism:
\begin{eqnarray}\label{WNGStability}
Z \colon  \mathbb{Z}^4 &\la & \Co \\ \nonumber
 (r_1,d_1,r_2,d_2) &\mapsto &  D_1d_1+ (C_1-1)r_1+ i(r_1+r_2), \nonumber 
\end{eqnarray}
where $D_1,C_1\in \mathbb{Q},$ and $D_1<0.$ For every $E\in \TCoh(C)$ with $E\neq j_*(T),$ where $T$ is a torsion sheaf, we define the phase of $E$ as $\lambda(E)=(1/\pi)\arg(Z([E]))\in (0,1].$

%\bd Let $E,F\in \TCoh(C)$ we say that $$E<F \textnormal{ if } E\subsetneq F \textnormal{ , } %0<r_2(F)<r_2(E) \textnormal { and } \lambda(E)< \lambda(F),$$
%$$E\sim F \textnormal{ if } E\subseteq F \textnormal{ , } 0<r_2(F)<r_2(E) \textnormal{ and } \lambda(E)=\lambda(F).$$ 
%\ed

\brem \label{SSS} Let $0\la A\la B \la C \la 0$ be a short exact sequence where $A,B,C\in \TCoh(C)$ and $A,B,C \neq 0\la T,$ where $T$ is a torsion sheaf, then 
 $$ \lambda(A)<\lambda(B) \iff \lambda(B)<\lambda(C) \textnormal{ and } \lambda(A)>\lambda(B) \iff \lambda(B)>\lambda(C).$$
\erem

%\bd An \emph{almost-torsion-free triple} is an object $E\in \TCoh(C),$ sucht that $j^!(E) \neq \Co(x).$
%\ed
For a triple $E=E_1\xrightarrow{\varphi}E_2\in \TCoh(C),$ let $T(E_i)$ be the torsion part  and $F(E_i)$ be the \linebreak torsion-free part of of $E_i$ for $i=1,2.$ By the functoriality of the   torsion part, we obtain \linebreak $T(E)=T(E_1)\la T(E_2)$ and the following short exact sequence \begin{equation} \label{TPHT}
0\la T(E)\la E \la F(E)\la 0,
\end{equation} where $F(E)=F(E_1)\la F(E_2)=E/T(E).$ 

%after choosing splitting, we assume that $E_i=T(E_i)\oplus F(E_i),$ for $i=1,2,$  where $T(E_i)$ is the torsion part of $E_i$ and $F(E_i)$ is the torsion-free part of $E_i.$ We give  a short exact sequence 

 A triple $E=E_1\xrightarrow{\varphi}E_2\in \TCoh(C)$ is called \emph{torsion-free} if $T(E_i)=0,$ for $i=1,2.$ We define the \emph{torsion-free triple} of  $E=E_1\xrightarrow{\varphi} E_2\textnormal{ as } F(E)=F(E_1)\xrightarrow{f} F(E_2).$

%If $E$ is almost-torsion-free, it's almost-torsion-free part is $E_1\la E_2.$

\bd A torsion-free triple $E\in \TCoh(C)$ is called $\lambda$-semistable if for all non-zero subobjects $F\subseteq E$ we have $\lambda(F)\leq \lambda(E).$ 
\ed

We will now show that the $\lambda$-semistable objects admit HN-filtrations. The proof goes along the lines of the classical proof for $\mu$-stability on \href{http://www.mathe2.uni-bayreuth.de/stoll/lecture-notes/vector-bundles-Faltings.pdf}{curves}. 

\bl [HN-filtration for $\lambda$-stability.] \label{HNNG} Let $F=F_1 \xrightarrow{\varphi} F_2\in \TCoh(C)$ be a  torsion-free object, then there is a unique Harder--Narasimhan filtration i.e.\ there is an increasing filtration $0\subseteq E_1 \cdots \subseteq E_{n-1}\subseteq E_{n}=F$ where $G_i=E_i/E_{i-1}$ is $\lambda$-semistable for each $i=0,\ldots,n$ and $\lambda(G_{1})> \cdots > \lambda(G_{n-1})> \lambda(G_{n}).$
Moreover, this filtration is unique. 
\el
\bdem If $F$ is $\lambda$-semistable, there is nothing to prove. Let us consider the subobjects  $E$ of $F.$ Note that $E$ is also torsion-free.  Take the object $E_1$ with maximal $\lambda(E_1)$ among all the subobjects of $F$ and with  maximal imaginary part among all the subobjects of $F$ with maximal $\lambda$-phase. This object exists because the phase is bounded. Indeed, by the correspondence between slope and phase, it follows from the fact that $-D>0$ and by Riemann--Roch. 
As a consequence, the subobject $E_1$ is necessarily $\lambda$-semistable and $F/E_1$ is torsion free. We also have that for all $E$ with $0\neq E/E_1\subseteq F/E_1,$ we get $\lambda(E/E_1)< \lambda(E_1).$ We now apply the same construction to $F/E_1.$ We get the desired filtration. 
\edem

%\bdem Let us assume that $F/E_1$ has torsion. Note that there is $F'_1\subseteq F_1$ such that $i^*(E)\subseteq F'_1 \subseteq F_1,$ where $F_1/F'_1$ is torsion-free and the quotient $F'_1/i^*(E_1)$ has finite support. We take $\varphi(F'_1)\subseteq F_2,$ there is $F'_2\subseteq F_2$ such that $\varphi(F'_1)\subseteq F'_2 \subseteq F_2$ and $E_2\subseteq F'_2,$ where $F_2/F'_2$ is torsion-free and the quotient $F'_2/E_2$ has finite support. As a consequence, if $F'=F'_1\xrightarrow{\varphi} F'_2,$ then
%$$E_1\subseteq F' \subseteq F,$$ and $F/F'$ is torsion free. By the $\lambda$-semistability of $E_1$ and by  \hyperref[SSS]{Lemma \ref*{SSS}} we have that $\lambda(E_1)\leq \lambda(F'),$ by definition of $F',$ we also get
%$\lambda(E_1)\geq \lambda(F'),$ and $\Im Z(E_1)=\Im Z(F')$ and $\Re Z(E_1)<\Re Z(F').$
%As a consequence we obtain that $E_1=F'.$
%\edem

\bl \label{TPNG} Let $\phi=3/4.$ There is a torsion pair $ (\T,\mathcal{F})$ on the category $\TCoh(C)$ defined as follows: $E\in \T$ if the Harder--Narasimhan $\lambda$-semistable factors  $A_{i}$ of $F(E)$ satisfy $\lambda(A_i)>\phi$ and  $i^!(E)\in \Coh(C).$  We say that  $E\in\mathcal{F}$  if either $i^*(E)$ is torsion-free and the Harder--Narasimhan factors $A_{i}$ of $F(E)$ satisfy  $\lambda(A_i)\leq \phi$ or $i^*(E)=0.$
\el
\bdem %We first prove the following claim:

%\bcla If $E\in\T,$ then $i^!(E)\in \Coh(C).$
%\ecla

%\bdem Let $E=E_1\la E_2\in \T,$ it is enough to show it for $\lambda$-stable  torsion-free object with $[E]=(r_1,d_1,r_2,d_2).$ Indeed, this implies that if $E\in \T,$ then $i^!(F(E))=0.$ By hypothesis $i^!(T(E))=0.$
%Therefore, after applying $i^!$ to the triangle induced by the short exact sequence $0\la T(E)\la E\la F(E)\la 0$ we obtain that $i^!(E)=0.$

Note that if $E\in \T,$ by our definition of $\T$ and the correspondence between slope and phase, we have that $F(E)$ satisfies 
\begin{equation} \label{TNG}
\frac{-D_1d_1-(C_1-1)r_1}{r_1+r_2}>-\cot(3\pi/4)\textnormal{ i.e.\ }-D_1d_1-C_1r_1-r_2>0.
\end{equation}
where here $d_i=\deg(F(E_i))$ and $r_i=\rank(F(E_i)),$ for $i=1$, $2.$

%\blue{We wont need this remark, it is there because I was not sure how to define almost-torsion-free.}
%Note that $E\neq \Co(x)\la E_2,$ because  $0\la E_2$ is a subobject with phase $1/2$, then $$%\phi(\Co(x)\la E_2)%geq 1/2,$$  it implies $\frac{-D_1d_1}{r_2}\geq 0.$ Since $-D_1>0,$ we obtain $d_1\geq 0,$ which is a contradiction with the fact that $E_1=\Co(x).$  

%We consider the following short exact sequence
%\begin{equation} \label{NGH}
 %\xymatrix{0\ar[r] \ar[d] & E_1 \ar[r]^{id} \ar[d]_{f} & E_1 \ar[r] \ar[d]_{f}& 0 \ar[d] \ar[r] &0 \ar[d]\\
%0 \ar[r]& \Img(f) \ar[r] & E_2 \ar[r] & \Coker(f) \ar[r] & 0}
%\end{equation}
%First note that $0\la\Coker(f)$ is not of torsion.  Indeed, we prove it by contradiction if that were the case, then $\rk(\Coker(f))=0,$ then $\rk(\Img(f))=\rk(E_2),$ it implies that $\lambda(E_1\la \Img(f))=\lambda(E),$ which contradicts the $\lambda$-stability of $E.$

%Let us assume that $\Coker(f)\neq 0.$ Note that our assumption implies $E_2\neq 0$ and $r_2>0.$ Since  $0\la\Coker(f)$ has slope 0, then by \hyperref[SSS]{Lemma \ref*{SSS}} and the correspondence between slope and phase, it follows that $$\frac{-D_1d_1-(C_1-1)r_1}{r_1+r_2}\leq 0$$ thus $-D_1d_1-(C_1-1)r_1\leq 0.$  As $r_2>0,$ this clearly forces
%$$-D_1d_1-C_1r_1-r_2 \leq 0,$$   which contradicts our assumption. 

%Therefore, we have $\Coker(f)= 0$ and $i^!(E)=j^*(E)[-1]=\Ker(f)\in \Coh(C).$

%\edem

We show that $( \T, \mathcal{F})$ is a torsion pair of $\TCoh(C).$ First, we prove that $\Hom_{\T_C}(\T,\mathcal{F})=0.$ By our definition of stability we have that $\Hom_{\T_C}(E,F)=0,$ for all objects $E\in\T$ and $F\in \mathcal{F}$ that are torsion-free.

Let $E=E_1\xrightarrow{\varphi} E_2 \in \T$ and $G=G_1\la G_2\in \mathcal{F}.$ Let us consider the following short exact sequences as in the triangle \eqref{TPHT}
$$0\la T(E)\la E \la F(E)\la 0\textnormal{ and }0\la T(G)\la G\la F(G)\la 0.$$ By definition of $\FF$, we have that $F(G)\in \FF$ and as $i^*(G)$ is torsion-free, we get $i^*(T(G))=0$ and $T(G)\in \FF.$ Then it is enough to show $\Hom_{\TCoh(C)}(E,G)=0$ for $G=0\la H,$ for any $H\in\Coh(C)$ and $G=G_1\la G_2\in \FF$ where $G_1,G_2$ are torsion-free.

\textbf{Case 1:} $G\in \FF$ a torsion-free triple.

By definition $F(E)\in \T$ and by stability  we have $\Hom_{\TCoh(C)}(F(E),G)=0.$ Also \linebreak $\Hom_{\TCoh(C)}(T(E),G)=0$ as $G$ is torsion-free. 
Therefore, it follows that $\Hom_{\TCoh(C)}(E,G)=0.$

\textbf{Case 2:} $G=0\la H.$ We have $\Hom_{\T_C}(E,G)=\Hom_{\T_C}(E,j_*(j^!(G))),$ by adjointness $$\Hom_{\T_C}(E,j_*(j^!(G)))=\Hom_{D^b(C)}(j^*(E),j^!(G))=\Hom_{D^b(C)}(\Ker(\varphi)[1],H)=0,$$ because $j^*(E)[1]=i^!(E)=\Ker(\varphi)\oplus \Coker(\varphi)[1]$ is in $\Coh(C),$ which implies that $\Coker(\varphi)=0.$

%Let us consider the triangle $$j_*(j^!(F))\la F \la i_*(i^*(F))\la j_*(j^!(F))[1],$$ we obtain a long exact sequence $$\cdots\la\Hom_{\T_C}(E,j_*(j^!(F)))\la \Hom_{\T_C}(E,F)\la \Hom_{\T_C}(E,i_*i^*(F))\la\cdots, $$
%it is enough to show that $\Hom_{\T_C}(E,j_*(j^!(F)))=0$ and $\Hom_{\T_C}(E,i_*(i^*(F)))=0.$

%By adjointness, we have $$\Hom_{\T_C}(E,i_*(i^*(F)))=\Hom_{D^b(C)}(i^*(E),i^*(F))=0.$$ Indeed, note that $E_1\in\mathcal{T}_1,$ as $\T$ is closed under quotients and  that $F_1\in \mathcal{F}_1.$ As $r_2=0,$ then $F_1\la 0$ is its torsion-free part and it implies directly that $F_1\in \FF_1.$ 
 
%We also have $$\Hom_{\T_C}(E,j_*(j^!(F)))=\Hom_{D^b(C)}(j^*(E),j^!(F))=\Hom_{D^b(C)}(\Ker(f)[1],F_2)=0.$$

%Now let $E=E_1\xrightarrow{f} E_2 \in \T$ with $E_2$ of torsion and $F=F_1\la F_2\in \mathcal{F}$ arbitrary. 

Note that for a torsion-free object $E,$ by Lemma \ref{HNNG}, there is a short exact sequence \linebreak $0\la T\la E\la F\la 0,$ with $T$ torsion-free such that the HN-factors $A_i$ of $T$ satisfy that $\lambda(A_i)>\frac{3}{4}$ and  $F=F_1\xrightarrow{f'} F_2\in \FF$ also torsion-free. Note that $T=T_1\xrightarrow{t'} T_2$ is not necessarily surjective, however we have the following claim. 
\bcla Either $\Coker(t')=0$ or $\Coker(t')$ is a torsion sheaf.
\ecla
\bdem
Because of Lemma \ref{HNNG}, it is enough to show the statement for a $\lambda$-semistable object \linebreak $T$ satisfying $\lambda(T)>\frac{3}{4}.$ 
Let us assume that $\Coker(t')\neq 0.$ Note that $t'\neq 0.$ 
Then \linebreak $\Img(t'),\Coker(t')\neq 0,$ and as a consequence, we have 
\begin{equation} 
 \xymatrix{ T_1 \ar@{->>}[r]\ar[d]_{t'}& T_1 \ar[d]_{0}\\
 T_2 \ar@{->>}[r] & \Coker(t')} \textnormal{ and } \xymatrix{ T_1 \ar@{^{(}->}[r] \ar[d]_{t'} & T_1 \ar[d]_{t'}\\
\Img(t') \ar@{^{(}->}[r] & T_2. }
\end{equation} If $\rk(\Coker(t'))>0,$ then by $\lambda$-semistability of $T,$ we obtain $\lambda(T)=\frac{1}{2}<\frac{3}{4},$ which gives us a contradiction. 
\edem
We obtain the following short exact sequence
\begin{equation}\label{TFSES}
 \xymatrix{0\ar[r] \ar[d] & T_1 \ar[r] \ar[d]_{t'} & E_1 \ar[r] \ar[d]_{\varphi} & F_1 \ar[d] \ar[r] &0 \ar[d]\\
0 \ar[r]& \Img(t') \ar[r] & E_2 \ar[r] & F_2' \ar[r] & 0.}
\end{equation}

Note that $T'\coloneqq T_1\xrightarrow{t'} \Img(t')$ is in $\T.$ Indeed, let us consider the last short exact sequence in its HN-decomposition $0\la E\la T'\sur A\la 0 $ with $A=A_1 \la A_2$ a $\lambda$-semistable torsion free sheaf. We want to show that $\lambda(A)>\frac{3}{4}.$
Note that $E$ is also a subobject of $T,$ as a consequence we consider the short exact sequence $0\la E \la T \la T/E\la 0.$ We have that $T/E$ is a quotient of $T$ and therefore $\lambda(F(T/E))>\frac{3}{4}.$ We also have that $\frac{3}{4}<\lambda(F(T/E))=\lambda(A).$ Moreover, we get that $F$ is the torsion-free part of $F_1\xrightarrow{f'} F_2'$ and this implies that $F_1\xrightarrow{f'} F_2'\in \FF.$ Therefore, if $E_1\la E_2$ is torsion-free, the triangle \eqref{TFSES} gives us the decomposition of $E$ in $(\T,\FF).$

%Note that if $A_2=0,$ then $A=A_1\la 0$ is also a quotient of $T'$ and it satisfies that $\lambda(A)>\frac{3}{4}.$ 

Let $E=E_1\xrightarrow{\varphi} E_2\in \TCoh(C).$ Let us, again, consider the short exact sequence
\begin{equation} \label{NGH}
 \xymatrix{0\ar[r] \ar[d] & T(E_1) \ar[r] \ar[d]_{t} & E_1 \ar[r] \ar[d]_{\varphi}& F(E_1) \ar[d]_{f} \ar[r] &0 \ar[d]\\
0 \ar[r]& T(E_2) \ar[r] & E_2 \ar[r] & F(E_2)\ar[r] & 0.}
\end{equation}

Since $F(E)$ is torsion-free, as mentioned before there is a short exact sequence $$0\la T^{'} \la F(E)\la F'\la 0,$$
with $T'\in \T$ and $F^{'}\in \FF.$ Explicitly \begin{equation} \label{NGH1}
 \xymatrix{0\ar[r] \ar[d] & T'_1 \ar[r]^{l_1} \ar[d]_{t'} & F(E_1) \ar[r]^{g_1} \ar[d]_{f}& F'_1 \ar[d]_{f'} \ar[r] &0 \ar[d]\\
0 \ar[r]& T'_2 \ar[r] ^{l_2} & F(E_2) \ar[r]^{g_2} & F'_2\ar[r] & 0.}
\end{equation}

After choosing a splitting  $E_i=F(E_i)\oplus T(E_i),$ for $i=1,2,$ we have a morphism $$l\colon F(E_1)\lai E_1\xrightarrow{\varphi} E_2\la T(E_2).$$  
%Consider the morphism $\psi\colon T^{'}_1\oplus T(E_1)\la T(E_2)$ as the sum of $t$ and $l\circ l_1.$
 We now define the following morphism: 
 $\beta_1\colon T^{'}_1\oplus T(E_1)\xrightarrow{\begin{bmatrix}
l_1 & 0\\ 
 0 & \id \\ 
 \end{bmatrix}} F(E_1)\oplus T(E_1).$

%$$\alpha_1\colon T^{'}_1\oplus T(E_1)\xrightarrow{\varphi\circ \beta_{1}} \Img(\varphi\circ \beta_{1}),$$ 
 %$$\beta_2\colon \Img(\varphi\circ \beta_{1}))\lai F(E_2)\oplus T(E_2)\textnormal{, } 
 %\delta_2 \colon F(E_2) \la G$$ and
 
%$$\alpha_2\colon F^{'}_1\xrightarrow{l'} G.$$

%Where $i_1\colon \Img(l)\la T(E_2)$ and $i_2\colon \Img(t)\lai T(E_2),$ $G\coloneqq\Coker(\pi_2\circ\varphi\circ \beta_1)$ and $\pi\colon T(E_2)\la C$ is the projection. 
We obtain the following short exact sequence:
\begin{equation} 
\xymatrixcolsep{4pc}  \xymatrix{0\ar[r] \ar[d] & T_1'\oplus T(E_1)  \ar[r]^{\beta_1} \ar[d]_{\varphi\circ \beta_{1}}  & F(E_1)\oplus T(E_1) \ar[r]^{(g_1,0)} \ar[d]^{\varphi} & F^{'}_1 \ar[d]_{l'} \ar[r] &0 \ar[d]\\
0 \ar[r]&  \Img(\varphi\circ \beta_{1}) \ar[r]^{i} & F(E_2)\oplus T(E_2) \ar[r] & G \ar[r] & 0.}
\end{equation}

Note that $\varphi(x,y)=(f(x),l(x)+t(y)),$ for $(x,y)\in  F(E_1)\oplus T(E_1).$

We claim that $T'_1\oplus T(E_1)\xrightarrow{\varphi\circ \beta_1} \Img(\varphi\circ \beta_1)\in \T \textnormal{ and } 
F_1'\xrightarrow{l'} G \in \FF.$
Indeed, note that we have the following decomposition:
\begin{equation} 
 \xymatrixcolsep{4pc}\xymatrix{0\ar[r] \ar[d] & T(E_1)  \ar[r] \ar[d]_{t} & T'_1\oplus T(E_1) \ar[r] \ar[d]_{\varphi\circ \beta_1} & T'_1 \ar[d]_{t'} \ar[r] &0 \ar[d]\\
0 \ar[r]&  \Ker(\pi) \ar[r] & \Img(\varphi\circ \beta_1) \ar[r]^{\pi} & T'_2 \ar[r] & 0.}
\end{equation}
Note that, as seen from triangle \eqref{NGH1},
we have that $\pi\colon \Img(\varphi\circ \beta_1)\subseteq T'_2\oplus T(E_2)\la T'_2$ is just the projection. Note also that we abused the notation by ignoring the inclusion $l_2.$ 
As \linebreak  $\Ker(\pi)=\Img(\varphi \circ \beta_1\mid_{\Ker{t'}\oplus T(E_1)})$ is given by the points $ (0,x)\in \Img(\varphi\circ \beta_1)\subseteq T'_2\oplus T(E_2)$ and we can see it as a subset of $\subseteq T(E)_{2}.$ Thus, $\Ker(\pi)$ is a torsion sheaf and by definition $T'_2$ is torsion-free. As a consequence, the torsion-free part of $T'_1\oplus T(E_1)\xrightarrow{\varphi\circ \beta_1} \Img(\varphi\circ \beta_1)\in \T$ is $T'_1\xrightarrow{t'} T'_2.$

Analogously, we have the following decomposition:
\begin{equation} 
 \xymatrix{0\ar[r] \ar[d] & 0  \ar[r] \ar[d] & F'_1\ar[r] \ar[d]_{l'} & F'_1 \ar[d]_{f'} \ar[r] &0 \ar[d]\\
0 \ar[r]&  \Ker([g_2,0]) \ar[r]^{i} & G \ar[r]^{[g\circ g_2,0]} & F'_2 \ar[r] & 0.}
\end{equation}

We first check that $[g_2,0]$ is well-defined. As $G=F(E_2)\oplus T(E_2) / \Img(\varphi\circ \beta_1),$ then if $(x,y)$ in $\Img(\varphi\circ \beta_1),$ there are $(x',y')\in T'_1\oplus T(E_1),$ which satisfy $(f(x'),l(x')+t(y'))=(x,y).$ Note that $f(x')\in T'_2,$ therefore $g_2(f(x'))=0.$  Since $g_2$ is surjective, we have that $[g_2,0]$ is clearly surjective.

Note that $\Ker([g\circ g_2,0])= T_2\oplus T(E_2)/\Img( \varphi\circ \beta_1).$ As a consequence, $\Ker([g_2])$  is a torsion sheaf. 
We now obtain that the torsion-free part of $F'_1\xrightarrow{l'} G$ is the same as the one of $F'_1\la F'_2$ and this implies that
$F'_1\xrightarrow{l'} G\in \FF.$ 
\edem

%We now check that $$\Ker([g_2])= T(E_2)/\Img(\pi_2\circ \varphi\circ \beta_1\mid_{\Ker(t')\oplus T(E_1)})$$ and $i([x])=([0,x]).$ Let $[(y,x)]\in G$ if $g \circ g_2(y)=0,$ then $y\in T_2',$ as $t'$ is surjective there is $y'\in T'_1$ such that $f(y')=y.$ Moreover, we have that $(f(y'),x)-(0,l(y')-x)\in \Img(\varphi\circ\beta_{1}),$ as $$\varphi\circ \beta_1(y',0)=(f(y'),l(y')).$$

After tilting $\TCoh(C)$ with respect to the torsion pair of Lemma \ref{TPNG} as in \cite[Proposition\ 2.1]{HRS}, we obtain the heart $$\A_{r}=\{E\in \T_C\mid H^i(E)=0\textnormal{ for } i\neq 1,0, H^1(E)\in \T \textnormal { and } H^0(E)\in \mathcal{F}\},$$ where $r=\frac{\arg(C_1+D_1i)}{\pi}\in (-1,0].$ It has a corresponding torsion pair $\A_{r}=(\mathcal{F},\T[-1]).$

We will now define a stability function $Z_{r}$ on $\A_{r}$ such that the pair $(Z_{r},\A_{r})$ is a Bridgeland stability condition.

\brem If we have $\sigma=\gl_{12}(\sigma_1,\sigma_2)$ a pre-stability condition and we act by $\sigma_2^{-1}=g_2$ in $\GL$ and obtain a non CP-glued pre-stability condition, then the heart of $\sigma g_2$ is going to be  given by $\A_{r}$ (see Subsection \ref{SST12}).
\erem

We now define\begin{eqnarray}\noindent \label{NGStability} Z_{r} \colon \mathbb{Z}^4 &\la & \Co \\ \nonumber
 (r_1,d_1,r_2,d_2) &\mapsto & A_1d_1+ B_1r_1-d_2+i(D_1d_1+r_1C_1+r_2),  \nonumber 
\end{eqnarray} where 
$M=\begin{bmatrix}
-A_1 & B_1  \\
-D_1 & C_1  \\
\end{bmatrix}
$
 with $\det(M)>0,$ $\det(M+I)>0,$ $A_1,B_1\in \mathbb{R},$ $C_1,D_1 \in \mathbb{Q}$ and $D_1<0.$

\brem \label{RZVL}Note that if $E\in \TCoh(C)$ is a $\lambda$-semistable torsion-free triple, then \linebreak $\lambda(E)>\frac{3}{4}$ if and only if 
$\Im(Z_r(E))<0.$ Moreover,  $\lambda(E)\leq \frac{3}{4}$ if and only if $\Im(Z_r(E))\geq 0.$
\erem 
If we consider the same $C_1,D_1$ as above for the construction of $\A_{r},$ we obtain the following lemma. 

\bl \label{GNGSC} The group homomorphism $Z_{r}$ is a stability function on $\A_{r}$ 
\el

\bdem First of all, we show that the image  of $E\in\A_{r}$ under $Z_{r}$ lies in $\mathbb{H}\cup\R_{<0}.$ Let $E\in \T,$ then we consider the short exact sequence $0\la T(E)\la E\la F(E)\la 0.$ Note that by definition $F(E)\in\T$ and the right exactness of $\Coker(-)$ we obtain that $i^!(F(E))\in \Coh(C).$ 

We prove now that $Z_r(E[-1])\in \mathbb{H}\cup\R_{<0}.$
We first show that $Z_r(F(E)[-1])\in \mathbb{H}\cup\R_{<0}.$ It is enough to assume that $F(E)$ is $\lambda$-semistable and this follows directly from Remark \ref{RZVL}.

As $\rk(T(E_1))=\rk(T(E_2))=0,$ if $T(E_1)\neq 0,$ then $\deg(T(E_1))\geq 0$ and $\Im(Z_r((T(E)[-1]))$ is precisely $-\deg(T(E_1))D_1>0.$  
If $\deg(T(E_1))=0,$ then $T(E_1)=0,$ and $F(E_1)\cong E_1.$ This implies $\Im(Z_{r}(F(E)[-1]))=\Im Z_{r}(E[-1])>0.$ 
Since $Z$ is additive with respect to short exact sequences, we obtain that  $Z(E[-1])\in \mathbb{H}\cup\R_{<0}.$

We will now show that if $E\in \FF,$ then $Z_{r}(E)\in \mathbb{H}\cup\R_{<0}.$ Once again we consider the short exact sequence $0\la T(E)\la E\la F(E)\la 0.$ As $i^*(E)$ is a torsion-free sheaf  $T(E)=0\la T(E_2),$ where $T(E_2)$ is a torsion sheaf.  Note that by definition $T(E)\in \FF$ and $F(E)\in\FF.$ 
Clearly $Z_r(T(E))=-\deg(T(E_2))<0,$ as $T(E_2)$ is a torsion-sheaf. 
 
Let $F(E)\in\FF$ be a torsion-free $\lambda-$semistable object. If $F(E)$ satisfies $\lambda(F(E))<3/4,$ then \linebreak $-D_1d_1-C_1r_1-r_2<0$ and $Z(E)$ lies in the upper-half plane. 

We now assume that $F(E)=F(E_1)\xrightarrow{f} F(E_2)$ in $\mathcal{F}$ is a torsion-free object with \linebreak $D_1d_1+r_1C_1+r_2=0.$ It suffices to show our statement for $F(E)$ a $\lambda$-semistable object. We will now prove that $A_1d_1+B_1r_1-d_2<0.$ 

First note that if $F(E_1)=0,$ then $0=D_1d_1+r_1C_1+r_2=r_2$ and this implies $F(E)=0.$ If $F(E_2)=0,$ then $D_1d_1+r_1C_1=0,$ and $A_1d_1+B_1r_1-d_2=A_1d_1+B_1r_1<0,$ because $\det(M)>0.$ Therefore, we assume $F(E_1)$ and $F(E_2)\neq 0.$ Moreover, note that $\rk(\Coker(f))= 0.$ Indeed, otherwise $j_*(\Coker(f))$ would be a quotient of $F(E)$ with slope $\frac{1}{2},$ which gives us a contradiction. 

%Therefore, we assume $F(E_1)$ and $F(E_2)\neq 0.$ Note that $\rk(\Coker(f))= 0.$ Indeed, if not $j_*(\Coker(f))$ is a quotient of $F(E)$ with slope $\frac{1}{2},$ which gives us a contradiction. 

%As $F(E_2)\neq 0,$ then $0\la F(E_2)$ is also a subobject of $F(E),$ then $\lambda(F(E))\geq 1/2.$ Therefore $\lambda(F(E))=1/2,$ which implies that $$-D_1d_1-C_1r_1+r_1=0.$$
%By assumption $$D_1d_1+r_1C_1+r_2=0,$$ then $r_1+r_2=0,$ since $r_1,r_2\geq 0,$ we obtain that $r_1=r_2=0.$ As we have that $F(E_1)$ is torsion-free, we get that $F_1(E)=0$ which is a contradiction.
%\edem

%\bcla $i^{!}(F(E))\in \Coh(C).$
%\ecla 

%\bdem

%If $F(E_2)=0,$ it follows directly. Therefore, we assume that $F(E_2)\neq 0.$
 %We consider the following short exact sequence in $\TCoh(C)$
%begin{equation} \label{NGH}
 %\xymatrix{0\ar[r] \ar[d] & F(E_1) \ar[r]^{id} \ar[d]_{f} & F(E_1) \ar[r] \ar[d]_{f}& 0 \ar[d] \ar[r] &0 \ar[d]\\
%0 \ar[r]& \Img(f) \ar[r] & F(E_2) \ar[r] & \Coker(f) \ar[r] & 0.}
%\end{equation}
%First note that $\rk(\Coker(f))\neq 0$. Indeed, we prove it by contradiction if that were the case, then  $\rk(\Img(f))=\rk(F(E_2)),$ it implies that $\lambda(F(E_1)\la \Img(f))=\lambda(F(E)),$ which contradict the $\lambda$-stability of $E.$

%Since  $0\la\Coker(f)$ has phase 1/2, then by \hyperref[SSS]{Lemma \ref*{SSS}} and the stability of $F(E)$ it follows that $\lambda(E)<1/2.$  But $0\la F(E_2)$ is also a subobject of $F(E),$ then $\lambda(F(E))>1/2,$ which gives us a contradiction.  
%Therefore, we obtain $i^{!}(F(E))=\Ker(f)\in \Coh(C).$

%\edem
Note that $\lambda(F(E_1)\la \Img(f))=\lambda(F(E)),$ and moreover 
$[F(E)]=[F(E_1)\la \Img(f)]+ (0,0,0,d_2''),$ where $d_2''=\deg(\Coker(f))>0,$  $d''_1=\deg(\Img(f))$ and $d_2=\deg(F(E_2)).$
As $$A_1d_1+B_1r_1-d_2=A_1d_1+B_1r_1-d''_1-d_2''<A_1d_1+B_1r_1-d''_1$$
and $F(E_1)\la \Img(f)\in \FF,$ since it is a subobject of $F(E),$ it is enough to show our statement for torsion-free objects $F(E_1)\xrightarrow{f} F(E_2)\in \FF$ with $\Coker(f)=0.$  As a consequence, we assume $r_1=\rk(F(E_1))\geq \rk(F(E_2))=r_2.$ 

Since $K=i_*(i^{!}(F(E)))=\Ker(f)\la 0$ is a subobject of $F(E)$ in $\TCoh(C)$ and $\mathcal{F}$ is closed under subobjects, we obtain that  $K\in \mathcal{F}.$ Since $[i_*(i^{!}(E))]=[K]=(r_1-r_2,d_1-d_2,0,0),$ where $d_i=\deg(F(E_i))$ for $i=1,2,$ it follows that $-D_1(d_1-d_2)-C_1(r_1-r_2)\leq 0.$ By hypothesis $D_1d_1+r_1C_1+r_2=0.$ Therefore we obtain $D_1d_2 \leq  -(C_1+1)r_2.$

Now we want to show that $A_1d_1+B_1r_1-d_2<0.$ First note that $d_1=\frac{r_2+C_1r_1}{-D_1},$ we obtain
\begin{eqnarray} \label{NGD} \nonumber
A_1d_1+B_1r_1-d_2 &= &  (\frac{1}{-D_1})(A_1r_2+r_1A_1C_1 - r_1B_1D_1 +D_1d_2).  \nonumber
\end{eqnarray}

Since $-D_1>0,$ it is enough to show that  $A_1r_2+r_1A_1C_1 - r_1B_1D_1 +D_1d_2<0.$ We have
\begin{eqnarray} \nonumber
A_1r_2+r_1A_1C_1 - r_1B_1D_1 +D_1d_2 &{\leq}  &  r_2(A_1-C_1-1)-r_1(\det(M)) \\ \nonumber
&\leq&  r_2(A_1-C_1-1)-r_2(\det(M))\\
&=& r_2(-\det(M+I))<0,
\end{eqnarray}
as $-r_1\leq -r_2$ and also $\det(M)>0$ and $\det(M+I)>0.$ Since $Z_r$ is additive with respect to short exact sequences we obtain that $Z_r(E)\in \mathbb{H}\cup\R_{<0}$ for $E\in \FF.$
\edem
\bl \label{PrestaNG} If $A_1,B_1,C_1,D_1\in \mathbb{Q},$ the pair $(Z_{r},\A_{r})$ is a pre-stability condition. 
\el
\bdem We follow the steps of \href{https://arxiv.org/pdf/math/0307164.pdf}{\cite[Proposition 7.1]{BS8}.}  First note that by \href{https://arxiv.org/pdf/0912.0043.pdf}{\cite[Proposition B.2]{BMLP}}, it is enough to show that if $E\in \A_{r}$ and $0 \subset L_1\subset L_2 \ldots \subset L_i \subset \ldots \subset E,$ where $L_i$ belongs to the full subcategory  $\mathcal{P}'(1)$ of objects with phase one, the sequence stabilizes. 
As $L_i\in \A_{r}=(\FF,\T[-1]),$ we consider the short exact sequence
$0\la F_i \la L_i \la T_i[-1]\la 0,$ where $F_i\in \FF$ and $T_i\in \T.$ As $\Im Z_r(L_i)=0,$ we obtain that $\Im Z_r(T_i[-1])=0.$ Note that this implies that $T_i=0.$  Therefore, we obtain that $\mathcal{P}'(1)\subseteq \FF\subseteq \TCoh(C),$ as $\TCoh(C)$ is noetherian, our result follows. 
\edem

\brem \label{RSHBS} \label{TPNG1} \begin{enumerate}[leftmargin=0.5cm] \item By Lemma \ref{DISCRETELF} the pre-stability condition $\sigma_{r}=(Z_r,\A_r)$ is locally finite. 
\item Let $\sigma=(Z_r,\A_{r})$ be a pre-stability condition constructed in Lemma \ref{PrestaNG}. Note that $i_*(\Co(x))[-1],$ $l_*(\Co(x))[-1]$ and $j_*(\Co(x))$ are in $\A_{r}$ and $j_*(\Co(x))$ is stable of  phase one. 
\item We consider the torsion pair $( \T_1,\mathcal{F}_1)=\Coh(C),$ which is given by truncating the Harder--Narasimhan filtration with respect to $\sigma_1=(Z(r,d)=D_1d+(C_1-1)r+ir,\Coh(C))$ at $\frac{3}{4}$ in $\Stab(C).$ 
\end{enumerate}
\erem

    \bl \label{THI} We have that $\Coh_2(C)\subseteq \A_r \textnormal{ and } \Coh_1^{r}(C)\subseteq \A_r \textnormal{ and } \Coh_3^{r_3}(C)\subseteq \A_r,$
where $r_3=\frac{\textnormal{acot}(\frac{C_1+1}{D_1})}{\pi}.$
\el
\bdem Note that $\Coh_2(C)\subseteq \FF\subseteq \A_{r}.$
We consider the torsion pair $\Coh(C)=(\T_1,\FF_1)$ of Remark \ref{TPNG1}, which also induces the heart $\Coh^r_1(C)$ after a right tilt, i.e.\ $\Coh^r_1(C)=( \FF_1,\T_1[-1]).$
Let $E\in\Coh_1(C)$ be a $\mu$-semistable object. If $E=i_*(\Co(x)),$ then by Remark \ref{RSHBS}, we have $E[-1]\in \A_r.$  We assume that $E$ is torsion free. As the only possible subobjects or quotients of $E$ are in $\Coh_1(C),$ we have that $E$ is $\mu$-semistable if and only if it is $\lambda$-semistable. It follows directly from the definition of $\T$ and $\FF$ that
$i_*(\T_1)\subseteq \T \textnormal{ and } i_*(\FF_1)\subseteq \FF.$

We also consider a torsion pair $\Coh(C)=(\T_3,\FF_3),$ as in Remark \ref{SCCHR}, such that after taking the right tilt, $\Coh^{r_3}(C)=(\FF_3,\T_3[-1]).$
Let $E\in\Coh_3(C)$ be a $\mu$-semistable object. We have that $E=l_*(\Co(x))\in \A_r$  by Remark \ref{RSHBS}. It is easy to see by computing the slope of the Harder--Narasimhan factors that $l_*(\T_3)\subseteq \T$ and $l_*(\FF_3)\subseteq \FF.$
\edem

\section{The stability manifold \texorpdfstring{$\Stab(\T_C)$}{TEXT}}
\label{SMC}

Lemma \ref{GKR} is the main tool to prove that $\Co(x)$ is a $\sigma$-stable object for every $\sigma\in \Stab(C)$ and every $x\in C.$ In this section, we prove the analogous statement for $\T_C,$ to give a characterization of every $\sigma\in\Stab(\T_{C})$ in terms of the skyscraper sheaves. In Subsection \ref{SST12}, we follow closely the steps of \cite[Lemma\ 10.1]{BS8} to describe the possible hearts appearing in pre-stability conditions on $\T_C.$  We finally prove that every pre-stability condition on $\T_C$ has to be given by one of the already constructed pairs in Lemma \ref{PrestaNG} or Subsection \ref{CSCTILT}  i.e.\ either by CP-gluing or by tilting. 

\newpage

\brem \label{EFGKR} Due to the assumption that $g(C)\geq 1,$ if $\Hom_{\Coh(C)}(E,A)\neq 0,$ then \newline $\Hom_{\Coh(C)}(E,A\otimes \omega_C)\neq 0.$ If $f$ is a non-zero morphism in $\Hom_{\Coh(C)}(E,A),$ we denote by $f_C$ the non-zero morphism  in $\Hom_{\Coh(C)}(E,A\otimes \omega_C)$ associated to $f.$
\erem 

\bl\label{GKRC}
If there is a distinguished triangle in $\T_{C}$ of the form \begin{equation} \label{GKRT}
\xymatrix{E_1 \ar[r] \ar[d]_{\varphi_E} & X \ar[r] \ar[d]& A_1 \ar[d]^{\varphi_A} \ar[r] & E_1[1] \ar[d]^{\varphi_E[1]}\\
E_2 \ar[r] & 0 \ar[r] & A_2 \ar[r]^{l} & E_2[1]}
\end{equation}
 with $X \in \Coh(C)$ and $\Hom_{\T_{C}}^{\leq 0}(E,A)=0,$  then $E_1,A_1\in \Coh(C)$.
\el
\bdem
First of all note that $l_2\coloneqq j^!(l)$ is an isomorphism in $D^b(C).$
We now consider the functor 
$F  \colon  \T_C \la \Mor(D^b(C))$ defined as $F(E_1\xrightarrow{\varphi_E} E_2)=E_1\xrightarrow{[\varphi_E]} E_2$ at the level of objects. For $\Phi$ in $\Hom_{\T_C}(E,F),$ by \ref{SQ31}, we get the commutative square \begin{equation}
\xymatrix{E_1 \ar[r]^{i^*(\Phi)} \ar[d]_{[\varphi_{E}]}  & F_1 \ar[d]^{[\varphi_F]}\\
E_2 \ar[r]^{j^!(\Phi)} & F_2}
\end{equation}
in $D^b(C),$ which is in fact a morphism in $\Mor(D^b(C)).$ 

\begin{claim} The functor $F\colon \T_C\la \Mor(D^b(C))$ is full. 
\end{claim}
\begin{proof}
Let $(\varphi_1,\varphi_2)\in \Hom_{\Mor(D^b(C))}(F(E),F(G)),$ with $E,G\in \T_C.$ By definition we have the commutative diagram 
 \begin{equation}
\xymatrix{E_1 \ar[r]^{\varphi_1} \ar[d]_{[\varphi_{E}]}  & G_1 \ar[d]^{[\varphi_G]}\\
E_2 \ar[r]^{\varphi_2} & G_2.}
\end{equation}
Note that we obtain a commutative square in $\T_C$ given by \begin{equation}
\xymatrix{i_*(E_1) \ar[r]^{t_E} \ar[d]_{\varphi_1}  & j_*(E_2)[1] \ar[d]^{\varphi_2}\\
i_*(G_1) \ar[r]^{t_F} & j_*(G_2)[1].}
\end{equation}
By taking cones of the horizontal arrows and using Remark \ref{Uniquemorphism}, we obtain that there is a morphism $\Phi\in \Hom_{\T_C}(E,F)$ with $F(\Phi)=(\varphi_1,\varphi_2).$
\end{proof}

By \cite[Corollary 3.15]{HUYFM}, for every $G\in D^b(C),$ we have that $G=\oplus_{i\in \Z} G_i[-i]$ with $G_i=H^i(G)\in \Coh(C).$ We also have the canonical morphisms (which come from the genuine chain maps)
$$G^\alpha[-\alpha]\xrightarrow{i_{\alpha}} G \xrightarrow{\pi_{\beta}} G^{\beta}[-\beta],$$ for $\alpha,\beta\in \Z.$ Moreover, note that if we have a morphism $\psi\colon G_1\la G_2$ in $D^b(C)$, then $H^i(\psi)[-i]=\pi^{2}_{i}\circ \psi \circ i^1_{i}\colon G^{i}_1[-i] \la G_2^{i}[-i].$

For  $G=G_1\xrightarrow{\varphi_G}G_2\in \T_C$ 
we get that $F(G)\in \Hom_{D^b(C)}(\oplus_i G_1^i[-i],\oplus_i G_2^i[-i]),$ with $G_j^i\in \Coh(C),$ for $j=1,2$ and $i\in \mathbb{Z}.$  Due to the fact that $\Coh(C)$ has homological dimension one, there is a non-zero morphism ${t_G}^i$ from $ G^i_1[-i]\la G_2^i[-i]\oplus G_2^{i-1}[-i+1]$ to $G_1\xrightarrow{\varphi_G}G_2$ in $\Mor(D^b(C)).$ We construct the morphism
$$\xymatrix{ A_1 \ar[r]^{\pi_i\circ[\varphi_A]} \ar[d]^{[\varphi_A]} & A_2^i[-i] \ar[r]^{id_C[-i]} \ar[d]^{\id} & A^{i}_2[-i]\otimes {\omega_C} \ar[r]^{l[-i]\otimes \omega_C} \ar[d]^{l[-i]\otimes \omega_C} &  E^{i+1}_2[-i]\otimes \omega_C \ar[d]
 \ar[r]_{t^i_{\Ser_{\T_C}(E)}} & \ar[d]^{[i_E]} E_2\otimes \omega_C[1] \\ A_2 \ar[r]^{\pi_i} &  A_2^i[-i] \ar[r] & E^{i+1}_2[-i]\otimes \omega_C  \ar[r]&  \C(\varphi_E)^{i+1}[-i]\oplus  \C(\varphi_E)^{i}[-i+1]\otimes \omega_C  \ar[r] &   C(\varphi_E)[1]}$$ in  $\Mor(D^b(C)).$ 
 Note that the morphism $\id_C$ is given by Remark \ref{EFGKR}.

 We obtain a morphishm $A\la \Ser_{\T_C}(E)$ in $\Mor(D^b(C)).$ Since $F$ is full, we have a  morphism in $\Hom_{\T_C}(A,\Ser_{\T_C}(E)).$ By Serre duality, we obtain $$\Hom_{\T_C}(A,\Ser_{\T_C}(E))\cong \Hom_{\T_C}(E,A)^*.$$ By hypothesis, this morphism has to be zero. This directly implies that $\id_C[-i]\circ (\pi_i\circ [\varphi_A])=0,$ because $l\otimes \omega_C $ is an isomorphism. Since the construction does not depend on the $i\in \Z,$ we obtain that the induced morphism $A^i_1[-i]\xrightarrow{\id_C[-i]\circ\pi_i\circ [\varphi_A]\circ i_i} A^i_2[-i]\otimes \omega_C$ is zero. Due to the fact that it is precisely the morphism $(\pi_i\circ [\varphi_A]\circ i_i)_C[-i]$ in $\Coh(C)[-i],$ as in Remark \ref{EFGKR}, we obtain that $A^i_1\xrightarrow{\pi_i\circ [\varphi_A]\circ i_1} A^i_2$ is also zero. The triangle \eqref{GKRT} in $\T_C$ induces a long exact sequence in cohomology and we get 
 $$\xymatrix{A_1^i \ar[r]^{\cong}\ar[d]_{\pi_i\circ [\varphi_A]\circ i_i}  & E_1^{i+1}\ar[d]^{\pi_{i+1}\circ [\varphi_E]\circ i_{i+1}}\\
A_2^i \ar[r]^{\cong} & E_2^{i+1},}$$ as the morphism in the component $i$ in $\Mor(D^b(C))$ is precisely the one given by the cohomology. Therefore $\pi_i\circ [\varphi_E]\circ i_1=0,$ for all $i\neq -1,0.$ 

This implies that there is a morphism from $E^i_1[-i]\la E^{i-1}_2[-i+1]$ to  $E_1\xrightarrow{\varphi_E} E_2,$ which, moreover, is a split monomorphism for $i\neq 0,1.$ In addition, it provides a split monomorphism from $A^i_1[-i]\la A^{i-1}_2[-i+1]$ to $A_1\xrightarrow{\varphi_A} A_2$ for $i\neq -1,0.$

We get the isomorphism 

$$\xymatrix{ A^i_1[-i] \ar[r] \ar[d] & A_1 \ar[r] \ar[d]^{\varphi_A} & E_1[1] \ar[r] \ar[d]^{\varphi_E[1]} &  E^{i+1}_1[-i] \ar[d]
  \\ A^{i-1}_2[-i+1] \ar[r]^{} &  A_2 \ar[r]^{l} & E_2[1] \ar[r] &  E^{i}_2[-i+1]}$$ in $\Mor(D^b(C))$, were the morphism of the first row is the isomorphism induced by the long exact sequence of cohomology of 
  the triangle \eqref{GKRT}.
  
  From the isomorphism above, we construct the morphism  $$\xymatrix{ E_1[1] \ar[r] \ar[d]_{[\varphi_E]} &  E^{i+1}_1[-i] \ar[r]^{\cong} \ar[d] & A^i_1[-i] \ar[r] \ar[d] &  A_1 \ar[d]_{[\varphi_A]}
  \\E_2[1] \ar[r]^{} &  E^{i}_2[-i+1]  \ar[r]^{\cong} & A^{i-1}_2[-i+1] \ar[r] &  A_2}$$ in $\Mor(D^b(C)).$ Once again, since the functor $F$ is full, this needs to be the zero morphism. As a consequence, we obtain that $A^i_1[-i], A^{i}_2[-i]$ are zero for all $i\neq 0,-1.$
  We now study the remaining object in the long exact sequence of cohomology \begin{equation}\label{GKRDIA1} \xymatrix{ 0 \ar[r ] \ar[d] & A^{-1}_1 \ar[r]_{t_1} \ar[d]_{0} &  E^{0}_1 \ar[r] \ar[d] & X \ar[r] \ar[d] &  A^{0}_1 \ar[d]_{0} \ar[r]_{t_2} & E^{1}_1 \ar[r] \ar[d] & 0 \ar[d]
  \\0 \ar[r] &  A^{-1}_2  \ar[r]^{\cong} &  E^{0}_2  \ar[r]^{} & 0 \ar[r] & A^{0}_2 \ar[r]^{\cong} & E^{1}_2 \ar[r] & 0.}\end{equation}
  
  We will now show that $t_i=0,$ for $i=1,2.$
  If $t_1\neq 0,$ by Remark \ref{EFGKR}, we obtain a non-zero morphism ${t_1}_C\colon A^{-1}_1\rightarrow E^0_1\otimes \omega_C.$ By Serre duality $$\Hom_{D^b(C)}(A_1^{-1},E^0_1\otimes \omega_C)\cong \Hom_{D^b(C)}(E^0_1[-1], A_1^{-1})$$ and there is a non-zero morphism $t'_1\colon E^0_1[-1] \la A^{-1}_1.$ As there is a split monomorphism from $E_1^0\la E^0_2$ to $E,$ we get the non-zero morphism 
  $$\xymatrix{ E_1 \ar[r] \ar[d]_{\varphi_E} &   E^0_1 \ar[r]^{t'_1[1]} \ar[d] & A_1^{-1}[1] \ar[r] \ar[d]_{0} &  A_1 \ar[d]_{\varphi_A}
  \\E_2 \ar[r] & E^0_2  \ar[r]_{0} & A_2^{-1}[1] \ar[r] &  A_2}$$  in $\Mor(D^b(C)).$ 
  As above, this implies that $t'_1=0$ and therefore $t_1=0.$
  We will now prove that $t_2=0$.
%\bl\label{GKRC}
%Given a distinguished triangle in $\T_{C}$ of the form \begin{equation} \label{GKRT}
%\xymatrix{E_1 \ar[r] \ar[d]_{\varphi_E} & X \ar[r] \ar[d]& A_1 \ar[d]^{\varphi_A} \ar[r] & E_1[1] \ar[d]^{\varphi_E[1]}\\
%E_2 \ar[r] & 0 \ar[r] & A_2 \ar[r]^{l} & E_2[1]}
%\end{equation}
 %with $X,A_1 \in \Coh(C),$ $A_2\in \Coh(C)[1]$
 %and $\Hom_{\T_{C}}^{\leq 0}(E,A)=0,$  then $E_1\in \Coh(C).$
%\el

The triangle $A_1\la A_2 \la C(\varphi_A)$ induces a long exact sequence of cohomology. It follows that $H^i(C(\varphi_A))=0$ unless $i=-1,$ i.e. $C(\varphi_A)\in \Coh(C)[1].$

Note that $A_2[-1]\la C(\varphi_A)[-1]\la A_1$ is a short exact sequence in $\Coh(C),$ then the morphism $C(\varphi_A)[-1]\xrightarrow{\pi_A} A_1$ is surjective in $\Coh(C).$

Let us assume that $t_2\neq 0.$
By Remark \ref{EFGKR}, there is a non-zero morphism $t'_2\colon A_1\la E^1_1\otimes \omega_C.$ We now consider $t_3\coloneqq t'_2\circ\pi_A$ in $\Coh(C).$ Since $\pi_A$ is surjective and $t'_2\neq 0,$ then $t_3\neq 0.$ 

This induces the morphism $\Phi\colon A \la \Ser_{\T_C}(E)$ given by

$$\xymatrix{ A_1 \ar[r]^{t'_2} \ar[d]_{\varphi_A} &   E^1_1\otimes \omega_C \ar[r]^{i} \ar[d] & E_1[1]\otimes \omega_C \ar[r]^{\varphi_E\otimes \id} \ar[d] & E_2[1]\otimes \omega_C \ar[d]  
  \\A_2 \ar[r] & 0 \ar[r]& 0 \ar[r] & \C(\varphi_E)[1]\otimes \omega_C}$$

 in $\Hom_{\T_C}(A,\Ser_{\T_C}(E)).$ Note that $i$ is a split monomorphism. By Serre duality, 
$$\Hom_{\T_C}(A,\Ser_{\T_C}(E))\cong \Hom_{\T_C}(E,A)^*=0.$$ Therefore, we have that $\Phi=0.$ After applying $i^!,$ we obtain  $$i^!(\Phi)=\C(\varphi_A)[-1]\xrightarrow{t_3} E^1_1\otimes \omega_C \xrightarrow{i} E_1[1]\otimes \omega_C,$$ which has to be zero. 
As consequence $t_3=0$ and we have a contradiction.
Hence, we get $A^{-1}_1=E^1_1=0,$ which implies that $A_1,E_1\in \Coh(C).$\edem

\brem From now on we assume that all pre-stability conditions on $\T_C$ that we are considering are locally finite.
\erem

\bp\label{2of3}
Let $X$ be either a skyscraper sheaf $\Co(x)$ or a line bundle $\Li$ on $C$. For any pre-stability condition $\sigma$ on $\T_{C}$, if $i_*(X)$ is not $\sigma$-semistable, then $j_*(X)$ and $l_*(X)$ are $\sigma$-stable.
\ep
\bdem We assume that $i_*(X)$ is not $\sigma$-semistable. Therefore, we consider the last triangle of its Harder--Narasimhan filtration $E\la i_*(X) \la A \la E[1]$ with $\Hom^{\leq 0}_{\T_{\C}}(E,A)=0$ and $A\in \T_{C}$ semistable. By Lemma \ref{GKRC}, we have $E_1,A_1\in \Coh(C),$ thus we obtain a short exact sequence 
$0\la E_1\la X\la A_1\la 0,$ in $\Coh(C)$ which is not possible. Hence, either $E_1=0$ or $A_1=0.$
If $E_1=0,$ we obtain a contradiction. As a consequence,  $A_1=0$ and $E_1\cong X.$ By adjointness we have $$\Hom^{\leq 0}_{\T_C}(E,A)=\Hom^{\leq 0}_{\T_C}(E,j_*(A_2))=\Hom^{\leq 0}_{D^b(C)}(C(\varphi_{E}),A_2)=0.$$ Exactness of the functor $j^*$ yields the triangle
$$\dtri{C(\varphi_{E})}{X[1]}{A_2}.$$ Since $\Hom^{\leq 0}(C(\varphi_{E}),A_2)=0,$ due to Lemma \ref{GKR}, the classical GKR lemma for curves, we get that $C(\varphi_{E})[-1], A_2[-1]\in \Coh(C)$  which once again is not possible. As $A_2$ cannot be zero, we get $C(\varphi_{E})=0$ and $A_2\cong X[1].$ This implies that $A\cong j_*(X)[1]$ and that $E\cong l_*(X)\in \C_3.$ As a consequence, $j_*(X)$ is $\sigma$-semistable. 

We will now show that $l_*(X)$ is $\sigma$-semistable. We proceed by contradiction. If $l_*(X)$ is not $\sigma$-semistable, we examine the last triangle of its HN-filtration and we apply the Serre functor and the autoequivalence $G_{\omega^*_C}.$ 
Arguing as above, we obtain $B\cong i_*(X)$ and $i_*(X)$ is $\sigma$-semistable, which contradicts our assumption. Therefore $l_*(X)$ is $\sigma$-semistable. 

We prove now that $l_*(X)$ and $j_*(X)$ are $\sigma$-stable. We start by proving by contradiction that $l_*(X)$ is stable. If $l_*(X)$ is not $\sigma$-stable, we consider its Jordan-H\"older filtration. Note that all its $\sigma$-stable factors $A_i$ have the same phase $\phi.$
We assume that $\Hom_{\T_{C}}(A_{i_{0}},l_*(X))\neq 0$ for a $\sigma$-stable factor $A_{i_0}.$ Therefore by \href{https://arxiv.org/abs/1111.1745}{\cite[Example 1.6]{HUYSC}}, we have that all the stable factors of $l_*(X)$ are isomorphic to $A_{i_0}.$ Hence, $[l_*(X)] = n[A_{i_0}]$, where $n$ is the number of stable factors. Since $[l_*(\Co(x))] = (0,1,0,1)$ and $[\Li] = (1, \deg(\Li),1, \deg(\Li))$, we must have $n = 1$, i.e.\ $l_*(\Co(x))$ and $l_*(\Li)$ are stable. 

An analogous proof works for the stability of $j_*(X).$ Instead of using the Serre functor, we use $\Ser^{-1}.$ 
Consequently, we obtain that $j_*(\Co(x))$ and $j_*(\Li)$ are stable. 
\edem
\brem  \label{HNX0}If $X$ is either $\Co(x)$ or $\Li$ as above, we use Proposition \ref{2of3} to prove that if $j_*(X)$ ($l_*(X)$) is not $\sigma$-semistable, then $i_*(X)$ and $l_*(X)$ ($j_*(X)$ and $i_*(X)$) are $\sigma$-stable. Meaning that if one of the objects $i_*(X),j_*(X), l_*(X)$ is not $\sigma$-semistable then the other two have to be $\sigma$-stable.
%We just apply the Serre functor and tensor by $l_*(\omega^*_C)$.
\erem 

\brem \label{HNX} If $i_*(X)$ is not $\sigma$-semistable, where
$X$ is either $\Co(x)$ or $\Li,$ then by Proposition \ref{2of3}, we obtain the HN-filtration for $i_*(X).$ It is given precisely by $$l_*(X)\la i_*(X)\la j_*(X)[1]\la l_*(X)[1].$$
After applying the Serre functor, we obtain the corresponding HN-filtration for $j_*(X)$ and $l_*(X).$
\erem

Moreover, we define 
$$\phi^0_{x}\coloneqq \phi_{\sigma}(i_*(\Co(x))) \textnormal { and } \phi^1_{\Li}\coloneqq \phi_{\sigma}(i_*(\Li)),$$
$$\phi^2_{x}\coloneqq \phi_{\sigma}(j_*(\Co(x))) \textnormal { and } \phi^3_{\Li}\coloneqq \phi_{\sigma}(j_*(\Li)),$$
$$\phi^4_{x}\coloneqq \phi_{\sigma}(l_*(\Co(x))) \textnormal { and } \phi^5_{\Li}\coloneqq \phi_{\sigma}(l_*(\Li)).$$

If $\Li=\Ot_C,$ then $\phi_1=\phi_{\sigma}(i_*(\Ot_C)),$ $\phi_3=\phi_{\sigma}(j_*(\Ot_C)),$ and $ \phi_5=\phi_{\sigma}(l_*(\Ot_C)).$

\brem Note that if $i_*(\Co(x))$ is not $\sigma$-semistable, then $\phi^4_{x}>\phi^2_{x}+1.$ Similarly, if $i_*(\Li)$ is not $\sigma$-semistable, then $\phi^5_{\Li}>\phi^3_{\Li}+1.$ 
\erem

\brem \label{HNX1} After applying the Serre functor to the HN-filtration of $i_*(X)$, we obtain the  analogous results for $j_*(X)$ and $l_*(X).$ For example, if $l_*(\Co(x))$ is not $\sigma$-semistable, then  $\phi^2_{x}>\phi^0_{x}.$ 
\erem

\brem \label{JHX} If $i_*(X)$ is strictly $\sigma$-semistable, where
$X$ is either $\Co(x)$ or $\Li,$ then $j_*(X)$ and $l_*(X)$ are $\sigma$-stable. Moreover, a Jordan-H\"older filtration is given by
$$l_*(X)\la i_*(X)\la j_*(X)[1]\la l_*(X)[1].$$
\erem

It now makes sense to define the following sets:

\bd We define the set $\Theta_{ij}$ of pre-stability conditions on $\T_C$ for $ij=12,23$ or $31$ as follows: \begin{alignat}{7}\nonumber \Theta_{12}&=\{&\sigma &\mid& i_*(\Co(x)), j_*(\Co(x)), i_*(\Li), j_*(\Li) \: \: \rm{stable} \textnormal{ for all line bundles $\Li\in\Coh(C)$ and all $x\in C$ }  &\}&,\\ \nonumber
\Theta_{23}&= \{&\sigma &\mid& j_*(\Co(x)), l_*(\Co(x)), j_*(\Li), l_*(\Li) \: \: \rm{stable} \textnormal{ for all line bundles $\Li\in\Coh(C)$ and all $x\in C$ } &\}&,\\ \nonumber
\Theta_{31}&=\{&\sigma &\mid& i_*(\Co(x)), l_*(\Co(x)), i_*(\Li), l_*(\Li) \:  \:  \rm{stable}  \textnormal{ for all line bundles $\Li\in\Coh(C)$ and all $x\in C$ }&\}&.
\end{alignat}
Recall that we assumed that all the pre-stability conditions are locally finite. 
\ed

\brem Note that 
$\Ser_{\T_C}(\Theta_{12})=\Theta_{23}\textnormal{ and }\Ser_{\T_C}(\Theta_{23})=\Theta_{31}\textnormal{ and }\Ser_{\T_C}(\Theta_{31})=\Theta_{12}.$
\erem

\bt\label{TEO1} If $\sigma$ is a pre-stability condition on $\T_C,$ then $\sigma\in \Theta_{12}\cup \Theta_{23}\cup \Theta_{31}.$
\et
\bdem We first assume that
$\sigma\notin \Theta_{23}$ and we prove that $\sigma\in \Theta_{12}$ or $\sigma\in \Theta_{31}.$ Thus,  there is a line bundle $\Li$ such that $j_*(\Li)$ or $l_*(\Li)$ is  not $\sigma$-stable, or either $j_*(\Co(x))$ or $l_*(\Co(x))$ is not $\sigma$-stable for some $x\in C.$ Assume that there is $x\in C$ such that $l_*(\Co(x))$ is not $\sigma$-stable. We need to show that $\sigma\in \Theta_{12},$ as it cannot be in $\Theta_{31}.$ By Remark \ref{HNX0},  it follows that  $j_*(\Co(x))$ and $i_*(\Co(x))$ are $\sigma$-stable. 
We will now show that for every line bundle $\Li$ and every $x\in C,$  we have that $j_*(X)$ and $i_*(X)$ are $\sigma$-stable, where $X=\Li$ or $X=\Co(x).$

We prove this by contradiction, assume that there is a line bundle $\Li,$ such that $i_*(\Li)$ is not $\sigma$-stable, which implies that $j_*(\Li)$ and $l_*(\Li)$ are $\sigma$-stable. Since $l_*(\Co(x))$ is not $\sigma$-stable, by Remark \ref{HNX} we obtain that $\phi_{x}^2\geq \phi^0_{x}.$ Analogously, since $i_*(\Li)$ is not $\sigma$-semistable, then we get $\phi^3_{\Li}+1 \leq \phi^{5}_{\Li}.$ Let us consider the non-zero morphism  $j_*(\Co(x))\la j_*(\Li)[1].$ Since by hypothesis both are stable and not isomorphic, we obtain that $\phi^2_{x}< \phi^{3}_{\Li}+1.$ 
We have a non-zero morphism $l_*(\Li)[1]\la i_*(\Co(x))[1],$ as both objects are $\sigma$-stable and not isomorphic, we get $\phi^5_{\Li}<\phi^0_{x}.$ We now put all the inequalities together, yielding 
$$\phi^0_{x}\leq \phi^2_{x}<\phi_{\Li}^3+1\leq\phi^5_{\Li}<\phi^0_{x},$$ which is a contradiction. Therefore $i_*(\Li)$ has to be $\sigma$-stable for all line bundles. Analogously we prove that $j_*(\Li)$ has to be also stable for all line bundles $\Li.$ 

We now assume that there is a point $y\in C,$ such that $j_*(\Co(y))$ is not $\sigma$-stable. Then, by Remark \ref{HNX0}  we obtain $\phi^0_{y}-1\geq\phi_{y}^4.$  Note that $[i_*(\Co(x))]=[i_*(\Co(y))]$ in the Grothendieck group. As a consequence, we obtain $\phi^0_{x}=\phi^0_{y}+m,$ with $m\in \Z.$ But as $i_*(\Ot_C)$ is $\sigma$-semistable and we have non-zero  morphisms from $i_*(\Ot_C)$  to $i_*(\Co(x)) $ and to $i_*(\Co(y),$ we also get non-zero morphisms from $i_*(\Co(x))$ and $i_*(\Co(y))$ to $i_*(\Ot_C)[1],$
we obtain 
$\phi_1<\phi^0_{y}< \phi_1+1\textnormal{ and } \phi_1<\phi^0_{x}<\phi_1+1,$
which implies that $\phi^0_{y}=\phi^0_{x}.$

Since we have a non-zero morphism $j_*(\Ot_C)\la l_*(\Co(y)),$ we obtain $\phi_3<\phi^4_{y}$ and from the morphism $j_*(\Co(x))\la j_*(\Ot_C)[1]$ we obtain $\phi^2_{x}<\phi_3+1.$

As a consequence, we get $\phi_3+1<\phi^4_{y}+1\leq\phi^0_{y}=\phi^0_{x}\leq\phi^2_{x}<\phi_3+1,$ which is a contradiction. Therefore, we obtain that $j_*(\Co(y))$ is $\sigma$-stable. Then $\sigma\in \Theta_{12}.$
The other cases follow analogously.
\edem

\bc \label{SPSK} Let $\sigma$ be a pre-stability condition on $\T_C,$ with $i_*(\Co(x))$ $\sigma$-stable for $x\in C$. Then, for all $y\in C,$ we have that $i_*(\Co(y))$ is $\sigma$-stable and $\phi^0_{x}=\phi^0_{y}.$
\ec
\bdem By Theorem \ref{TEO1}, we have that if $\sigma\in \Theta_{12}$ or $\sigma\in \Theta_{31},$ it follows directly that for all $y\in C$ we have that $i_*(\Co(y))$ is $\sigma$-stable.
Note that $[i_*(\Co(x))]=[i_*(\Co(y))]$ in the Grothendieck group. This implies that $\phi^0_{x}=\phi^0_{y}+m,$ with $m\in \Z.$ By stability, we obtain $m=0.$
\edem

\brem Analogously, by using Serre duality, we prove Corollary \ref{SPSK} for $j_*(\Co(x))$ and $l_*(\Co(x)).$ 
\erem
By Corollary \ref{SPSK}, if $i_*(\Co(x))$ is $\sigma$-stable for some $x\in C,$ then $ \phi^0_{x}$  does not depend on $x.$ Therefore, we define $\phi_0\coloneqq \phi^0_{x}$,  $\phi_2\coloneqq \phi^2_{x}$, $\phi_4\coloneqq \phi^4_{x}.$
\bl \label{NGSCStability} \label{NGSL}\label{NGSL2}\label{NGSL3}  Let $\sigma=(Z_{r},\A_{r})$ be a pre-stability condition constructed in Lemma \ref{PrestaNG}. Then \begin{enumerate}[leftmargin=0.5cm]
\item $i_*(\Co(x))$ $l_*(\Co(x))$ and $j_*(\Co(x))$ are stable. 
\item $i_*(\Ot_C)$ and $j_*(\Ot_C)$ are $\sigma$-stable if and only if $\phi_1<\frac{3}{2}.$
\item $i_*(\Ot_C),$ $j_*(\Ot_C)$ and $l_*(\Ot_C)$ are $\sigma$-stable if and only if $\frac{1}{2}=\phi_3<\phi_5<\phi_1<\frac{3}{2}.$
\item $j_*(\Ot_C)$ and $l_*(\Ot_C)$ are $\sigma$-stable and $i_*(\Ot_C)$ is not $\sigma$-stable if and only if $\phi_1\geq \frac{3}{2}$ and $\phi_5>\frac{1}{2}.$
Moreover, we have that $i_*(\Ot_C)$ and $l_*(\Ot_C)$ are $\sigma$-stable and $j_*(\Ot_C)$ is not $\sigma$-stable if and only if $\phi_1\geq \frac{3}{2}$ and $\phi_5\leq\frac{1}{2}.$
\end{enumerate}
\el
\bdem We start by proving the first part. By Remark \ref{RSHBS}, we have that $i_*(\Co(x))[-1],$ $l_*(\Co(x))[-1]$ and $j_*(\Co(x))$ are in $\A_{r}$ and that $j_*(\Co(x))$ is stable of phase one. We will now show that $i_*(\Co(x))[-1]$ is $\sigma$-stable. By contradiction we first assume that $i_*(\Co(x))$ is not $\sigma$-stable. As a consequence of Remark \ref{HNX1}, we have that $l_*(\Co(x))$ is $\sigma$-stable and $\phi_4\geq \phi_2+1,$ but $\phi_2=1$ and $1<\phi_4<2,$ which gives a contradiction. The same reasoning works to prove that $l_*(\Co(x))[-1]$ is stable. 
For the second part, note that $\phi_1=\phi(i_*(\Ot_C))$ and $\phi_3=\phi(j_*(\Ot_C))=\frac{1}{2}$ makes sense, as Lemma \ref{THI} implies that $j_*(\Ot_C)\in \A_{r}$ and that $i_*(\Ot_C)$ is in $\A_r$ or in $\A_r[1].$ Moreover, if $i_*(\Ot_C)$ and $j_*(\Ot_C)$ are $\sigma$-stable, then $\phi_1<\phi_3+1=\frac{3}{2},$ because there is a non-zero morphism $i_*(\Ot_C)\la j_*(\Ot_C)[1].$  We will now prove the other direction. We assume that $\phi_1<\frac{3}{2}$ and that $i_*(\Ot_C)$ is not stable. By Remark \ref{HNX1}, the statement follows. The other two statements follow in a similar way.
\edem

\subsection{Pre-stability conditions in \texorpdfstring{$\Theta_{12}$}{TEXT}}
\label{SST12}

  We are going to show that pre-stability conditions $\sigma\in \Theta_{12}$ are given as the ones constructed in Corollary \ref{CP1} or in Lemma \ref{PrestaNG}.

We first characterise  the hearts of the pre-stability conditions in terms of the stability of the skyscraper sheaves. We study pre-stability conditions satisfying that $j_*(\Co(x))$ is stable of phase one. We separate them into two cases when $l_*(\Co(x))$ is not $\sigma$-stable and when it is. If $l_*(\Co(x))$ is not $\sigma$-stable, then the pre-stability conditions are CP-glued. If $l_*(\Co(x))$ is $\sigma$-stable, we obtain stability conditions of the form of Lemma \ref{GNGSC}. We follow closely \cite[Proposition 10.1]{BS8}.

\bl\label{BHNS} Let $\sigma=(Z,\A)$  be a pre-stability condition such that $i_*(\Co(x))$ is $\sigma$-stable, $j_*(\Co(x))$ is $\sigma$-stable of phase one and $l_*(\Co(x))$ is not $\sigma$-stable. 

We assume that $i_*(\Co(x))[n]\in\A.$ If $E= E_1\xrightarrow{\varphi} E_2 \in \A,$ then $H^i(E_1)=0,$ unless \linebreak $i=-n-1,-n.$ Also $H^i(E_2)=0$ unless $i=0$ and  $H^i(C(\varphi))=0,$ unless $i=-n-1,-n-2,0.$ 
\el
\bdem 
First note that  $n\geq 0.$ Indeed, as $l_*(\Co(x))$ is not $\sigma$-stable, it implies that $\phi_0-\phi_2\leq 0.$ As $\phi_2=1$ we have $\phi_0\leq 1,$ also $0<\phi(i_*(\Co(x))[n])=\phi_0+n\leq 1.$ Combining the inequalities above, we get $n\geq 0.$

Let $E\in \A$ be stable with phase $0<\phi(E)<1.$ As $j_*(\Co(x))$ is stable we have  that $\Hom_{\T_C}^{i}(E,j_*(\Co(x)))=0,$ for all $i<0.$ 
By adjointness, we have $$0=\Hom_{\T_C}^{i}(E,j_*(\Co(x)))=\Hom^{i}_{D^b(C)}(C(\varphi),\Co(x)).$$ Moreover, by stability we have $\Hom^{i}(i_*(\Co(x))[n],E)=0,$ for all $i<0.$ By adjointness and Serre duality in $D^b(C),$ we get 
$$0=\Hom_{\T_C}^{i}(i_*(\Co(x))[n],E)=\Hom_{D^b(C)}^{i}(\Co(x)[n],C(\varphi)[-1])=\Hom^{n-i+2}_{D^b(C)}(C(\varphi),\Co(x))^*,$$ i.e.\ $\Hom_{D^b(C)}^{j}(C(\varphi),\Co(x))=0$ for $j>n+2.$ Therefore, we obtain $\Hom^{i}(C(\varphi),\Co(x))=0,$ unless $0\leq i \leq n+2.$ By \textup{\href{https://arxiv.org/pdf/math/9908022.pdf}{\cite[Proposition 5.4]{BFM}}}, it follows that $C(\varphi)$ is isomorphic to a  complex of locally-free sheaves and $H^i(C(\varphi))=0 \textnormal{ unless } -n-2 \leq i\leq 0.$

Similarly, by  stability $\Hom_{\T_C}^{i}(E,i_*(\Co(x))[n])=0,$ for all $i<0.$ By adjointness, we get \linebreak $0=\Hom_{\T_C}^{i}(E,i_*(\Co(x))[n])=\Hom^{i+n}_{D^b(C)}(E_1,\Co(x)),$  i.e.\ $\Hom_{D^b(C)}^{j}(E_1,\Co(x))=0$ for $j<n.$ Since $l_*(\Co(x))$ is not $\sigma$-semistable, it follows from its HN-filtration that $\phi_2\geq \phi_4\geq \phi_0,$ which implies that $l_*(\Co(x))[n] \in \mathcal{P}(0,n+1],$ because $j_*(\Co(x))\in\mathcal{P}(0,1]$ and $i_*(\Co(x))[n]\in\mathcal{P}(0,1].$ As a consequence, we have $\Hom^{i}(l_*(\Co(x))[n],E)=0,$ for all $i<0.$ By adjointness and Serre duality in $D^b(C),$ 
we get $0=\Hom_{T_C}^{i}(l_*(\Co(x))[n],E)=\Hom^{n+1-i}_{D^b(C)}(E_1,\Co(x))^*,$ i.e.\ $\Hom_{D^b(C)}^{j}(E_1,\Co(x))=0$ for $j>n+1.$  Consequently, we obtain  $\Hom^i(E_1,\Co(x))=0$ unless $n\leq i\leq n+1.$ By \linebreak  \textup{\href{https://arxiv.org/pdf/math/9908022.pdf}{\cite[Proposition 5.4]{BFM}}}, we have that $E_1$ is isomorphic to a length two complex of locally-free sheaves and  $H^i(E_1)=0 \textnormal{ unless } i=-n,-n-1.$

We will now prove a similar result for $E_2.$ In this case we use that $j_*(\Co(x))$ is stable of phase one. First of all, by the same reasoning as above, $l_*(\Co(x)) \in \mathcal{P}(-n,1].$ Hence, we get 
$\Hom^{i}(E,l_*(\Co(x)))=0,$ for all $i<0.$ By adjointness, we obtain 
$0=\Hom_{\T_C}^{i}(E,l_*(\Co(x)))=\Hom^{i}_{D^b(C)}(E_2,\Co(x)).$ In addition, since $j_*(\Co(x))$ is stable of phase $1,$ we get  $\Hom_{\T_C}^{i}(j_*(\Co(x)),E)=0$ for $i\leq 0.$ By adjointness and Serre duality in $D^b(C),$ we have 
$0=\Hom_{\T_C}^{i}(j_*(\Co(x)),E)=\Hom^{1-i}_{D^b(C)}(E_2,\Co(x))^*,$
i.e.\ $\Hom_{D^b(C)}^{j}(E_2,\Co(x))=0$ for $j>0.$ Thus, we obtain $\Hom^{i}(E_2,\Co(x))=0,$ unless $i=0.$ By \textup{\href{https://arxiv.org/pdf/math/9908022.pdf}{\cite[Proposition 5.4]{BFM}}}, we have that $E_2$ is isomorphic to a length one complex of locally-free sheaves and  $H^i(E_2)=0 \textnormal{ unless } i=0.$

The triangle $E_1\la E_2\la C(\varphi)$ induces a long exact sequence in cohomology. As a consequence, we have that $H^i(C(\varphi))=0 \textnormal{ unless } -n-2,-n-1,0.$

If $E\ncong j_*(\Co(x))$ is $\sigma$-stable with phase one, then $H^i(E_1)=0$ unless $i=-n-1 \textnormal{ and }$\linebreak  $ H^i(E_2)=0 \textnormal{ for all }i .$ This implies that $H^i(C(\varphi))=0 \textnormal{ unless  } i=-n-2$ and that $E_1$ is torsion free.
\edem

\bl\label{BH}Let $\sigma=(Z,\A)$ be a pre-stability condition such that $l_*(\Co(x)), i_*(\Co(x))$ are  $\sigma$-stable and $j_*(\Co(x))$ is $\sigma$-stable of phase one. Then, for $E=E_1 \xrightarrow{\varphi} E_2 \in \T_C$ we have 
\begin{enumerate}[leftmargin=0.5cm]

\item If $E \in \A,$ then $H^i(E_j)=0,$ unless $i={0,1},$ for $j=1,2.$ Also $H^i(C(\varphi))=0,$ unless $i={-1,0}.$
Moreover, $H^{-1}(C(\varphi)),H^0(E_1)$ are torsion-free.
%\item If $E$ be stable of phase one, then either $E=j_*(\Co(x))$ or $E\in\TCoh(C)$ with $H^0(C(\phi))=0.$ Moreover, we have that $E_1$ and $E_2$ are locally free.  
\item If $E$ is stable of phase one, then either $E=j_*(T),$ where $T\in \Coh(C)$ a torsion sheaf, or $E\in\TCoh(C)$ with $H^0(C(\varphi))=0.$ We have that $E_1$ and $E_2$ are torsion-free.
\item $\TCoh(C)\subseteq \mathcal{P}(0,2]$
\item The pair $\mathcal{T}=\TCoh(C)\cap \mathcal{P}(1,2]\textnormal{ and } \mathcal{F}=\TCoh(C)\cap\mathcal{P}(0,1]$ defines a torsion pair on $\TCoh(C).$ Moreover, the heart $\A$ is the corresponding tilt. 
\end{enumerate}
\el
\bdem

First note that  $i_*(\Co(x))[-1], l_*(\Co(x))[-1]\in \A.$ Indeed, there are non-zero morphisms $i_*(\Co(x))\la j_*(\Co(x)[1]),$ $j_*(\Co(x))\la l_*(\Co(x))$ and $l_*(\Co(x)) \la i_*(\Co(x)).$ By stability, it follows that $1<\phi_4<\phi_0< 2.$

We start by proving part one. Let $E\in \A$ be stable with phase $0<\phi(E)<1.$ As $j_*(\Co(x))$ is stable we have  that $\Hom_{\T_C}^{i}(E,j_*(\Co(x)))=0,$ for all $i<0.$ By adjointness, we have \linebreak $0=\Hom_{\T_C}^{i}(E,j_*(\Co(x)))=\Hom^{i}_{D^b(C)}(C(\varphi),\Co(x)).$ Moreover, by stability we have \linebreak $\Hom^{i}(i_*(\Co(x))[-1],E)=0,$ for all $i<0.$ By adjointness and Serre duality in $D^b(C),$ we get 
$0=\Hom_{T_C}^{i}(i_*(\Co(x))[-1],E)=\Hom^{1-i}_{D^b(C)}(C(\varphi),\Co(x))^*,$ i.e.\ $\Hom_{D^b(C)}^{j}(C(\varphi),\Co(x))=0$ for $j>1.$ Therefore, we obtain $\Hom^{i}(C(\varphi),\Co(x))=0,$ unless $0\leq i \leq 1.$ By \textup{\href{https://arxiv.org/pdf/math/9908022.pdf}{\cite[Proposition 5.4]{BFM}}}, it follows that $C(\varphi)$ is isomorphic to a length two complex of locally-free sheaves and $H^i(C(\varphi))=0 \textnormal{ unless } i={-1,0}.$

Similarly, we obtain   $\Hom^i(E_1,\Co(x))=0$ unless $-1\leq i\leq 0.$ By \textup{\href{https://arxiv.org/pdf/math/9908022.pdf}{\cite[Proposition 5.4]{BFM}}}, we have that $E_1$ is isomorphic to a length two complex of locally-free sheaves and $H^i(E_1)=0 \textnormal{ unless } i={0,1}.$

We prove a similar result for $E_2.$ In this case we use that $j_*(\Co(x))$ is stable of phase one. We obtain $\Hom^{i}(E_2,\Co(x))=0,$ unless  $-1\leq i \leq 0.$  By \textup{\href{https://arxiv.org/pdf/math/9908022.pdf}{\cite[Proposition 5.4]{BFM}}}, we have that $E_2$ is isomorphic to a length two complex of locally-free sheaves and  
$H^i(E_2)=0 \textnormal{ unless } i={0,1}.$ This concludes the proof of the first part. 

We will now proceed to proving the second part. Let $E\in \mathcal{P}(1)$ be  a stable  object, which is not isomorphic to $j_*(T),$ where $T$ is a torsion sheaf.
Since the phase of $E$ is one, we have $H^i(E_1)=0 \textnormal{ unless } i={0}\textnormal{ and } H^i(E_2)=0 \textnormal{ unless } i={0}.$

For the third part, we assume that $E\in \TCoh(C).$ If $F\in\mathcal{P}((2,\infty)),$ then by the first part $F\in D^{\leq -2}\subseteq D^{\leq -1},$ where $(D^{\leq 0},D^{\geq 0})$ is the standard t-structure. Consequently, we have \linebreak $0=\Hom_{\T_C}(D^{\leq 0}, D^{\geq 1 })=\Hom(D^{\leq -1},D^{\geq 0}),$ therefore $\Hom_{\T_C}(F,E)=0.$ Analogously, we have that if $B\in \mathcal{P}(\leq 0),$ then  $B \in D^{\geq 1}.$ As $\Hom_{\T_C}(D^{\leq 0},D^{\geq 1})=0,$ we obtain $\Hom_{\T_C}(E,B)=0.$ It follows that $E\in\mathcal{P}(0,2].$

We will now prove the fourth part. Let $E\in \TCoh(C).$ By the third part of the statement, there is a triangle $A\la E \la B\la A[1],$ where $A\in \mathcal{P}(1,2]$ and $B\in \mathcal{P}(0,1].$ 
After applying $i^*,$ we obtain a long exact sequence of cohomology and by part one we obtain
$H^{-1}(i^*(A))=H^{1}(i^*(B))=0.$
This implies that $i^*(A), i^*(B)\in \Coh(C).$ Analogously, we have $j^!(A), j^!(B)\in \Coh(C)$ and we obtain $A,B\in \TCoh(C).$

Moreover, if $A\in\mathcal{T},$ we have additional information. Indeed, as  we have $H^i(j^*(B))=0$ unless $i=-1,0$ and $H^{i}(j^*(A))=0$ unless $i=-2,-1.$ After taking the long exact sequence we obtain  $H^i(j^*(A))=0$ unless $i=-1.$ As a consequence, if $A=A_1\xrightarrow{g} A_2,$ it follows that $\Coker(g)=0.$
\edem 

By applying Serre duality of doing exactly the same proof under the new hypothesis, we obtain the following results.

\bl\label{BH1}Let $\sigma=(Z,\A)$ be a pre-stability condition such that $j_*(\Co(x)), i_*(\Co(x))$ are  $\sigma$-stable and $i_*(\Co(x))$ is $\sigma$-stable of phase one. Then, for $E=E_1 \xrightarrow{\varphi} E_2 \in \T_C$ we have 
\begin{enumerate}[leftmargin=0.5cm]

\item If $E \in \A,$ then $H^i(E_j)=0,$ unless $i={-1,0},$ for $j=1,2.$ Also $H^i(C(\varphi))=0,$ unless $i={-1,0}.$ We also have that $H^{-1}(E_1), H^{-1}(E_2)$ are torsion-free.

\item If $E$ is stable of phase one, then either $E=i_*(T),$ where $T\in \Coh(C)$ is a torsion sheaf, or $E\in \mathcal{H}_{31}$ with $H^0(E_1)=0.$ Moreover, we have that $E_2$ and $C(\varphi)$ are torsion-free.
\item $\HH_{31}\subseteq \mathcal{P}(0,2]$
\item The pair $\mathcal{T}=\HH_{31}\cap \mathcal{P}(1,2]\textnormal{ and } \mathcal{F}=\HH_{31}\cap\mathcal{P}(0,1]$ defines a torsion pair on $\HH_{31}.$ Moreover, the heart $\A$ is the corresponding tilt. 
\end{enumerate}
\el

\bl\label{BH2}Let $\sigma=(Z,\A)$ be a pre-stability condition such that $i_*(\Co(x)), j_*(\Co(x))$ are  $\sigma$-stable and $l_*(\Co(x))$ is $\sigma$-stable of phase one. Then, for $E=E_1 \xrightarrow{\varphi} E_2 \in \T_C$ we have 
\begin{enumerate}[leftmargin=0.5cm]

\item If $E \in \A,$ then $H^i(E_1)=0,$ unless $i={0,1}$ and $H^i(E_2)=0,$ unless $i={0,-1}.$  We also obtain that $H^i(C(\varphi))=0,$ unless $i={-1,0}.$ We also have that $H^{-1}(C(\varphi)),H^{-1}(E_2)$ are torsion-free.

\item If $E$ is stable of phase one, then either $E=l_*(T),$ where $T\in \Coh(C)$ a torsion sheaf, or $E \in \HH_{23}$  with $H^0(E_2)=0$ and $H^0(E_1)$ torsion-free sheaves.  

\item $\HH_{23}\subseteq \mathcal{P}(0,2]$

\item The pair $\mathcal{T}= \HH_{23}\cap \mathcal{P}(1,2]\textnormal{ and } \mathcal{F}=\HH_{23}\cap\mathcal{P}(0,1]$defines a torsion pair on $\HH_{23}.$ Moreover, the heart $\A$ is the corresponding tilt. 

\end{enumerate}
\el

We will now study the orbit of $\sigma\in \Theta_{12}$ under the action of $\GL$ in order to choose a simpler representative. 

\bp \label{12UTA} For every $\sigma\in \Theta_{12},$ there is $g\in \GL$ such that for $\sigma g=(Z,\A)$ we can find stability conditions $\sigma_1=(Z_1,\Coh^{r}(C))\textnormal{ and }\sigma_2=(Z_\mu,\Coh(C))\in \Stab(C)$ with $\Coh^{r}_1(C)\subseteq \A$ and $r>-1,$ $\Coh_2(C)\subseteq \A,$ $\restr{Z}{\mathcal{C}_1}=Z_1$ and $\restr{Z}{\mathcal{C}_2}=Z_\mu.$
\ep

\bdem

By the stability of $i_*(\Co(x))$ and $i_*(\Ot_C),$ we have $\phi_1<\phi_0<\phi_1 +1.$ Therefore, there is  an orientation preserving transformation $M \colon \R^2\la  \R^2$ satisfying that  
$(A,D) \mapsto  (-1,0)$ and $ (B,C) \mapsto   (0,1), $
where $Z(i_*(\Co(x))=A+Di$ and $Z(i_*(\Ot_C))=B+Ci.$ There is an increasing function $f \colon \R\la \R$ that satisfies  $f(x+1)=f(x)+1$ with $f(1)=\phi_0,$  $f(1/2)=\phi_1.$  We obtain $(T,f)\in\GL$. The stability condition $\sigma^{'}\coloneqq \sigma (T,f)$ satisfies 
$$
Z^{'}(r_1,d_1,0,0)=-d_1+ r_1i   \textnormal{ and }  \Coh_1(C)\subseteq \A^{'},$$ 
where $\sigma^{'}=(Z^{'},\A^{'}).$ Indeed, we have  
 \begin{eqnarray*} \nonumber
i_*(\Co(x)) \in \Pd(\phi_0)&=&\Pd^{'}(1),\\ \nonumber
i_*(\Ot_C) \in \Pd(\phi_1)&=&\Pd^{'}(1/2),\\   \nonumber  
i_*(\Li) \in \Pd(\phi_{\Li})&=&\Pd^{'}(t_{\Li}), \nonumber
\end{eqnarray*}  with $t_{\Li}\in(0,1).$ Therefore, all skyscraper sheaves and line bundles in $\Coh_1(C)$ are in $\Pd^{'}(0,1].$ Since any object in $\Coh_1(C)$ admits a filtration with quotients either isomorphic to skyscraper sheaves or to line bundles, we obtain $\Coh_1(C)\subseteq \A^{'}.$

 %For this reason, the semistable HN-factors with respect to $\sigma_\mu$ of every coherent sheaf in $\Coh_1(C)$ are in $A^{''}=\Pd^{'}(0,1],$ this is enough to show that 

As the action of $\GL$ on the set of pre-stability conditions preserves the stability of the objects and $j_*(\Co(x)),$ $j_*(\Ot_C)$ are $\sigma$-stable, we have $\phi_3<\phi_2<\phi_3 +1$ in $\sigma^{'}.$ Therefore, we can find $g_1=(T_1,f_1)\in\GL,$ such that $\sigma^{''}\coloneqq \sigma' (T_1,f_1)$ satisfies $\restr{Z^{''}}\C_2 = Z_{\mu}$ and $\Coh_2(C)\subseteq \A^{''}.$ 

Moreover, if we consider $\sigma_1=(Z_1,\Coh^r(C)) \in \Stab(C)$ with $f_1(0)=r\in\R$, which under our correspondence with $\GL$ is $(T_1,f_1)\in \GL,$ then $\sigma^{''}$ also satisfies that $\restr{Z^{''}} \C_1= Z_1$ and  $\Coh_1^r(C)\subseteq \A^{''}.$ Indeed, by definition $Z^{''}= T_1^{-1}\circ Z^{'}.$ Consequently,  we can assert that $\restr{Z^{''}} \C_1=T_1^{-1}\circ Z_{\mu}.$

We will now show that $\Coh_1^{r}(C)\subseteq \A^{''},$ where $f_1(0)=r=n+\theta,$ and $n\in \Z$ and $\theta\in [0,1).$ We prove this in several steps. We first show that $\Coh_1(C)[n]\subseteq \mathcal{P}(-1,1].$ Then, we construct a torsion pair of $\Coh_1(C)[n]$ and we compare it with $\Coh(C)[n]=\lin \T_{\theta}[n],\FF_{\theta}[n]\rin,$ which is the torsion pair given in Remark \ref{SCCHR}.

%\bcla  and $r>-1.$
%\ecla

We first show that $i_*(\Co(x)[n])\in \A^{''}.$ Note that $f_1(\phi_0)=1,$ because $$i_*(\Co(x))\in \Pd^{'}(1)=\Pd^{''}(f_1^{-1}(1))=\Pd^{''}(\phi_0)\textnormal{ and } f_1(-n)=\theta.$$
We apply $f_1^{-1}$ to the following inequality $\theta<1\leq \theta+1$ and we obtain 
\begin{equation} \label{SK1}
-n<\phi_0\leq -n+1,
\end{equation} which is equivalent to $i_*(\Co(x)[n])\in \A^{''}.$ By stability $-n<2.$ 

\bcla \label{SB12} $f_1(0)=r>-1.$
\ecla
\bdem  Since $r=n+\theta$ and $ n\geq -1,$ we obtain  $r\geq -1.$ We just need to prove that $r\neq -1.$ Assume that $r=-1.$ Then  if $Z_1(r_1,d_1)=A_1d_1+Br_1+(C_1r_1+D_1d_1)i,$ we get $D_1=0.$ By the stability of $i_*(\Co(x))$ and $j_*(\Co(x))$ in $\sigma^{''},$ we get $\phi_0<2.$ But, by the definition of $Z_1,$ we obtain \linebreak $\phi_0=\phi(i_*(\Co(x))[-1])+1=2,$ which is a contradiction. 
\edem

We will now study line bundles $\Li.$ Note that $f_1(\phi_{\Li})=t_{\Li},$ with $t_{\Li}\in (0,1)$ because $$i_*(\Li)\in \Pd^{'}(t_{\Li})=\Pd^{''}(f^{-1}(t_{\Li}))=\Pd^{''}(\phi_{\Li}).$$ The proof naturally falls into two cases:

\textbf{Case 1}: $\theta< t_{\Li} \leq \theta+1.$ We apply $f_1^{-1}$ to the inequality above and we obtain $-n<\phi_{\Li}\leq -n+1,$ which is equivalent to $i_*(\Li[n])\in \A^{''}.$ 

\textbf{Case 2}: $\theta-1<t_{\Li} \leq \theta.$ We apply $f_1^{-1}$ to the inequality above and we obtain $-n-1<\phi_{\Li}\leq -n,$ which is equivalent to $i_*(\Li[n+1])\in \A^{''}.$ 

\bcla $\Coh_1(C)\subseteq \mathcal{P}(-1,1].$
\ecla
\bdem
We just proved that all line bundles and the skyscraper sheaves are in $\mathcal{P}(-n-1,-n+1],$ therefore $\Coh_1(C)[n]\subseteq \mathcal{P}(-1,1].$
\edem  
We set $\mathcal{T}_1=\Coh_1(C)[n]\cap\mathcal{P}(0,1] \textnormal{ and }\mathcal{F}_1=\Coh_1(C)[n]\cap\mathcal{P}(-1,0].$

\bcla $(\mathcal{T}_1,\mathcal{F}_1)$ is a torsion pair of $\Coh_1(C)[n].$
\ecla
\bdem
The proof falls naturally into two cases:

\textbf{Case 1:} $l_*(\Co(x))$ is not $\sigma^{''}$-stable.

Let $E\in \Coh(C).$ We have $E[n]\in \Coh(C)[n].$
Since $\Coh_1(C)[n]\subseteq \mathcal{P}(-1,1],$ there are objects $T\in \mathcal{P}(0,1]$ and $F\in\mathcal{P}(-1,1]$ such that
$0\la T\la i_*(E)[n]\la F \la 0.$

We obtain long exact sequences in cohomology and by Lemma \ref{BHNS}, we have that \linebreak $H^i(T_1)=H^i(F_1)=0$ unless $i=-n$ and that $T_2=F_2=0.$  As a consequence, we get $\Coh_1(C)[n]=( \T_1, \mathcal{F}_1)$ and $\mathcal{T}_1,\mathcal{F}_1[1]\subseteq \A^{''}.$

\textbf{Case 2:} $l_*(\Co(x))$ is $\sigma^{''}$-stable.

Note that the stability of $l_*(\Co(x))$ implies, as proved in Lemma \ref{BH},  that $n=-1$ and that $l_*(\Co(x))[-1]\in A^{''}.$ 

Let $E\in \Coh(C).$ We have $E[-1]\in \Coh(C)[-1].$
Since $\Coh_1(C)[-1]\in\mathcal{P}(-1,1],$ there are objects $T\in \mathcal{P}(0,1]$ and $F\in\mathcal{P}(-1,1]$ such that
$0\la T\la i_*(E)[-1]\la F \la 0.$ 
We obtain long exact sequences in cohomology. By Lemma \ref{BH}, we have $H^i(T_1)=0$ unless $i=0,1$ and $H^i(F_1)=0$ unless $i=1,2.$ Then, we obtain that $H^i(T_1)=H^i(F_1)=0$ unless $i=1.$
We also get that $H^i(T_2)=0$ unless $i=0,1$ and $H^i(F_2)=0$ unless $i=1,2. $ 
Therefore, we have $T_2=F_2=0.$
\edem

We will now proceed to prove that $\Coh_1^r(C)\subseteq \A^{''}.$ We  show that $\T_1=\T_{\theta}$ and $\mathcal{F}_1=\mathcal{F}_{\theta}.$ It is easy to show that $\T_{\theta}\subseteq \T_1$ and $\mathcal{F}_{\theta}\subseteq \mathcal{F}_1,$ where $\Coh(C)=( \T_{\theta}, \mathcal{F}_{\theta})$ is the torsion pair described in Remark \ref{SCCHR} (up to shift) and $\theta\neq 0.$

It is enough to show this for slope stable torsion free sheaves $E\in \Coh(C)$, such that \linebreak $E[n]\in \Coh(C)[n].$   Let $0\la T\la E[n]\la F \la 0$ be the triangle induced by the torsion pair $(\T_1,\mathcal{F}_1).$ If $T$ and $F$ are non-trivial, then as  $i_*(F[1])\in\mathcal{A}^{''}$ and $i_*(T)\in \mathcal{A}^{''},$  we get that $\Im(Z)(i_*(T))>0,$ which implies
$\mu(T[-n])> -\cot(\theta \pi)$ and $\Im(Z)(i_*(F))\leq 0,$ which implies $\mu(F[-n])\leq -\cot(\theta \pi).$ Since $E$ is $\mu$-stable, we obtain a contradiction.

If $\theta=0,$ then $\T^{\theta}_1=\Coh(C)[n]$ and $\FF^{\theta}_1=0.$ Since $r\in \Z,$ we obtain  $$Z_1(r,d)=A_1d_1+Br_1+(Cr_1)i,$$ i.e.\ $D_1=0.$ This implies that  
$\T_1=\Coh_1(C)[n].$ Therefore $\Coh^{r}_1(C)\subseteq \A^{''}.$
\edem

\brem Let us consider $\sigma\in \Theta_{12}$ and $g'=(T',f')\in \GL.$ By Proposition \ref{12UTA}, there is $g=(T,f)\in \GL$ such that for  $\sigma g=(Z,\A)$ we can find stability conditions $$\sigma_1=(Z_1,\Coh^{r}(C))\textnormal{ and }\sigma_2=(Z_\mu,\Coh(C))\in \Stab(C)$$ with $\Coh^{r}_1(C)\subseteq \A$ and $r>-1,$ $\Coh_2(C)\subseteq \A.$ Note that the proof of the inclusion of the hearts in $\A$ depends only on the stability of $i_*(\Co(x))$ and $j_*(\Co(x))$ for all $x\in C$ and $i_*(\Li)$ and  $j_*(\Li),$ for every line bundle $\Li.$ As a consequence, we obtain that $\sigma g g'=(Z',\A')$ satisfies that $$\Coh_1^{f\circ f'(0)}(C)\subseteq A'\textnormal{ and }\Coh^{f'(0)}_2(C)\subseteq \A'.$$  
\erem

The following lemma gives us some CP-glued stability conditions in $\Theta_{12}.$

\bl \label{GUTA} Let $\sigma=(Z,\A)\in \Theta_{12},$ such that there are stability conditions $$\sigma_1=(Z_1,\Coh^{r}(C))=(T_1,f_1)\textnormal{ and }\sigma_2=(Z_\mu,\Coh(C))\in \Stab(C)$$ with $\Coh^{r}_1(C)\subseteq \A,$ $\Coh_2(C)\subseteq \A,$ and $\restr{Z}\C_1=Z_1$ and $\restr{Z}\C_2=Z_\mu.$  If $f_1(0)=r \geq 0,$ then $\sigma=\gl_{12}(\sigma_1,\sigma_{\mu}).$
\el

\bdem

Since $\sigma_1$ and $\sigma_2$ satisfy gluing conditions i.e\ $f_1(0)\geq f_2(0)=0,$ then there is a pre-stability condition $\sigma_{12}=\gl_{12}(\sigma_1,\sigma_2)$ on $\T_C.$ It follows directly from Proposition \ref{CP1} that $\sigma =\sigma_{12}.$  
\edem
We will now study pre-stability conditions $\sigma=(Z,\A)\in \Theta_{12}.$ For $$\sigma_1=(Z_1,\Coh^{r}(C))\textnormal{ and }\sigma_2=(Z_\mu,\Coh(C))\in \Stab(C)$$ we have $\Coh^{r}_1(C)\subseteq \A,$ $\Coh_2(C)\subseteq \A,$  $\restr{Z}\C_1=Z_1,$ $\restr{Z}\C_2=Z_\mu,$ and $-1<r<0,$ where \linebreak $r=n+\theta=f_1(0)$ and $n\in \Z$ with $\theta\in [0,1).$ Here  $\sigma_1$ is given by $(T_1,f_1)$ under the correspondence with $\GL.$

\brem \label{Cx3stable} $l_*(\Co(x))$ is $\sigma$-stable. Assume that $l_*(\Co(x))$ is not stable. By Remark \ref{HNX1} this implies that $n+\theta=f_1(0)\geq0,$ which contradicts our assumption.
\erem

\subsection{Pre-stability conditions in \texorpdfstring{$\Theta_{23}$}{TEXT} and \texorpdfstring{$\Theta_{31}$}{TEXT}}

We have the analogous statement of Proposition \ref{12UTA} for $\Theta_{23}$ and $\Theta_{31}.$ We leave the proof to the reader, since it follows exactly the same steps as the result for $\Theta_{12}.$
\bp \label{23UTA}For every $\sigma\in \Theta_{23},$ there is a $g\in \GL$ such that for $\sigma g=(Z,\A)$ we can find stability conditions $$\sigma_3=(Z_3,\Coh^{r}(C))\textnormal{ and }\sigma_2=(Z_\mu,\Coh(C))\in \Stab(C)$$ with $\Coh^{r}_3(C)\subseteq \A,$ $\Coh_2(C)\subseteq \A,$ $r<0,$ $\restr{Z}\C_3=Z_3$ and $\restr{Z}\C_2=Z_\mu.$
\ep
\bp \label{31UTA}For every $\sigma\in \Theta_{31},$ there is a $g\in \GL$ such that for $\sigma g=(Z,\A)$ we can find stability conditions $$\sigma_3=(Z_3,\Coh^{r_3}(C))\textnormal{ and }\sigma_1=(Z_1,\Coh^{r_1}(C))\in \Stab(C)$$ with $\Coh^{r_3}_3(C)\subseteq \A,$ $\Coh^{r_1}_1(C)\subseteq \A,$ with  $r_3-r_1>0,$ $\restr{Z}\C_3=Z_3,$ $\restr{Z}D_1=Z_1$ and $M_3-M_1=I.$
\ep

\subsection{Non-gluing pre-stability conditions}

We will now study pre-stability conditions $\sigma$ on $\T_C$ satisfying that $i_*(\Co(x))$ and $l_*(\Co(x))$ are stable, where  $j_*(\Co(x))$ is stable with  phase one and $j_*(\Ot_C)$ has phase 1/2. Note that as proved in Lemma \ref{BH}, we have $l_*(\Co(x))[-1]\in \A.$ We use Lemma \ref{BH} to prove that these stability conditions are precisely given by the pair $(Z_r,\A_r)$ constructed in Lemma \ref{PrestaNG}.

\bl \label{NGUTA} Let $\sigma=(Z,\A)$ be a pre-stability condition on $\T_C$  such that $i_*(\Co(x)),l_*(\Co(x))$ are $\sigma$-stable and the object $j_*(\Co(x))$ and $j_*(\Ot_C)$ are in $\A$ and are also $\sigma$-stable with $$Z([j_*(\Co(x))])=-1 \textnormal{ and } Z([j_*(\Ot_C)])=i.$$ 
Then $\sigma$ is given by the pairs constructed in Lemma \ref{PrestaNG}.
\el

\bdem We prove the statement for $\sigma\in \Theta_{12},$ the other two cases follow analogously.
 By Proposition \ref{12UTA}, there are stability conditions $\sigma_1=(Z_1,\Coh^{r}(C))\textnormal{ and }\sigma_2=(Z_\mu,\Coh(C))\in \Stab(C),$ such that $\Coh^{r}_1(C)\subseteq \A$ and $\Coh_2(C)\subseteq \A,$ with $\restr{Z}\C_1=Z_1$ and $\restr{Z}\C_2=Z_\mu.$ Note that $-1<f_1(0)<0,$ where $\sigma_1=(T_1,f_1)\in \GL.$
First of all we show that $Z$ is given by equation \ref{NGStability}. Our stability function $Z$ is completely determined by $Z_1$ and $Z_\mu,$ therefore it has the following form
$$Z(r_1,d_1,r_2,d_2)=Ad_1+Br_1-d_2+i(Cr_1+Dd_1+r_2).$$ Let $M\coloneqq T_1^{-1}=\begin{bmatrix}
-A & B\\
-D & C\\
\end{bmatrix}.$  Since $\sigma_1$ is a stability condition, we have that $\det(M)>0.$
As $-1<f_1(0)<0,$ then by Lemma \ref{BH}, we obtain $1<\phi_0<2$ and $D<0.$  

We will now show that $\det(M+I)>0.$
If $\Tr(M)\geq 0,$  there is nothing to prove because \linebreak $\det(M+I)=\det(M)+\Tr(M)+1>0.$ Note that if $\Tr(M)<0,$ then $l_*(\Ot_C)$ is $\sigma$-stable.  Since $l_*(\Co(x))$ and $l_*(\Ot_C)$ are stable,  we obtain $\phi_5<\phi_4< \phi_5+1$ which is equivalent to the fact that $\det(M+I)>0.$

We consider the torsion pair $( \T,\FF)=\TCoh(C)$ given in Lemma \ref{BH}. We are going to show that it is equal to  the torsion pair $( \T^{'}, \FF^{'})$ given by Lemma \ref{TPNG}.  It is enough to show that $\T^{'}\subseteq \T$ and $\FF^{'}\subseteq \FF.$

Let us take a torsion-free $\lambda$-semistable object $E=E_1\la E_2\in \T',$ which  by definition satisfies $\lambda(E)> 3/4.$
There is a short exact sequence $0\la T\la E\la F\la 0,$ with $T\in \T$ and \linebreak $F=F_1\la F_2\in \FF.$
By Lemma \ref{BH} and Lemma \ref{TPNG} we have that $i^!(E), i^!(T)\in \Coh(C).$ It follows that $i^!(F)\in \Coh(C).$ 

As $F_1$ is torsion-free and $F_1\twoheadrightarrow F_2,$ then $F\neq 0\la G_2,$ where $G_2$ is a torsion sheaf. We apply Remark \ref{SSS} and we obtain that $3/4<\lambda(E)\leq \lambda(F)\leq 3/4,$ which gives us a contradiction and this implies that $E\in \T.$

Let us take a $\lambda$-semistable torsion-free object $E=E_1\la E_2\in \FF',$ and by definition holds that $\lambda(E)\leq 3/4.$ We start with the case $\lambda(E)< 3/4.$ There is a short exact sequence given by \linebreak $0\la T\la E\la F\la 0,$ with $T\in \T$ and $F\in \FF.$ We get $\lambda(T)\geq 3/4$ and $\lambda(F)\leq 3/4.$ If $T\neq 0,$ then
by $\lambda$-semistability $3/4\leq\lambda(T)\leq \lambda(E)< 3/4,$ which gives us a contradiction and this implies that $E\in F.$

%Moreover, we have that $F\neq j_*(j^!(F_2))$ with $\rk(F_2)= 0,$ indeed, if $F=j_*(j^!(F_2))$, then $T=T_1\la T_2,$ we get $T_1\cong E_1$ and $\rk(E_2)=\rk(T_2),$ it implies that $\lambda(T)=\lambda(E),$ which contradicts the $\lambda$-stability of $E.$ 

%By the $\lambda$-semistability of $E,$ it follows that $\lambda(T)\leq \lambda(E)$ and by \hyperref[SSS]{Lemma \ref*{SSS}}, we obtain $\lambda(T)\leq \lambda(E)<\leq \lambda(F),$ which is contradiction. In this case either $T=0$ of $F=0,$ it implies $E\in \T$ or $E\in \FF.$ 

Let us take a torsion-free $\lambda$-semistable object $E=E_1\xrightarrow{\varphi} E_2\in \FF'$ with  $\lambda(E)=3/4.$ Consider the short exact sequence $0\la T\la E\la F\la 0,$ with $T\in \T$ and $F\in \FF.$  Note that the inequality $3/4\leq \lambda(T)\leq \lambda(E)=3/4$ holds, so that $3/4=\lambda(T).$ If $T=T_1\xrightarrow{\varphi'} T_2,$ we get  $T[-1]\in \mathcal{P}'(1).$ Moreover by Lemma \ref{BH}, we have that $\mathcal{P}'(1)\subseteq \TCoh(C)$ which implies directly that $T=0.$

Note that for all torsion-free objects $E\in \T^{'},$ after using its Harder--Narasimhan filtration, we obtain that $E\in \T.$ Analogously for $E\in \FF^{'}.$ We easily extend the result above to any object in $\T$ and $\FF.$ Consequently $\T=\T^{'}$ and $\FF=\FF^{'}.$
\edem

\brem\label{3SC} Under the assumption of the last proposition, for $\sigma\in \Theta_{12}$ we can define a third stability condition $\sigma_3=((M+I)^{-1},f_3)\in \Stab(C),$ that comes from the data given by $l_*(\Co(x))$ and $l_*(\Ot_C),$ where the integer part of $f_3(0)=r_3$ is $-1.$ Note that by Lemma \ref{THI}, we obtain $\Coh^{r_3}_3(C)\subseteq \A.$
%If $\sigma\in \Theta_{23},$ the third stability condition is given by  $\sigma_1=((I-M_3)^{-1},f_1)\in \Stab(C).$ Finally, if $\sigma\in \Theta_{31},$ the third stability condition is given by  $\sigma_2=\sigma_{\mu}\in \Stab(C).$

\erem

\bp \label{P01}  Let $\sigma$ be a pre-stability condition in $\Theta_{12}.$ There is an element $g\in \GL$ such that $\sigma g$ is given by a CP-glued pre-stability condition or one constructed by tilting in Subsection \ref{CSCTILT}.
\ep

\bdem After applying Proposition \ref{12UTA}, this follows directly from Lemma \ref{GUTA} and Lemma \ref{NGUTA}.
\edem

\subsection{Characterising CP-glued pre-stability conditions in \texorpdfstring{$\Theta_{12}$}{TEXT}}
\label{CSC12}
 
 Let us consider the sets $\Theta_{i}$ consisiting of pre-stability conditions, for $i=1,2$ or $3,$ which are, up to the action of $\GL,$ CP-glued with respect to the semiorthognal decomposition $\lin \C_i,{^{\perp}}\C_i\rin.$ Note that by Proposition \ref{CP1}, it follows that the set $\Theta_1\subseteq \Theta_{12}.$

The aim of this section is to characterise the sets $\Theta_i$ inside $\Theta_{12}.$  

\textbf{Condition (*)}: If $\sigma\in \Theta_{12}$ satisfies  that 
there are stability conditions $\sigma_1=(\Coh^{r}_1(C),Z_1)$ and $\sigma_2=(\Coh_2(C),Z_{\mu})\in \Stab(C)$ with $\Coh^{r}_1(C)\subseteq \A,$ $\Coh(C)_2\subseteq \A,$ and $\restr{Z}\C_1=Z_1$ and $\restr{Z}\C_2=Z_{\mu}.$ Let $f_1(0)=n+\theta,$ with $n\in\Z$ and $\theta\in [0,1),$ we say that it satisfies condition (*). Let $\sigma\in \Theta_{12}.$ By Proposition \ref{12UTA}, we have that there is $g\in \GL,$ such that $\sigma g$ satisfies condition (*).

Under condition (*), we define $f_{12}(\sigma)(x)=f_1(x)-x.$

From now on we assume that $\sigma$ satisfies condition (*) and $\sigma_1=(T,f_1).$ If $f_{12}(\sigma)(0)\geq 0$ by Proposition \ref{GUTA}, we already know that $\sigma\in \Theta_1.$ Now we assume $f_{12}(\sigma)(0)<0.$ We would like to understand when $\sigma\in \Theta_{12}$ is in $\Theta_i,$ for $i=1,2,3.$ 

Recall that by Remark \ref{3SC} there is also $\sigma_3=(Z_3,\Coh^{r_3}(C))$ with $\restr{Z}\C_3=Z_3$ and  $\Coh^{r_3}_3(C)\subseteq \A,$ and we can also define $f_{23}(\sigma)(x)=x-f_3(x) \textnormal{ and } f_{31}(\sigma)(x)=f_3(x)-f_1(x).$ 

%The following lemma characterizes $\Theta_1$ inside $\Theta_{12}.$ 

\bl \label{f12} Let $\sigma$ be as in (*), such that $f_{12}(\sigma)(0)<0.$ There is a $t\in \R$ such that $f_{12}(\sigma)(t)=0$ if and only if $\sigma\in \Theta_1.$
\el

\bdem Assume that there is $t\in\R$ such that $f_{12}(\sigma)(t)=0.$ Since $f_2(t)=t,$ this implies that $f_1(t)=t.$ Let $g=(K_{t\pi},f_{t\pi})\in\GL.$  Since $f_1\circ f_{t\pi}(0)=f_1(t)=t$ and $f_2\circ f_{t\pi}(0)=f_2(t)=t,$ as a consequence $f_{12}(\sigma g)(0)=0$ and it follows directly  from Lemma \ref{GBCP} that $\sigma g=\gl_{12}(\sigma_{1}g,\sigma_{2}g).$ 

We now assume that $\sigma \in \Theta_1,$ then there is a $g\in \GL$ such that $\sigma  g$ satisfies the gluing conditions for $\sigma_1 g$ and $\sigma_2 g$. Without losing generality, we take $g=(K_{l\pi}, f_{l\pi})\in \GL,$ with $l\in \R $  Thus,  $f_1(l)=(f_1\circ f_{l\pi})(0)\geq (f_2\circ f_{l\pi})(0)=f_{l\pi}(0)=l.$ 
By hypothesis $f_{12}(\sigma)(0)<0,$ but also  $f_{12}(\sigma)(l)\geq 0.$ Since $f_{12}(\sigma)(x)$ is a continuous function, there is a $t\in \R$ that satisfies $f_{12}(\sigma)(t)=0,$ as we desired. 
\edem 

\brem\label{FIXEDPOINT}[Fixed point] Since in our case $f_{12}(\sigma)(x)=f_1(x)-x,$ we obtain that  $f_{12}(\sigma )(x)=0$ if and only if  there is $x\in\R$ such that $f_1(x)=x.$ After using that the restrictions of $f_1$ and $T$ to $S^1$ agree, the search of points $x\in\R$ such that $f_1(x)=x$ reduces to the study of the eigenvectors of $T.$

\erem

Let $M\coloneqq T^{-1}=\begin{bmatrix}
-A& B\\
-D& C\\
\end{bmatrix}.$

We will now study the characteristic polynomial $p(x)=x^2-\Tr(M)x+\det(M)$ of $M.$ It plays an important role in determining if $\sigma\in \Theta_i$ for some $i=1,2,3.$ The sign of the discriminant \linebreak $\Delta(M)=\Tr(M)^2-4\det(M)$ of  $p(x)$ tells us about the existence of real eigenvalues.

\bp \label{12_1} If $f_{12}(\sigma)(0)<1,$ the discriminant of $\Delta(M)=\Tr(M)^2-4\det(M)$ is non-negative and the eigenvalues are positive, then there is $g\in \GL$ such that $\sigma g=(Z,\TCoh^l(C)),$ with $l\in\R$ and $-1<l\leq 1.$ Moreover $\sigma\in \Theta_1.$
\ep

\bdem Since $\Delta(M) \geq 0,$ it guarantees the existence of real eigenvalues. By hypothesis the eigenvalues are positive. The same follows for $T_1.$ Let $\lambda\in \R$ be an eigenvalue of $T_1$ and $v\in \R^2$ its corresponding eigenvector. We obtain $T_1 v= \lambda v.$  Let us consider the polar coordinates of $v=(m \cos(\phi), m \sin(\phi))$ with $\phi \in (-\pi, \pi ]$ and $m\in \R_{>0}.$ 

We claim now that $\sigma g\in \Theta_1,$ where  $g=(K_{\phi}, f_{\phi})\in \GL.$ First of all we consider \linebreak $\sigma_1 g =(T_1 K_{\phi}, f_1\circ f_{\phi}).$ By the correspondence between $f_1\circ f_{\phi}$ and $T_1K_{\phi}$ over $S^1$ and the fact that $v$ is an eigenvector, we obtain that $(f_1\circ f_{\phi})(0)=\phi/\pi.$ Now we consider $\sigma_2 g=(K_{\phi},f_{\phi}).$ As a consequence, $f_{12}(\sigma g)=0.$ By Lemma \ref{CP1}, it is clear that $\sigma g=(Z,\TCoh^\phi(C))=\gl_{12}(\sigma_1 g,\sigma_2 g).$
\edem

\brem Note that from Lemma \ref{12_1} and Remark \ref{FIXEDPOINT}, we obtain that if $\sigma\in \Theta_{12}$ satisfies (*), then $\sigma\in \Theta_1$ if and only if the eigenvalues of $M$ are positive. Moreover, by Lemma \ref{f12}, if $\Delta(M)<0,$  then $\sigma$ can never be in $\Theta_1.$ 
\erem

Under the assumption that the discriminant $\Delta(M)$ of $M$ is non-negative and that the eigenvalues $\lambda_1, \lambda_2$ of $M$ are both negative we prove the following lemmas. 

\brem Under the assumption that the discriminant $\Delta(M)$ of $M$ is non-negative and that the eigenvalues are  negative, then if $\det(M+I)>0,$  there are just two options, either  $\lambda_1, \lambda_2< -1\textnormal{ or }\lambda_1, \lambda_2> -1.$ 
\erem

The pre-stability conditions in $\Theta_2$ or $\Theta_3$ satisfy that $l_*(\Li)$ is stable for all line bundles $\Li\in \Coh(C).$ Hence, in order to determine which $\sigma\in \Theta_{12}$ are in $\Theta_2$ or $\Theta_3$ we need to study the stability of $l_*(\Li).$

\bl \label{EQ1} Let $\sigma=(Z,\A)$ be 
a pre-stability condition as in (*). Let $\Li\in \Coh(C)$ be a line bundle with $\deg(\Li)=d\leq \frac{-C}{D}$. The object $l_*(\Li)$ is stable in $\sigma$ if and only if the inequality $-Dd^2-(A+C)d-B>0$
holds.
\el 
\bdem 
Since $\deg(\Li)=d\leq \frac{-C}{D}$ we have $i_*(\Li)\in \A.$ 
By considering the triangle 
$$\dtri{j_{*}(\Li)}{l_*(\Li)}{i_{*}(\Li)}$$ and the correspondence between slope and phase the result follows. 
\edem

\brem\label{EQ2}  \begin{enumerate}[leftmargin=0.5cm]
\item  The discriminant of the quadratic equation $Dd^2+(A+C)d+B$ is given by $(A+C)^2-4BD=\Delta(M),$ i.e.\  it has  the same discriminant as $p(x).$ 
\item By Remark \ref{JHX}, if $\Li$ is a line bundle and $i_*(\Li)[-1]\in \A,$ then $l_*(\Li)$ is stable. 
\end{enumerate}
\erem

%\bdem Assume that $l_*(\Li)$ is not stable, therefore by  we have that $\phi_{\Li}^1-\phi_{\Li}^3\leq 0.$ As $0<\phi_{\Li}^3\leq 1,$ then it implies $\phi_{\Li}^1 \leq 1,$ which contradicts our hypothesis. Indeed, if $i_*(\Li)[-1]\in \A,$ then $1<\phi_{\Li}^1\leq 2.$
%\edem

%The morphism $i_*(\Li)[-1]\la j_*(\Li)$ is a morphism in $\A^{''}.$ Since they are both stable, we have $\phi(i_*(\Li)[-1])\leq\phi(j_*(\Li)).$ Let us assume that $l_*(\Li)$ is not stable, then by ....we have

%Under the assumption that the \hyperref[L3]{inequality \ref*{L3}} holds for all $d\in\Z,$ we prove the following lemmas. 

%
\bl \label{EQ3} If $\lambda_1,\lambda_2<0,$ then inequality  $-Dd^2-(A+C)d-B>0$ holds for all $d\leq -\cot(\theta\pi).$ 
\el
\bdem Let us consider the polynomial $q(x)=Dx^2+(A+C)x+B.$ As the discriminant of $q(x)$ is also $\Delta(M)\geq 0,$ we get that $q(x)$ has real roots $\mu_1,\mu_2\in \R.$ We assume that $\mu_1\leq \mu_2.$
It is enough to conduct the proof for $ \mu_1\geq \frac{-C}{D}.$ Indeed, if $\frac{-C}{D}\leq \mu_1$ and $d< \frac{-C}{D}$ then $q(x)<0.$ 
As $0>2\lambda_1=(-A+C)-\sqrt{\Delta(M)},$ we obtain 
\begin{eqnarray}\nonumber
\frac{C}{-D}&<&\frac{-(A+C)-\sqrt{\Delta(M)}}{2D}=\mu_1,
\end{eqnarray}
as we wanted to prove. 
\edem

\bc \label{LBS}If $\lambda_1,\lambda_2<0,$ then $l_*(\Li)$ is $\sigma$-stable for all line bundles $\Li.$ Moreover, $\sigma$ is in  $\Theta_{23}$ and $\Theta_{31}.$ 
\ec
\bdem This follows directly from Lemma \ref{EQ1}, Lemma \ref{EQ2} and Lemma \ref{EQ3}.
\edem
%\bdem  
%Since $f_1(0)<0.$ By \hyperref[Cx3stable]{Lemma  \ref*{Cx3stable}} , we have that $l_*(\Co(x))$ is stable and $n=-1,$  as a  consequence $l_*(\Co(x))[-1]\in \A.$  Motivated by this fact we define $$r_3=\arg((C+1) + D\I)/\pi \in (-1,0).$$  

%Here the argument is considered between $[-\pi, \pi)$
%We show that $\Coh_3^{r_3}(C)\subseteq \A.$ Indeed, start by proving that $l_*(\Li)$ is stable for a line bundle $\Li\in\Coh(C).$ If $d=\deg(\Li)$ and $d\leq -cot(\theta \pi),$ then $l_*(\Li)$ is stable  by \hyperref[EQ1]{Lemma \ref*{EQ1}} and our hypothesis. If $d\leq -cot(\theta \pi),$ then  $l_*(\Li)$ is stable  by \hyperref[EQ2]{Lemma \ref*{EQ2}}. Then $\sigma\in \Theta_{23},$ by \hyperref[23UTA]{Proposition \ref*{23UTA}}, we obtain  $\Coh_3^{r_3}(C)\subseteq \A$ and $\restr{Z}{D_3}=Z_3.$
%\edem 

%Under the assumption that the  \hyperref[L3]{inequality \ref*{L3}} holds for all $d\in \Z$, we use the last lemma to study when the two additional cases of gluing appearing in $\Theta_{12}.$

\bl \label{f31}\label{f23} \begin{enumerate}[leftmargin=0.5cm]
    \item There is $t\in \R$ such that $f_{31}(\sigma)(t)=1$ if and only if $\sigma\in \Theta_3.$
    \item There is a $t\in \R$ such that $f_{23}(\sigma)(t)=1$ if and only if $\sigma\in \Theta_2.$
\end{enumerate}
 \el
\bdem The proof goes along the lines of Lemma \ref{f12}.
\edem

\bp \label{12_3}\label{12_2} \begin{enumerate}[leftmargin=0.5cm]
    \item If $\lambda_1, \lambda_2< -1,$ then $\sigma\in \Theta_3.$
    \item If $0>\lambda_1, \lambda_2> -1,$ then $\sigma\in \Theta_2.$
\end{enumerate} 
\ep
\bdem The proof goes exactly as the one of Lemma \ref{12_1}.
\edem

\brem \begin{enumerate}[leftmargin=0.5cm]
\item Note that from Proposition \ref{12_3} and Lemma \ref{f31}, we obtain that if $\sigma\in \Theta_{12}$ satisfies (*), then $\sigma\in \Theta_3$ if and only if the eigenvalues $\lambda_1,\lambda_2$ of $M$ are $< -1.$
\item Note that from Proposition \ref{12_2} and Lemma \ref{f23}, we obtain that if $\sigma\in \Theta_{12}$ satisfies (*), then $\sigma\in \Theta_2$ if and only if the eigenvalues $\lambda_1,\lambda_2$ of $M$  satisfy  $0>\lambda_1,\lambda_2>-1.$
\end{enumerate}
\erem

\brem Note that if $\sigma\in \Theta_{12}$ satisfies (*), then either $\sigma\in \Theta_1$ or $\sigma\in \Theta_{12}\cap \Theta_{23}.$ Indeed, if the eigenvalues of $M$ are positive, we have that $\sigma\in \Theta_{1}.$ If the eigenvalues are negative, this follows from Corollary \ref{LBS}. If $\Delta(M)<0,$ it follows from Lemma \ref{EQ1}.
This remark is useful when studying $\Theta_{23}$(or $\Theta_{31}$) because either $\sigma\in \Theta_{2}$ or $\sigma\in \Theta_{12}\cap\Theta_{23},$ thus we can use our classification inside $\Theta_{12}$ to understand $\Theta_{23}.$
\erem

\brem Note that if $\sigma\in \Theta_{23},$ then by Proposition \ref{23UTA}, up to the $\GL$-action we can find stability conditions $\sigma_3=(Z_3,\Coh^{r}(C))=(T_3,f_3)\textnormal{ and }\sigma_2=(Z_\mu,\Coh(C))\in \Stab(C)$ with $\Coh^{r}_3(C)\subseteq \A,$ $\Coh_2(C)\subseteq \A,$ $r<0,$ $\restr{Z}\C_3=Z_3$ and $\restr{Z}\C_2=Z_\mu.$ If $i_*(\Co(x))$ is $\sigma$-stable for $x\in C$ and $\Delta(M_3)<0,$ with $M_3\coloneqq T_3^{-1},$ by following the steps of Lemma \ref{EQ1} we can show that $i_*(\Li)$ is $\sigma$-stable for all line bundles $\Li.$ As a consequence, by Lemma \ref{NGUTA}, we have that $\sigma$ is given, up to the $\GL$-action, by a pre-stability condition constructed in Lemma \ref{PrestaNG}. Moreover, $M_3=I+M,$ the $M$ appearing in Lemma \ref{PrestaNG} and  $\Delta(M)=\Delta(M_3)<0.$ Analogously for the case $\sigma\in \Theta_{31}.$
\erem

%We finally obtain the analogous version of \hyperref[TWOOPTIONS]{Proposition \ref*{TWOOPTIONS}} for $\Theta_{23}.$ If $\sigma \in \Theta_{23}$ either is given by a CP-gluing stability condition up to the action of by one constructed 

\bt \label{TE02} For all pre-stability conditions $\sigma$ on $\T_C,$ we have that 
$$\sigma\in \Theta_1\cup \Theta_2 \cup \Theta_3\cup \Gamma,$$ where
$\Gamma$ is the set of pre-stability conditions, which up to the $\GL$-action is given by Lemma \ref{PrestaNG} with $\Delta(M)< 0.$
Moreover, note that $\Gamma\subseteq \Theta_{ij},$ where $ij=12,23,31.$
\et

\bdem By Theorem \ref{TEO1}, we have that $\sigma\in \Theta_{12}\cup \Theta_{23} \cup \Theta_{31}.$ By Serre duality, it suffices to check for $\sigma\in \Theta_{12}.$ By Proposition \ref{12UTA}, we have that up to the action of $\GL,$ there are \linebreak  $\sigma_1=(\Coh^{r}_1(C),Z_1)=(T,f)$ and $\sigma_2=(\Coh_2(C),Z_{\mu})\in \Stab(C)$ with  $\Coh^{r}_1(C)\subseteq \A,$ \linebreak $\Coh(C)_2\subseteq \A,$ and $\restr{Z}\C_1=Z_1,$ $\restr{Z}\C_2=Z_{\mu}$ with $f(0)\geq -1.$ We reduce to the study of these pre-stability conditions, because the sets $\Theta_i$ and $\Gamma,$ for $i=1,2,3,$ are defined up to the $\GL$-action.
If $f(0)\geq 0,$ then by Lemma \ref{GUTA} $\sigma\in \Theta_1.$ If $-1<f(0)<0,$ we classify these pre-stability conditions in two sets if $\Delta(M)\geq 0$ or $\Delta(M)<0.$ If $\Delta(M)<0$ then $\sigma\in \Gamma.$ If $\Delta(M)\geq 0,$ with positive eigenvalues then by Lemma \ref{12_1} we get $\sigma\in \Theta_1.$ If the eigenvalues are smaller than $-1,$ then by Lemma \ref{12_3} we have that $\sigma\in \Theta_3$ and if the eigenvalues are between $0$ and $-1$ by Lemma \ref{12_2}, we obtain that $\sigma\in \Theta_2$.
\edem

\brem Note that by Lemma \ref{f12} and Lemma \ref{f31} we have that $$\Theta_i\cap \Theta_j =\emptyset \textnormal{ and } \Theta_i\cap \Gamma=\emptyset \textnormal{ for } i,j\in \{1,2,3\}.$$
\erem

%\begin{figure}[h!]
 % \includegraphics[scale=0.7]{Stab(T_C)_1.png}
  %\caption{$\Stab(\T_C)$.}
 % \label{fig:Stab}
%\end{figure}

\section{Support property for \texorpdfstring{$\T_{C}$}{TEXT}, with \texorpdfstring{$g(C)\geq 1$}{TEXT}}
\label{SPC}

In this section, we prove the support property for the already constructed pre-stability conditions. We start by studying CP-glued pre-stability conditions $\sigma=\gl_{12}(\sigma_1,\sigma_2)$ on  $\T_{C}=\lin \C_1, \C_2\rin,$ with $\sigma_i=(Z_i,\A_i)\in \Stab(D^b(C)),$ for $i=1,2.$ By Theorem \ref{GLC} there is always $g=(T,f)\in \GL,$ such that $\sigma_1=\sigma_2 g.$ Note that the gluing condition states precisely that $f(0)\geq 0$. Therefore, the proof of the support property falls naturally into the cases $f(0)\geq 1$ and $1> f(0)\geq 0.$
More precisely we reduce the proof of the support property to three cases. Namely, in Lemma \ref{SP2G} we prove it for pre-stability conditions satisfying that $f(0)\geq 1,$ in Lemma \ref{SP1} we prove it for CP-glued pre-stability condition with $1>f(0)\geq 0$ and $\Delta(T^{-1})>0.$ In Lemma \ref{SPCPNDG} for CP-glued pre-stability conditions with $1>f(0)\geq 0$ and negative discriminant. Finally, in Subsection \ref{SPNGND} we prove it for $\sigma\in \Gamma.$ We give some inequalities to study the $\sigma$-semistable objects in our pre-stability conditions. These inequalities follow closely the steps of \cite[Section\ 3]{GP2}. In Subsection \ref{SPNGND} we prove the support property under the assumption that $g(C)=1,$ for non-gluing pre-stability conditions with $\Delta(M)<0.$ We conjecture that for $g(C)>1$ the same result holds. 

\subsection{Support property for CP-glued pre-stability conditions}
\label{SPSOC}

We start by proving the support property for pre-stability conditions with a stronger orthogonality condition.

\bl \label{SP2G} If the pair $\sigma=(Z,\A)=\gl_{12}(\sigma_1,\sigma_2)$ is a CP-glued pre-stability condition satisfying that $\Hom_{\T_C}^{\leq 1}(i_*\A_1,j_*\A_2)=0$ or equivalently  $f(0)\geq 1,$ then it satisfies the support property and in consequence it is a Bridgeland stability condition.    
\el

\bdem  
We use the following notation $Z_1([E])=Z_1([i^*(E)])$ and $Z_2([E])=Z_2([j^{!}(E)]).$

First note that we can linearly  extend $Z$ and the homomorphism induced by the exact functors $i^*,j^!$ to $\N(\T_{C})\otimes \R\cong \R^4.$ We define the quadratic form  $$Q\colon \N(\T_{\C})\otimes \R\la \R \textnormal{ as } Q(v)=\Im(Z_1(v))\Im(Z_2(v))+\Re(Z_1(v))\Re(Z_2(v)),$$ where $\Re(\alpha)$ and $\Im(\alpha)$ are the real and the imaginary parts of $\alpha\in \R^4$ respectively.   

By the linearity of $Z$, it is clear that $Q$ is a quadratic form. 
We first show that it is negative definite on $\Ker(Z)=\{v\in \R^4\mid \Re(Z_1(v))=-\Re(Z_2(v))\textnormal{ and } \Im(Z_1(v))=-\Im(Z_2(v))\}.$
Indeed, we have that $Q(v)=-\Im(Z_1(v))^2-\Re(Z_1(v))^2\leq 0,$
for $v\in \Ker(Z).$ Note that $Q(v)=0,$  implies that $Z_i(v)=0,$ for $i=1,2.$ Then $v=0$ as $\rank(\N(D^b(C))=2.$

Let $E= E_1\xrightarrow{\varphi} E_2$ be a $\sigma$-semistable object. By \cite[Lemma\ A.6]{BMS3} it is enough to show that \linebreak $Q(E)\geq 0$ for $\sigma$-stable objects. Note that $\varphi=0.$ Indeed, as $\sigma$ is a \linebreak CP-glued pre-stability condition, by the definition of  $\gl_{12}(\A_1,\A_2),$ we have that $E_1\in \A_1$ and $E_2\in \A_2.$ As  $\varphi\in \Hom_{D^b(C)}(E_1,E_2)=\Hom_{\T_C}(i_*(E_1),j_*(E_2)[1])$ and by hypothesis \linebreak $\Hom_{\T_C}(i_*(E_1),j_*(E_2)[1])=0,$ we obtain that $\varphi=0.$

Since $E$ is $\sigma$-stable and by the definition of a CP-glued heart, we have $i_*(E_1), j_*(E_2)\in \A.$ It follows from the short exact sequence $0\la j_*(E_2)\la E\la i_*(E_1)\la 0$  $E_1=0$ or $E_2=0.$ As a consequence, $Q([E])=0.$  
\edem

We will now follow the steps of \cite[Section\ 3]{GP2} in order to study the $\sigma$-semistable objects in the CP-glued pre-stability conditions $\sigma=(Z,\A)=\gl_{12}(\sigma_1,\sigma_2),$ on $\T_{C}$ such that $\sigma_1=\sigma_2 g$ where $g=(T,f)$ in $\GL,$ with $f(0)=0.$ Note that by the same argument as in Proposition \ref{12_1}, these are precisely, up to the $\GL$-action, the pre-stability conditions with $1>f(0)\geq 0$ and positive discriminant. 

We have that $g=(T,f)\in \GL$ has the following form $T^{-1}=\begin{bmatrix} -A&B\\
0&C
\end{bmatrix},$
with $C>0$ and $\det(T)>0.$ This implies that $\A_1=\A_2=\HH.$
We obtain that $\sigma_1=(Z_1,\HH)$ and $\sigma_2=(Z_2,\HH).$  
By definition we have \begin{eqnarray} \label{SameheartG}
\Re(Z_1(w))&=&-A\Re(Z_2(w)))+B\Im(Z_2(w)),\\ \nonumber
\Im(Z_1(w)) &=& C\Im(Z_2(w)) \textnormal{ where } w\in \mathbb{Z}^4.\nonumber
\end{eqnarray} 
We use the following notation 
\begin{alignat}{6} \label{NotCPPD}
 d_2&=&-\Re(Z_2([j^!(E]))\textnormal{, } d_1&=&-\Re(Z_2([i^*(E)])),\\ \nonumber
 r_2&=&\Im(Z_2([j^!(E)]))\textnormal{ and } r_1&=&\Im(Z_2([i^*(E)])).
 \end{alignat}

Therefore, we have $Z_2([j^!(E)])=-d_2+i r_2 \textnormal{ and }  Z_1([i^*(E)])=Ad_1+Br_1+i(Cr_1).$
Note that for every $E\in \B,$ we have that $r_1,r_2\geq 0.$ If $r_1,r_2\neq 0,$ we define $\mu_{\sigma}(E)=\frac{-Ad_1-Br_1+d_2}{Cr_1+r_2}.$

We obtain the following inequalities. 

\bl \label{IBGG} If $E=E_1\xrightarrow{\varphi} E_2\in \A$ with $r_1,r_2\neq 0$ is a $\sigma$-semistable object and $[\varphi]\neq 0,$ then $$-B\geq (A+C)\mu_{\sigma}(E)$$.
\el 
\bdem As  $\HH$ is an abelian category, we can compute $\Ker(\varphi),\Coker(\varphi)\in \HH$ and by the definition of $\A,$ we obtain morphisms $E\twoheadrightarrow i_*(\Img(\varphi))$ and $j_*(\Img(\varphi))\lai E$ in $\A.$
%\begin{equation}
   % \xymatrix{ E_1 \ar@{->>}[r] \ar[d]_{\varphi}& \Img(\varphi) \ar[d] \\
% E_2 \ar[r] & 0\ar[r] }\textnormal{ and }
%\xymatrix{ 0 \ar[r] \ar[d] & E_1  \ar[d]_{\varphi}\\
 %\Img(\varphi)\ar@{^{(}->}[r] & E_2.}
%\end{equation}

Let $Z_2(\Img(\varphi))=-d''_1+r''_1i.$ Note that $r''_1\neq 0.$ Indeed, if $r''_1= 0,$ then $\phi(j_*(\Img(\varphi)))=1,$ and by the $\sigma$-semistability $\phi(j_*(\Img(\varphi)))\leq \phi(E)\leq 1,$ therefore $\Img(Z(E))=0,$ which implies that $r_1,r_2=0,$ which is a contradiction.

From these short exact sequences and the correspondence between slope and phase, we obtain that
$$\frac {d''_1}{r''_1}\leq \mu_{\sigma}(E)\textnormal{ and } \mu_{\sigma}(E)\leq \frac{-Ad''_1-Br''_1}{Cr''_1}.$$
It follows that $\mu_{\sigma}(E) \leq  \frac{-A}{C}\frac{d''_1}{r''_1}-\frac{B}{C},$ then $ \mu_{\sigma}(E) \leq  \frac{-A}{C}\mu_{\sigma}(E) -\frac{B}{C}.$ As we have that $-A,C>0,$ therefore we obtain $\mu_{\sigma}(E) (A+C)\leq -B.$ 
\edem

\be Consider $\sigma_{\alpha}$ as in Example \ref{EXGP94} on $\T_{C}.$  As $\sigma_{\alpha}=\gl_{12}(\sigma_\mu g, \sigma_{\mu})$
with $g=(T,f),$ where $T^{-1}=\begin{bmatrix}1& \alpha\\ 0&1 \end{bmatrix}$
and  $f(0)=0.$  If there are $\sigma_{\alpha}$-semistable objects, as $A+C=0$ and $-B=\alpha,$ then $\alpha\geq 0$ (see \cite[Proposition 3.13]{GP2}).
\ee

\bl \label{IIBGG} If $E=E_1\xrightarrow{\varphi} E_2\in \A$ with $r_1, r_2\neq 0$ and $[\varphi]\neq 0$ is $\sigma$-semistable, then  $$-Ad_1+d_2-\mu_{\sigma}(E)(r_2-Ar_1)\leq 0.$$
\el
\bdem
We clearly have 
$$\mu_{\sigma}(E)  =\frac{-Ad_1+d_2-Br_1}{Cr_1+r_2} = \frac{-Ad_1+d_2}{Cr_1+r_2}+\frac{-Br_1}{Cr_1+r_2}.$$  
By Lemma \ref{IBGG}, we obtain 
\begin{eqnarray} \nonumber 
\mu_{\sigma}(E) & \geq & \frac{-Ad_1+d_2}{Cr_1+r_2} +(A+C)\frac{\mu_{\sigma}(E)r_1}{Cr_1+r_2}, \\ \nonumber
\mu_{\sigma}(E) (r_2-Ar_1) & \geq & -Ad_1+d_2. \nonumber  
\end{eqnarray} \edem 

\bl \label{IE1GG}If $E=E_1\xrightarrow{\varphi} E_2\in \A$ with $r_1, r_2\neq 0$ with $[\varphi]\neq 0$ is a $\sigma$-semistable object, then $$(r_2-r_1)\mu_{\sigma}(E)\leq d_2-d_1$$ 
\el
\bdem 
If $\Ker(\varphi)=0,$ we have the exact sequence 
$0\la l_*(E_1)\la E \la j_*(\Coker(\varphi))\la 0.$
By the $\sigma$-semistability of $E,$ we directly obtain $\mu_{\sigma}(E)\leq \frac{d_2-d_1}{r_2-r_1}.$

If $\Coker(\varphi)=0,$ we have the exact sequence 
$0\la i_*(\Ker(\varphi))\la E\la l_*(E_2)\la 0.$
We obtain $$\mu_{\sigma}(E)\geq \frac{-A(d_1-d_2)-B(r_1-r_2)}{C(r_1-r_2)}=\frac{-A(d_1-d_2)}{C(r_1-r_2)} +\frac{-B}{C}.$$

By Lemma \ref{IBGG}, we have $\mu_{\sigma}(E)\geq \mu_{\sigma}(E)+\mu_{\sigma}(E)\frac{A}{C}-\frac{A(d_1-d_2)}{C(r_1-r_2)}.$
Since we assumed $-A,C\geq 0,$ we obtain directly that $$\mu_{\sigma}(E)\geq \frac{d_1-d_2}{r_1-r_2}.$$ As in this case $r_1-r_2\geq 0,$ because $\rk(\Ker(\varphi))=r_1-r_2$ and $\Ker(\varphi)\in \mathcal{H},$
which implies $$\mu_{\sigma}(E)(r_2-r_1)\leq (d_2-d_1). $$

Now we assume that $\Coker(\varphi)\neq 0$ and $\Ker(\varphi)\neq 0.$ 
From the morphism $i_*(\Ker(\varphi))\hookrightarrow E$ 
 follows that $\frac{-Ad'_1-Br'_1}{Cr'_1}\leq \mu_{\sigma}(E),$ where $Z_2(\Ker(\varphi))=-d'_1+r'_1i.$  Note that $r'_1\neq 0.$ If $r'_1=0,$ then $\phi(i_*(\Ker(\varphi)))=1$ and by $\sigma$-semistability we have that $1=\phi(i_*(\Ker(\varphi)))\leq \phi(E)\leq 1,$ which implies that $r_1,r_2=0,$ and this gives us a contradiction.
 
Since $\frac{-Ad'_1}{Cr'_1} \leq \mu_{\sigma}(E)+\frac{B}{C},$
 we obtain
 \begin{equation} \label{E2GG}
 -Ad'_1r_1+\mu_{\sigma}(E)r_2r'_1-r'_1(-Ad_1+d_2)\leq 0.
 \end{equation}
Note that if $r_2-r''_1\neq 0,$ then from the morphism $E\twoheadrightarrow j_*(\Coker(\varphi))$
follows that 
\begin{equation} \label{E3GG}
\mu_{\sigma}(E)\leq \frac{d_2-d''_1}{r_2-r''_1},\\
\end{equation}
where $Z_2(\Img(\phi))=-d''_1+i r''_1,$ as before. 

As $d''_1=d_1-d'_1$ and $r''_1=r_1-r'_1,$ after multiplying  
\eqref{E3GG} by $-Ar_1$ and adding \eqref{E2GG} we obtain
$$Ar_1(d_2-d_1)-r'_1(-Ad_1+d_2)+\mu_{\sigma}(E)(-Ar_1(r_2-r_1)+r'_1(-Ar_1+r_2))\leq 0.$$
By Lemma \ref{IIBGG}, we have $d_1-d_2+\mu_{\sigma}(E)(r_2-r_1)\leq 0.$
As a consequence, we get $$\mu_{\sigma}(E)(r_2-r_1)\leq d_2-d_1.$$

If $r_2-r''_1=0,$ then we obtain  that $d_2-d''_1\geq 0.$
Due to the fact that $d''_1=d_1-d'_1$ and $r''_1=r_1-r'_1,$ we get that $$-Ar_1(d_1-d'_1-d_2)\leq 0.$$
Adding the last inequality with \eqref{E2GG}, we have 
\begin{eqnarray} \nonumber
\mu_{\sigma}(E)r_2r'_1-r'_1(-Ad_1+d_2)-Ar_1(d_1-d_2)& \leq & 0.
\end{eqnarray}
By Lemma \ref{IIBGG}, we obtain 
\begin{eqnarray}\nonumber
\mu_{\sigma}(E)r_2r'_1-r'_1(\mu_{\sigma}(E)(r_2-Ar_1)-Ar_1(d_1-d_2))&\leq & 0.
\end{eqnarray}
As $-A>0,$ we conclude that $\mu_{\sigma}(E)(r_2-r_1)\leq (d_2-d_1).$
\edem

\bl \label{COTASSPGG}Let $\sigma=(Z,\A)$ be as above with $Z=Ad_1+Br_1-d_2+i(Cr_1+r_2).$ If there is a $\sigma$-semistable object with $r_2>r_1>0$  then $ Cy+Ax\leq -B,$ where $x=\frac{d_1}{r_1}$ and $y=\frac{d_2}{r_2}.$ Moreover, if $[\varphi]\neq 0$ then  $y-x\geq 0,$   and $$-B\in [Cy+Ax, (Ax+Cy+y-x-\frac{r_1}{r_2}x(A+C)) \frac{r_2}{r_2-r_1}].$$
\el
\bdem Let $E=E_1\xrightarrow{\varphi} E_2$ be a $\sigma$-semistable object. Let us consider the following short exact sequence 
$0\la j_*(E_2)\la E\la i_*(E_1)\la 0.$
By the semistability of $E$ and the correspondence between slope and phase it follows that 
\begin{equation} \label{12SESGG}
\frac{d_2}{r_2}\leq \mu_{\sigma}(E)\leq  \frac{-Ad_1-Br_1}{Cr_1}.
\end{equation}
Thus it is implied that $Cr_1d_2+Ad_1r_2\leq-\frac{B}{C}r_1r_2,$
as $r_1,r_2,C>0.$ We obtain  $Cy+Ax\leq -B.$

As $\varphi\neq 0,$ by Lemma \ref{IE1GG}, we have the following inequality $\frac{-Ad_1-Br_1+d_2}{Cr_1+r_2}\leq \frac{d_2-d_1}{r_2-r_1}.$
From this equation we obtain 
$-B\leq \frac{Cd_2r_1+Ad_1r_2-(C+A)d_1r_1+d_2r_1-d_1r_2}{(r_2-r_1)r_1,}.$
After dividing both sides of the inequality into $r_1r_2$ we get
$$-B\in [Cy+Ax, (Ax+Cy+y-x-\frac{r_1}{r_2}x(A+C)) \frac{r_2}{r_2-r_1}].$$
From the equation 
$Cy+Ax \leq-B,$
we obtain 
\begin{eqnarray} \nonumber
Cy+Ax &\leq & (Ax+Cy+y-x-\frac{r_1}{r_2}x(-A-C)) \frac{r_2}{r_2-r_1},\\ \nonumber
0&\leq & (y-x)(r_2+Cr_1). \nonumber
\end{eqnarray}
As $r_2+Cr_1>0,$ then  we obtain $y-x\geq 0$.
\edem 

We will now write Lemma \ref{COTASSPGG} explicitly in terms of the $r_1,r_2,d_1,d_2$ from the Grothendieck group.

\bl \label{COTASSPGG2C}Let $\sigma=(Z,\TCoh^{\theta}(C))=\gl_{12}(\sigma_1,\sigma_2)$ be a pre-stability condition on $\T_C$ with \linebreak  $\sigma_i=(Z_i,\Coh^{\theta}(C))$ and $\theta\in [0,1).$ If $ Z_2(d_2,r_2)=A_2d_2+Br_2+i(C_2r_2+D_2d_2)$ and $\sigma_1=\sigma_2g,$ were $g=(T,f)\in \GL$ and $$T^{-1}=\begin{bmatrix}
-A& B\\0 & C
\end{bmatrix}$$ with $C>0$ and there is a $\sigma$-semistable object $E\in \TCoh^{\theta}(C)$ with $[E]=[r_1,d_1,r_2,d_2]$ and  $C_2r_2+D_2d_2>C_2r_1+D_2d_1>0,$ then $Cy+Ax\leq -B,$ where $x=\frac{-A_2d_1-B_2r_1}{C_2r_1+D_2d_1}$ and $y=\frac{-A_2d_2-Br_2}{C_2r_2+D_2d_2}.$ Moreover, if $[\varphi]\neq 0$ then  $y-x\geq 0,$   and $$-B\in [Cy+Ax, (Ax+Cy+y-x-\frac{C_2r_1+D_2d_1}{C_2r_2+D_2d_2}y(A+C)) \frac{C_2r_2+D_2d_2}{C_2r_2+D_2d_2-C_2r_1-D_2d_1}].$$
\el

\subsection{Duality and semistability on \texorpdfstring{$\T_{C}$}{TEXT}}

The aim of this subsection is to prove an analogous statement to Lemma \ref{COTASSPGG} for $\sigma$-stable objects $E$ with $r_1>r_2>0,$ where $\sigma$ is a pre-stability condition on $\T_C.$ We follow closely \cite{GJ1} to define the derived dual in this case and to show that it induces an anti-autoequivalence. 

%$F\colon QCoh(X)\la QCoh(X)$ be a left exact functor such that $$RF\colon D^b_{\Coh(X)}(QCoh(X)) \la D^b_{\Coh(X)}(QCoh(X))$$ is an equivalence, then $$RF_{Q}\colon D_{Q_{Coh(X),n}}^b(Q_{QCoh(X),n})\la D_{Q_{Coh(X),n}}^b(Q_{QCoh(X),n})$$ is also an equivalence.

%We now consider the following functor
%\begin{eqnarray} \nonumber
%\mathcal{H}om(,\Ot_X)\colon \Coh(X) &\la & \Coh(X)\\ \nonumber
%E_1\xrightarrow{\varphi}E_2 & \la & \HHom(E_1, \Ot_X)\xrightarrow{\HHom(\varphi)} \HHom(E_2,\Ot_X).
%\end{eqnarray}

Now we consider the following functor $\DD=\RHHom(-,\Ot_C)\colon D^{b}(C)^{\rm{op}}\la D^{b}(C).$ By \linebreak \textup{\href{https://arxiv.org/pdf/1708.00522.pdf}{\cite[Section\ 2.3]{K4}}}, we have that $\DD$ is an equivalence of categories and $\DD^2=\Id.$

Let us consider the functor 
\begin{eqnarray}\nonumber
\HHom(-,\Ot_X)_{\TCoh(C)}\colon \TCoh(C)&\la & \TCoh(C) \\ \nonumber
E_1\xrightarrow{\varphi} E_2 &\mapsto & \HHom(E_2,\Ot_X)\xrightarrow{\HHom(\varphi)} \HHom(E_1,\Ot_X).
\end{eqnarray}

\bl Every bounded object $E$ in $\Kom^b(\TCoh(C))$  is quasi-isomorphic to a complex \linebreak $F\in \Kom^b(\TCoh(C))$ of locally free sheaves in $\T_C.$
\el
\bdem It is enough to show that for $E=E_1\xrightarrow{\varphi}E_2\in \TCoh(C)$ there is $F\in \TCoh(C)$ of locally free sheaves with $F \sur E.$ There are locally-free sheaves $F_i$ and surjective morphisms $F_i\xrightarrow{\pi_i} E_i$ for $i=1,2.$ Let us consider the triple $G=F_1\la F_1\oplus F_2$ as the inclusion of the first component. Note that $F_1\oplus F_2$ is also locally free. We have a surjective morphism $G\la E$ in $\TCoh(C)$.
\edem

For any bounded acyclic object $E\in \Kom^b(\TCoh(C))$ of locally free sheaves, we have that $\HHom(E,\Ot_C)_{\TCoh(C)}$ is acyclic. As a consequence, the class of objects with locally free sheaves as components in $\TCoh(C)$ is adapted for the left exact functor $\HHom(-,\Ot_C)_{\TCoh(C)}$ and by \linebreak \cite[Rem.\ 2.51]{HUYFM}, we can define the right derived functor $\RHHom(-,\Ot_C)_{\T_C}\colon \T_C^{\rm{op}}\la \T_C.$

%For the definition of adapted class see \cite[Rem. 2.51]{HUYFM}  

%\bt \textup{\cite[Theorem.\ 24]{GJ1}} Let $F\colon \A \la \B$ be a left exact functor, if \linebreak $RF\colon D^b(\A)\la D^b(C)$ is an equivalence of categories  then $$R(F_{Q_n})\colon D^b(Q_{\A,n})\la D^b(Q_{\B,n})$$ is also an equivalence of categories. 
%\et

\bp \label{DUAL12} The functor $\mathbb{D}_{1}\coloneqq\RHHom(-,\Ot_C)_{\TCoh(C)}\colon \T_C^{\rm{op}}\la \T_C$ is an equivalence of categories. 
\ep
\bdem Note that $\T_C^{\rm{op}}=\lin \C_2^{\rm{op}}, \C_1^{\rm{op}}\rin.$ Moreover, we have that $\restr{\mathbb{D}_1}{\C_2^{\rm{op}}}\colon \C_2^{\rm{op}} \la \C_1$ and $\restr{\mathbb{D}_1}{\C_1^{\rm{op}}}\colon \C_1^{\rm{op}} \la \C_2$ are equivalences of categories. As $\mathbb{D}_1(\C_3^{\rm{op}})\subseteq C_3,$ by \href{https://arxiv.org/pdf/0812.2519.pdf}{\cite[Lemma 1.3]{Ka1}}
we obtain that $\mathbb{D}_{1}$ is an equivalence of categories. 
\edem

%\brem \label{DUAL12} Moreover, if we consider $$\mathbb{D}_{1}\coloneqq\RHHom(-,\Ot_X)_{Q_1}\colon \T_{X} \la \T_{X}$$ as a contravariant functor we obtain that $$\restr{\mathbb{D}_{1}}{D_1}\colon D_1\la D_2,$$ $$\restr{\mathbb{D}_{1}}{D_2}\colon D_2\la D_1$$ are giving precisely by the duality $\DD$ on $D^b(X).$ Indeed, as the locally free resolution of a object $E\in D_i$ is also in $D_i,$ for $i=1,2.$ 
%\erem
%\bl $\mathbb{D}_1^{2}(E)=E,$ for every $E=E_1\xrightarrow{\varphi} E_2\in \TCoh(C).$
%\el
%\bdem Let us consider the triangle $$j_*j^!(E)\la E\la i_*(i^*(E))\la j_*(j^!(E))[1]$$ induced by the semiorthogonal decomposition $\T_C=\lin D_1,D_2\rin.$ We apply the functor $\DD_1,$ by \hyperref[DUAL12]{Remark \ref*{DUAL12}}  we obtain the triangle 
%$$E_2\la \DD_1^2(E) \la  E_1 \xrightarrow{t_{\DD^2(E)}} E_2[1].$$
%By adjointness the morphism $t_{\DD^2(E)}$ is precisely the induced by $E_1\xrightarrow{[\varphi]} E_2.$ By \hyperref[TL12]{Lemma \ref*{TL12}}, 
%we have that $\DD^{2}(E)\cong E.$ 
%\edem

We first study the functor $\DD$ on $D^b(C),$ where $C$ is a curve with $g(C)\geq 1.$ By \cite[Corollary 3.40]{HUYFM}, we obtain that $\DD(\Co(x))=\Co(x)[-1]$ and for a locally free sheaf $E\in \Coh(C),$ we obtain that $\DD(E)=E^{\vee}=\HHom(E,\Ot_C)$ wich satisfies that $\deg(E^{\vee})=-\deg(E)\textnormal{ and } \rank(E)=\rank(E^{\vee}).$
Moreover, in general it follows that if $E\in D^b(C)$ and $[E]=[r,d]$ then $[\DD(E)]=[r,-d].$

Let us consider the following torsion pair
\brem \label{TPSPG1} Let $\sigma=(Z,\A) \in \Stab(D^b(C)).$ We have the following torsion pair on $\A$.
$$\T=\{E\in \A\mid \phi_{\sigma}(E)=1 \},$$
$$\FF=\{E\in \A\mid \textnormal{Its } \sigma \textnormal{-semistable factors } F_i \textnormal{ satisfy } \phi_{\sigma}(F_i)<1\}.$$
\erem

If $\Coh^{\theta}(C)\subseteq \T_C$ with $\theta\in [0,1),$
we rename the torsion pair given in Remark \ref{TPSPG1} because \linebreak $\Coh^{\theta}(C)=(T_{\sigma}^{\theta},F_{\sigma}^{\theta}).$ Recall that by Remark \ref{SCCHR}, we also have that $\Coh^{\theta}(C)=(\FF_{\theta}[1],\T_{\theta})$ where $\T_{\theta}=\mathcal{P}_{\mu}(\theta, 1]$ and $\FF_{\theta}=\mathcal{P}_{\mu}(0,\theta]$ in  $\Coh(C).$

\bd Let $\sigma=(Z(r,d)=Ad+Br+i(Cr+Dd),\Coh^{\theta}(C))=\Stab(C).$ We define $\sigma^{\vee}\coloneqq (Z(r,d)=Ad-Br+i(Cr-Dd),\Coh^{1-\theta}(C)[-1])\in \Stab(C).$
\ed

\brem \label{SLOPEDUAL}\label{DTFP}

\begin{enumerate}[leftmargin=0.5cm]
\item If we have $\sigma_1=(Z_1,\Coh^{\theta}(C))$ and $\sigma_2=(Z_2,\Coh^{\theta}(C))$ with $0\leq\theta<1,$ then $F_{\sigma_1}^{\theta}=F^{\theta}_{\sigma_2}.$ Therefore, for any stability condition $\sigma$ with heart $\Coh^{\theta}(C),$ we denote $F^{\theta}_{\sigma}$ by $F^{\theta}.$
\item Note that $\DD(\Coh^{\theta}(C))\neq \Coh^{1-\theta}(C)[-1].$
\item $\sigma^{\vee}_{\mu}=\sigma_{\mu}.$
\item $\mu_{\sigma}(E)=-\mu_{{\sigma}^{\vee}}(\DD(E)).$
\item Let $\sigma=(Z,\Coh^{\theta}(C)),$ then $\DD(F^{\theta})=F^{1-\theta}[-1],$ where $F^{1-\theta}[-1]$ is given with respect to $\sigma^{\vee}.$
\end{enumerate} 
\erem

%\bdem  We start by studying the image of $\FF_{\theta}$ and $\T_{\theta}$ under $\DD.$ Let $E$ be a $\sigma$-stable object with $E\in \T_{\theta}\subseteq\Coh(C)$ then as $E$ is locally free, we obtain that $\Hom(E,\Ot_C)=E^{\vee}\in \FF_{1-\theta}.$ 
%If $E$ is a $\sigma$-stable object with $E\in \FF_{\theta}\subseteq\Coh(C)$ and $\frac{\deg(E)}{\rk(E)}<-\cot(\theta\pi)$ we obtain that $\Hom(E[1],\Ot_C)=E^{\vee}[-1] \in \T_{1-\theta}[-1].$ 
%By the Harder--Narasimhan filtration, we can extend this result to any object in $\FF_{\theta}$ that satisfies that for its HN-factors $A_i$ holds that $\frac{\deg(E)}{\rk(E)}<-\cot(\theta\pi),$  which and $\T_{\theta},$ which is also in $F^{\theta}.$ 
%Let $E\in F^{\theta}\subseteq \Coh^{\theta}(C).$ We have that a short exact sequence in $\Coh^{\theta}(C)$ given by $0\la F[1]\la E \la T\la 0 ,$ with $F\in \FF_{\theta}$ and $T\in \T_{\theta}.$
%As $E\in F^{\theta}$ and $F^{\theta}$ is closed under subobjects, we obtain that $F[1]\in F^{\theta}$ and by definition $T\in F^{\theta}.$
%After applying $\DD,$ we obtain a triangle $$\DD(T)\la \DD(E)\la \DD(F[1])\la \DD(T)[1],$$ as shown above $\DD(T)\in \FF_{1-\theta}$ and $\DD(F[1])\in \T_{1-\theta}[-1],$ as a consequence $\DD(E)\in \Coh^{1-\theta}(C)[-1].$ Moreover, we get that $\DD(E)\in F^{1-\theta}[-1]$ as $\DD(T),\DD(F)\in F^{1-\theta}[-1].$ 
%\edem

\be If $\sigma_{\mu}=(Z_{\mu},\Coh(C)),$ the torsion pair $T^{0}$ and $F^{0}$ is given precisely by the subcategory of torsion sheaves and torsion-free sheaves respectively. We have that $\DD(\Coh(C))=\lin \Co(x)[-1], \Li \rin,$ for any line bundle $\Li$ in $\Coh(C)$ and any point  $x\in C.$ Note that $\lin \Co(x)[-1], \Li \rin$ is a heart of a bounded t-structure on $D^b(C)$ which does not admit a stability function. Moreover, we trivially have that $\DD(F^{0})=F^{0}.$
\ee

We now study $\DD_1$ on $\T_C.$
Let $\sigma=(Z,\TCoh^{\theta}(C))=\gl_{12}(\sigma_1,\sigma_2)$ be a pre-stability on $\T_C$ with $\sigma_1=\sigma_2 g$ where $g=(T,f)\in \GL$ with $f(0)=0$ and $T^{-1}=\begin{bmatrix} -A & B \\ 0 & C \end{bmatrix}.$

%\bdem If $E_1=0$ or $E_2=0,$ then it follows from the definition. We assume that $E_1\neq 0$ and $E_2\neq 0,$ therefore by the stability of $E,$ we get that $\varphi\neq 0.$ As $E_2\in \Coh^{\theta}(C),$ then we have a short exact sequence $0\la T_2 \la E_2\la F_2 \la 0$ in $\Coh^{\theta}(C)$  with $T_2\in T^{\theta}$ and $F_2\in F^{\theta}.$ If $T_2\neq 0,$ then we have a subobject of the form $0\la T_2$ of $E$ with $\phi(0\la T_2)=1,$ which contradicts the stability of $E.$ Therefore, we have $E_2\in F^{\theta}.$
%For $E_1$ we also have a short exact sequence $$0\la T_1 \la E_1\la F_1 \la 0$$ with $T_1\in T^{\theta}$ and $F_1\in F^{\theta}.$ If $T_1\neq 0,$ we have a subobject of $E$ of the form $T_1\la 0,$ as there are no morphisms from $T^{\theta}$ to $F^{\theta}.$  It contradicts the stability of $E$ and $E_1\in F^{\theta}.$
%\edem 

\bd Let $\sigma=(Z,\TCoh^{\theta}(C))=\gl_{12}(\sigma_1,\sigma_2),$ we define the dual stability condition on $\T_C$ as $\sigma^{*}=\gl_{12}(\sigma_2^{\vee},\sigma_1^{\vee})=(Z^{*},\TCoh^{1-\theta}(C)[-1]).$
\ed

\brem \label{STFCG} \label{SLOPEDUALT} 
\begin{enumerate}[leftmargin=0.5cm]
\item Let $\sigma=(Z,\TCoh^{\theta}(C))$ be a pre-stability condition as above. If $E=E_1\xrightarrow{\varphi} E_2$ is $\sigma$-stable with $\phi_{\sigma}(E)<1,$ i.e.\ $\Im(Z(E))\neq 0,$ then $E_1,E_2\in F^{\theta}.$ 
\item Let $E=E_1\la E_2\in \TCoh^{\theta}(C),$ with $E_1,E_2\in F^{\theta}$
then $\mu_{\sigma}(E)=-\mu_{\sigma*}(\DD_1(E)).$ Moreover, we also have that 
$\sigma^*=\gl_{12}(\sigma_2^{\vee},\sigma^{\vee}_{2}gN_{\frac{2B}{C}}).$                                                                                                                                                                                                                                                                                                                                                                                                                                                                                                                                                                                                                                                                                                                                                                                                                                                                                                                                                                                                                                                                                                                                                                                                                                                                                                                                                                                                                                                                                                                                                                                                                                                                                                                                                                                                                                                                                                                                                                                                                                                                                                                                                                                                                                                                                                                                                                                                                                                                                                                                                                                                                                                                                                                                                                                                                                                                                                                                                                                                                                                                                                                                                                                                                                                                                                                                                                                                                                                                                                                                                                                                                                                                                                                                                                                                                                                                                                
\item If $\sigma=\gl_{12}(\sigma_{\mu}g,\sigma_{\mu}),$ with  $g=(T,f)\in \GL,$ where $$T^{-1}=\begin{bmatrix}
1&-\alpha\\
0& 1
\end{bmatrix}$$ and $f(0)=0,$
then $\sigma^*=\gl_{12}(\sigma_{\mu},\sigma_{\mu}g')$ with $g'=g N_{-2\alpha}.$ Note that $g'=g^{-1}.$
By the $\GL$-action, an object $E\in \TCoh(C)$ is $\sigma^{*}$-stable if and only if it is $\sigma^{*}g$-stable, where \linebreak $\sigma^{*}g=\gl_{12}(\sigma_{\mu}g,\sigma_{\mu}).$
\end{enumerate}
\erem

%g N_{\frac{2B}{C}}

\bl \label{SSDCorrespondence} Let us consider $\sigma=(Z,\TCoh^{\theta}(C))=\gl_{12}(\sigma_1,\sigma_2).$ An object $E\in \TCoh^{\theta}(C) $ is $\sigma$-stable with $\phi_{\sigma}(E)<1$ if and only if $\mathbb{D}_1(E)$ is $\sigma^*$-stable and $\phi_{\sigma^*}(\DD_1(E))<1.$
\el

\bdem 

%We start by writing explicitly the stability functions
%$$Z_2(r_2,d_2)=A_2d_2+B_2r_2+i(C_2r_2+D_2r_2),$$ where
%\det \begin{bmatrix}-A_2 & B_2 \\-D_2 & C_2 \end{bmatrix}>0 and $-\frac{C_2}{D_2}=-\cot(\pi\theta)$ and $D_2>0.$ Therefore, we have that $$Z_1(r_1,d_1)=A(A_2d_2+B_2r_2)+B(Cr_2+D_2d_2)+i(C(C_2r_2+D_2r_2))$$
%and by definition $$Z_2^{\vee}(r_1,d_1)=Z_1'(r_1,d_1)=A_2d_1-Br_1+i(C_2r_1-D_2d_1)$$ and if $\sigma_2^{\vee}N_{\frac{2B}{C}}=(Z_'2,\Coh^{1-\theta}(C)),$
%then $$Z'_2(r_2,d_2)=A(A_2d_2-B_2r_2)-B(Cr_2-D_2d_2)+i(C(C_2r_2-D_2r_2)).$$

%We now write the slopes explicitly, 
%Note that by....., we have that $$\mu_{\sigma}(E)=\mu_{sigma'}$$
%$$\mu_{\sigma}(r_1,d_1,r_2,d_2)=\frac{-A_2d_2-B_2r_2-A(A_2d_1+B_2r_1)-B(D_2d_1+C_2r_1)}{C(D_2d_1+C_2r_1)+ D_2d_2+C_2r_2},$$
%if $\Im(Z(r_1,d_1,r_2,d_2))\neq 0.$

%Note that for any $F=F_1\la F_2\in\TCoh^{\theta}(C)$ with $F_1,F_2\in F^{\theta},$ we have that $\mu_{\sigma}(F)<\infty$ and $$\mu_{\sigma}(F)=-\mu_{\sigma}(\DD(F)).$$

As in Remark \ref{STFCG} we have that $E_1,E_2\in F^{\theta}.$ By Remark \ref{DTFP}, it follows that $\DD(E_1),\DD(E_2)$ are in $F^{1-\theta}[-1].$ Then $\DD_1(E)=\DD(E_2)\la \DD(E_1)\in \TCoh(C)^{1-\theta}[-1]\textnormal{ and }\phi_{\sigma*}(\DD_1(E))<1.$
Let us consider  $Q=Q_1\la Q_2\in \TCoh^{1-\theta}(C)[-1],$ satisfying $Q_1, Q_2 \in F^{1-\theta}[-1]$ and $$0\la G\la \DD_1(E)\la Q \la 0$$ a short exact sequence in  $\TCoh^{1-\theta}(C)[-1],$ where $G=G_1\la G_2.$  Note that it is enough to prove that $\phi_{\sigma*}(\DD_1(E))<\phi_{\sigma*}(Q)$ to show the stability of $\DD_1(E).$ Indeed, if $P=P_1\la P_2$  is an arbitrary quotient of $\DD_1(E),$ either there is an object $Q=Q_1\la Q_2$ with  $Q_1,Q_2\in F^{1-\theta}[-1]$ with $P\sur Q$ and $\phi_{\sigma*}(Q)<\phi_{\sigma^*}(P)$ or $\phi_{\sigma^*}(P)=1$ and we trivially obtain that $\phi_{\sigma^*}(\DD_1(E))<\phi_{\sigma^*}(P).$
We also have that $G_1,G_2\in F^{1-\theta}[-1]$ as $F^{1-\theta}[-1]$ is closed under subobjects. Moreover, the duality gives us a correspondence between short exact sequences on $F^{1-\theta}[-1]$ and $F^{\theta}.$
We obtain a short exact sequence in $\TCoh^{\theta}(C)$ given by $0\la \DD_1(Q)\la E\la \DD_1(G)\la 0.$ By the stability of $E$ we obtain that
$\mu_{\sigma}(\DD_1(Q))<\mu_{\sigma}(E)$ and it follows if and only if $\mu_{\sigma*}(\DD_1(E))<\mu_{\sigma*}(Q).$ As a consequence, $\DD_1(E)$ is $\sigma^{*}$-stable. 
%$$\frac{A_2d_2-B_2r_2+A(A_2d_1-B_2r_1)+B(D_2d_1+C_2r_1)}{C(D_2d_1+C_2r_1)+ D_2d_1+C_2r_1}<\frac{-A_2(d'_1)-B_2r'_1-A(A_2d'_2-B_2r'_2)+B(-D_2d'_2+C_2r'_2)}{C(-D_2d_2'+C_2r'_2)-D_2d_1'+C_2r'_1}$$
%To show the stability of $\DD(E),$ it is enough to show that Note that it is enough to sh
\edem

Let $\sigma=\gl_{12}(\sigma_2g,\sigma_2)=(Z,\TCoh^{\theta}(C))$ as above and $E=E_1\xrightarrow{\varphi} E_2$ a $\sigma$-stable object with $\phi(E)<1,$ $\varphi\neq 0$ and $0<C_2r_2+Dd_2<C_2r_1+D_2d_1,$ where $[E]=(r_1,d_1,r_2,d_2).$ Note that we cannot apply Lemma \ref{COTASSPGG2C} directly.

Note that $\sigma*=\gl_{12}(\sigma_2^{\vee},\sigma_2^{\vee}gN_{\frac{2B}{C}}).$ After applying the $\GL$-action and by Lemma \ref{SSDCorrespondence},  we have that $E$ is $\sigma$-stable if and only if $\DD_1(E)$ is $\sigma'=\sigma^*g'=(\sigma_2^{\vee}g',\sigma^{\vee}_2)$-stable, where $$g'=N_{\frac{-2B}{C}}g^{-1}=(T',f')\in \GL$$
and $T'^{-1}=\frac{1}{\det(T^{-1})}\begin{bmatrix} C & B\\0&-A
\end{bmatrix}.$

If $\sigma_2=(Z_2(r,d)=A_2d_2+B_2r_2+i(C_2r_2+D_2d_2),\Coh^{\theta}(C)),$
then $$\sigma^{\vee}_2=(Z_2^{\vee}(r,d)=A_2d_2-B_2r_2+i(C_2r_2-D_2d_2),\Coh^{1-\theta}[-1](C)).$$

If we apply Lemma \ref{COTASSPGG2C} to $\sigma',$ we have that if $F=F_1\xrightarrow{\varphi_{F}} F_2$ is $\sigma$-stable with $$C_2\rk(F_2)-D_2\deg(F_2)>C_2\rk(F_1)-D_2\deg(F_1)$$ and $[\varphi_{F}]\neq 0,$
then $$\frac{-A_2\deg(F_1)+B_2\rk(F_1)}{C_2\rk(F_1)-D_2\deg(F_1)}\leq \frac{-A_2\deg(F_2)+B_2\rk(F_2)}{C_2\rk(F_2)-D_2\deg(F_2)}.$$

Note that $[\DD_1(E)]=[r_2,-d_2,r_1,-d_1].$ As a consequence, we obtain that 
$\DD_1(E)$ satisfies $$C_2\rk(\DD(E_1))-D_2\deg(\DD(E_1))>C_2\rk(\DD(E_2))-D_2\deg(\DD(E_2)),$$ because  $C_2r_2+Dd_2<C_2r_1+D_2r_1.$ We can now apply Lemma \ref{COTASSPGG2} to $\sigma'$ and we obtain 
$$\frac{A_2d_2+B_2r_2}{C_2r_2+D_2d_2}\leq \frac{A_2d_1+B_2r_1}{C_2r_1+D_2d_1},$$ which is precisely 
$$\frac{-A_2d_1-B_2r_1}{C_2r_1+D_2d_1}\leq \frac{-A_2d_2-B_2r_2}{C_2r_1+D_2d_2}.$$

Moreover, we obtain the following result, which gives us the necessary conditions to have \linebreak $\sigma$-stable objects. 

\bl \label{COTASSPGG2}Let $\sigma=(Z,\TCoh^{\theta}(C))=\gl_{12}(\sigma,\sigma_2)$ be a pre-stability condition on $\T_C$  with $Z_2(d_2,r_2)=A_2d_2+Br_2+i(C_2r_2+D_2r_2)$ and $\sigma_1=\sigma_2g,$ were $g=(T,f)\in \GL$ and $$T^{-1}=\begin{bmatrix}
-A& B\\ 0 & C
\end{bmatrix}.$$
If there is a $\sigma$-semistable object with $C_2r_1+D_2d_1>C_2r_2+D_2d_2>0$ then $Cy+Ax\leq -B,$ where $x=\frac{-A_2d_1-B_2r_1}{C_2r_1+D_2d_1}$ and $y=\frac{-A_2d_2-Br_2}{C_2r_2+D_2d_2}.$ 

Moreover, if $[\varphi]\neq 0$ then  $y-x\geq 0$ and $$-B\in [Cy+Ax, Ax+Cy+y-x-\frac{C_2r_2+D_2d_2}{C_2r_1+D_2d_1}y(A+C) \frac{C_2r_1+D_2d_1}{C_2r_1+D_2d_1-C_2r_2-D_2d_2}].$$
\el

%\bl \label{COTASSPGG2}Let $\sigma=(Z,\A)$ a pre-stability condition on $\T_C$ 

%with $Z=Ad_1+Br_1-d_2+i(Cr_1+r_2),$ if there is a $\sigma$-semistable object with $r_1>r_2>0$  then $$Cy+Ax\leq -B,$$ where $x=\frac{d_1}{r_1}$ and $y=\frac{d_2}{r_2}.$ Moreover, if $\varphi\neq 0$ then  $$y-x\geq 0,$$   and $$-B\in [Cy+Ax, Ax+Cy+y-x-\frac{r_2}{r_1}y(A+C) \frac{r_1}{r_1-r_2}].$$
%\el

\bc\label{TRIALPHA} If $\sigma_{\alpha}$ as defined in Example  \ref{EXGP94} on $\T_C$ and $E$  is a $\alpha$-stable object  with $[E]=(r_1.d_1,r_2,d_2)$ and $r_1\neq r_2,$ then $\alpha\in [\frac{d_2}{r_2}-\frac{d_1}{r_1},(\frac{d_2}{r_2}-\frac{d_1}{r_1})(1+\frac{r_1+r_2}{|r_2-r_1|})].$ 
\ec

%In \cite{GP2} and \cite{AS3}, it was shown that quasi-projective moduli spaces  $\mathcal{M}_{\alpha}(r_1,d_1,r_2,d_2)$ of $\sigma_\alpha$-stable holomorphic triples of exist. Moreover, if $r_1+r_2$ and $d_1+d_2$ are coprime and $\alpha$ is generic then $\mathcal{M}_{\alpha}(r_1,d_1,r_2,d_2)$ is projective. 

\brem Note that Corollary \ref{TRIALPHA} agrees with the necessary conditions for the existence of $\sigma_{\alpha}$-stable objects  {\cite[Theorem.\ 6.1]{GP2}}.
\erem

We will now prove the support property for this type of pre-stability conditions. 

\bl \label{SP1}  Let $\sigma=(Z,\A)=\gl_{12}(\sigma_1,\sigma_2)$ be a pre-stability condition, such that there is \linebreak  $g=(T,f)\in \GL$ with $\sigma_1=\sigma_2 g$ and $T^{-1}=\begin{bmatrix}
-A& B\\
0& C\\
\end{bmatrix}.$ Then $\sigma$ satisfies the support property and therefore it is a Bridgeland stability condition. 
\el
\bdem First note that we can linearly extend $Z$ to $\mathcal{N}(\T_{C})\otimes \R\cong \R^4.$

We denote $Z_1([E])=Z_1([i^*(E)])$ and $Z_2([E])=Z_2([j^!(E)]).$ As mentioned before, since $i^*$ and $j^!$ are exact, they induce homomorphisms on the Grothendieck groups. We use the notation given above for elements $v\in \mathcal{N}(\T_{C})\otimes \R,$ i.e.\ $Z_1(v)$ and $Z_2(v).$
We claim that for $\delta=\frac{-CA}{B^2}>0,$ the pre-stability condition $\sigma$ satisfies the support property with the quadratic form $Q\colon \R^4\la \R,$ defined as 
$$Q(v)=-\Re(Z_1(v))\Im(Z_2(v)+\Im(Z_1(v))\Re(Z_2(v)) +\Im(Z_1(v))\Im(Z_2(v))+ \delta(\Re(Z_1(v))\Re(Z_2(v)),$$ where $\Re(\alpha)$ and $\Im(\alpha)$ are the real and the imaginary parts respectively and $\alpha \in \R^4.$
Note that under the notation of equation \eqref{NotCPPD}, 
we obtain that
$$Q(v)=-Ad_1r_2-Cd_2r_1-Br_1r_2+Cr_1r_2+\delta(-Ad_1d_2-Bd_2r_1).$$

We first show that it is negative definite on $$\Ker(Z)=\{v\in \R^4\mid \Re(Z_1(v))=-\Re(Z_2(v))\textnormal{ and } \Im(Z_1(v))=-\Im(Z_2(v))\}.$$
In fact, we get that \begin{eqnarray} \nonumber 
Q(v)&=&\Re(Z_2(v))\Im(Z_2(v))-\Im(Z_2(v))\Re(Z_2(v)) -\Im(Z_2(v))\Im(Z_2(v))- \delta(\Re(Z_2(v))\Re(Z_2(v))\\\nonumber
&=&  -\Im(Z_2(v))^2-\delta(\Re(Z_2(v))^2)< 0\nonumber
\end{eqnarray}if $v\neq 0.$ Note that if $Q(v)=0,$ for $v\in \Ker(Z),$ then $Z_i(v)=0$ for $i=1,2.$ It follows that $v=0,$ as $\rank(\N(D^b(C)))=2.$

Let $E=E_1\xrightarrow{\varphi} E_2$ be a $\sigma$-stable object. We consider the following short exact sequence \linebreak  $0\la j_*(E)\la E\la i_*(E)\la 0.$ If $[\varphi]=0,$ Remark \ref{TL12} implies that either $E_1=0$ or $E_2=0,$ as $E$ is indecomposable. If $E_1=0$ or $E_2=0,$ we clearly have that $Q([E])=0.$

If $\Im(Z_1(E))=0,$ it follows that $-\Re(Z_1(E))>0.$ We claim that $\Im(Z_2(E))=0.$ Indeed, after considering the decomposition of $E_2$ with respect to the standard torsion pair $( \T,\mathcal{F} )$ given in Remark \ref{TPSPG1} on $\A,$ we have that there is a subtriple $T_2=0\la T(E_2)$ of $E,$ such that $\Im(Z_2(T_2))=0.$ If $T_2\neq 0,$ then 
$1=\phi(T_2)\leq \phi(E)\leq 1,$ which implies that $\phi(E)=1,$ thus $\Im(Z_2(E))=0.$ 
If $T_2=0,$ then $E_2\in F$ and the morphism $\varphi=0,$ as $E_1$ is in $T.$ This implies that $E_1\la 0$ is a subtriple and once again we obtain that $$1=\phi(E_1\la 0)\leq \phi(E)\leq 1.$$ So, $\phi(E)=1$ and $\Im(Z_2(E))=0$ and this implies that $E=0.$ We obtain that $\Im(Z_2(E))=0$ and $-\Re(Z_2(E))>0.$ Hence, we get $Q(E)=\delta(\Re(Z_1(E))\Re(Z_2(E)))>0.$

If $\Im(Z_2(E))=0,$ then $-\Re(Z_2(E))>0.$ By the short exact sequence above, we obtain that $$1=\phi(0\la  E_2)\leq \phi(E_1\la 0)\leq 1.$$ This implies $\Im(Z_1(E))=0$ and $-\Re(Z_1(E))>0.$ As a consequence, we get  $$Q(E)=\delta(\Re(Z_1(E))\Re(Z_2(E)))>0.$$

Now we assume $\Im(Z_1(E)),\Im(Z_2(E))\neq 0$ and therefore $[\varphi]\neq 0$

Let $y=\frac{-\Re(Z_2(E)))}{\Im(Z_2(E))}$ and $x=\frac{-\Re(Z_2(j_*i^*(E))}{\Im(Z_2(j_*i^*(E)))}.$ Note that $\Im(Z_1(E))\neq 0,$ implies $\Im(Z_2(j_*i^*(E)))\neq 0.$
After dividing $Q([E])$ into $\Im(Z_2(E))\Im(Z_2(j_*i^*(E)))$ and  by Equation \eqref{SameheartG}, we obtain that \linebreak $Q(v)\geq 0$ if and only if $$-Ax-Cy-B+C+\delta(-Axy-By)\geq 0.$$
First note that by the semistability of $E$ and Lemma \ref{COTASSPGG} we have that $Cy+Ax\leq -B$ and $0\leq y-x.$

%$$Q(v)=(-A\Re(Z_2(v_1)))+B\Im(Z_2(v_1))\Im(Z_2(v))+C\Im(Z_2(v_1))\Re(Z_2(v))+C\Im(Z_2(v))\Im(Z_2(v_1))+\delta(-A\Re(Z_2(v_1))\Re(Z_2(v))-B\Im(Z_2(v_1)\Re(Z_2(v))).$$

Our proof falls naturally into two cases:

\textbf{Case 1}: $y\geq 0.$
As $C>0,$ we have $0\leq Cy\leq -Ax-B.$ Then $-Axy-By\geq 0$ and since $-Ax-Cy-B+C\geq 0,$ we get that $Q([E])\geq 0.$

\textbf{Case 2} If $y<0,$ then $x<y<0.$ Moreover if $-Ax-B\leq 0,$ we have that $-Axy-By\geq 0,$ and we argue as before. If $0<-Ax-B,$ as $-A>0$ we get $0<-Ax-B\leq -Ay-B<-B.$ Then $Cy\leq -Ax-B\leq -B$ and $\frac{B}{-A}<y <0.$ Let us consider the following function 

\begin{eqnarray} \nonumber
f\colon(0, -B)\times (\frac{B}{-A},0)& \la & \R \\ \nonumber
(x,y) &\mapsto & \frac{-Ax-Cy-B+C}{-Axy-By}. \nonumber
\end{eqnarray}
Note $B\leq -Ax-Cy,$ then $0<C\leq-Ax-Cy-B+C.$  Also $\frac{B^2}{A} < y(-Ax-B),$ as $xB>0.$ Therefore, we obtain $\frac{-A}{-B^2}>\frac{1}{y(x+\alpha)}.$ It follows that $f(x,y)<\frac{C}{y(x+\alpha)}<\frac{-CA}{-B^2}.$
Since $\delta=\frac{-CA}{B^2},$ we have that $f(x,y)<-\delta.$ Hence, we get $-Ax-Cy-B+C+\delta(-Axy-By)>0.$\edem
 
\subsubsection*{Support property for CP-glued pre-stability conditions with negative discriminant}

Let $\sigma=(Z,\A)=\gl_{12}(\sigma_1,\sigma_2)$ be a pre-stability condition, where $\sigma_1=(Z_1,\A_1)$ and \linebreak $\sigma_2=(Z_2,\A_2),$ such that there is $g=(T,f)\in \GL$ with $\sigma_1=\sigma_2g$ where $(T,f)$ satisfies that $M\coloneqq T^{-1}=\begin{bmatrix}
-A& B\\
-D& C\\
\end{bmatrix}$
with $$\Delta(M)=\Tr(M)^2-4\det(M)=(A+C)^2-4BD<0\textnormal{ and } 0<f(0)<1.$$

\brem \label{CLANDG} If $E=E_1\xrightarrow{\varphi}E_2$ is a $\sigma$-semistable object with
$\Im(Z_2(E_1))<0.$ Then $E_1\in \A_2[1].$
\erem
%\bdem By the definition of gluing $E_2\in \A_2.$ By \hyperref[NDUPAG]{ Claim \ref*{NDUPAG}},  we have that for all $h\in \GL,$ the object $\sigma h$ is a CP-glued pre-stability condition.
%As $\sigma_1=\sigma_2 g,$ if we act by $g_1^{-1},$ then $\sigma'=\sigma g^{-1}=\gl_{12}(\sigma_2, \sigma_2 g_1^{-1})=(Z',\A').$ Since $E$ is $\sigma'$-semistable, then $E\in \A'[n]$ for some $n.$ By \hyperref[CP]{ Lemma \ref*{CP}}, we obtain  $E_1\in \A_2[n],$ because $\A_2$ is equal to $\mathcal{P}_1(f^{-1}(0),f^{-1}(1)]=\mathcal{P}_2(0,1],$ where 
%$\mathcal{P}_i$ give us the slicing of $\sigma_i,$ for $i=1,2.$ As $0\leq f(0) <1,$ we obtain that $\A_2\subseteq \mathcal{P}_1(-1,1],$ which implies $n=0,1.$ As $\Im(Z_2(E_1))<0,$ we get that $E_1\in \A_2[1].$
%\edem

\bl \label{SPCPNDG} Let $\sigma=(Z,\A)$ be as above. Then $\sigma$ is a Bridgeland stability condition.  
\el
\bdem We just need to prove that it satisfies the support property. Note that we can extend $Z$ to $\mathcal{N}(\T_\A)\otimes \R\cong \R^4.$  We claim that $\sigma$ satisfies the support property with respect to the  quadratic form \linebreak $Q\colon \R^4\la \R$ defined as $$Q(v)= -\Im(Z_2(j^!(v))))\Re(Z_2(i^{*}(v)))+\Im(Z_2(i^*v)))\Re(Z_2(j^!(v)).$$
We use the following notation 
$$d_1=-\Re(Z_2(i^{*}(v))) \textnormal{ and } r_1=\Im(Z_2(i^*(v))),$$
$$d_2=-\Re(Z_2(j^!(v))) \textnormal{ and } r_2=\Im(Z_2(j^!(v))).$$

Note that $r_2\geq 0.$ We first show that $Q$ is negative definite on $$\Ker(Z)=\{v\in \R^4 \mid \Re(Z_2(j^!v))=-\Re(Z_1(i^*v)) \textnormal{ and } \Im(Z_2(i^*v))=-\Im(Z_1(j^!v))\}.$$ 
Note that $Z_1(w)=-A\Re(Z_2(w))+B\Im(Z_2(w))+i(-D\Re(Z_2(w))+C\Im(Z_2(w))),$ with $w\in \R^2.$ We get that $Q(v)= -Dd_1^2+(A+C)d_1r_1+Br_1^2<0.$ Indeed, as $0\leq f(0)=t<1$ and \linebreak  $t<1<t+1,$ we have $D>0.$ Moreover since $\Delta(M)=(A+C)^2-4BD<0,$ then $B>0$ and $-Dd^2-(A+C)dr-Br^2<0$ for all $(r,d)\in \R^2.$ If $(r,d)=0,$ we obtain that $v=0.$

Now let $E=E_1\xrightarrow{\varphi} E_2\in \A$ be a $\sigma$-stable object. We show that $Q([E])\geq 0.$ First of all, if $[E]=v\in \N(\T_\A)$ with  $\Im(Z_2(v))=0,$ then $\Re(Z_2(v))< 0$ and  
after considering the following short exact sequence $j_*(E_2)\la E \la i_*(E_1)\la j_*(E_2)[1]$ in $\A,$ we obtain $$1=\phi(0\la E_2)\leq \phi(E_1\la 0)\leq 1.$$ Then $\phi(E_1\la 0)=1,$ and this implies that $Cr_1+Dd_1=0$ and $r_1\leq 0.$  As $-d_2< 0,$ we get $$Q([E])=-d_2r_1\geq 0.$$

If $r_1=0,$ as $Cr_1+Dd_1\geq 0,$ then $Dd_1\geq0$ and $d_1\geq 0.$
Therefore, we have $Q([E])=d_1r_2\geq 0.$

Let us now assume that $r_2,r_1\neq 0.$ From the short exact sequence mentioned above and the correspondence between slope and phase, we obtain 
$$\frac{d_2}{r_2}\leq \frac{-Ad_1+Br_1+d_2}{Cr_1+Dd_1+r_2} \textnormal{ i.e.\ }-Dd_1d_2-Cr_1d_2-Ad_1r_2-Br_1r_2\geq 0. $$ 

We define $x=\frac{d_1}{r_1}$ and $y=\frac{d_2}{r_2}.$ Let us consider two cases:

\textbf{Case 1:} $r_1>0.$ Since by definition $Cr_1+Dd_1\geq 0,$ then  $C+Dx\geq 0.$
Moreover we have $C+Dx>0.$ Indeed, if $C+Dx=0,$ we get $E_1\in \A_2[1]$ and $r_1\leq 0,$ which contradicts our assumption. Then, we obtain
$y\leq \frac{-Ax-B}{Dx+C}.$

As $\Delta(M)<0$ and $D> 0,$ then for all $t\in \R$ we have that $Dt^2+(A+C)t+B>0.$
Thus, we obtain $x>\frac{-Ax+B}{Dx+C}.$ Finally, we get $x>y$ or equivalently $Q([E])=d_1r_2-d_2r_1>0.$

\textbf{Case 2:} $r_1<0.$ We claim that $[\varphi]=0.$ Indeed, by Remark \ref{CLANDG}, we have that $E_1\in \A_2[1].$ Then $\varphi\in\Hom_{D^b(\A)}(E_1,E_2)=0,$ as $\A_2$ is the heart of a bounded t-structure in $D^b(\A).$ 
This implies that $E_1=0$ or $E_2=0,$ and as a consequence, $Q([E])=0.$
\edem
  
We will now prove that all CP-glued pre-stability conditions on $\T_C$ satisfy the support property. 

\bp \label{SPGG} If the pair $\sigma=(Z,\mathcal{A})$ on $\T_C$ is a pre-stability condition with ${\sigma=\gl_{12}(\sigma_1,\sigma_2),}$ then it satisfies the support property and therefore, it is a Bridgeland stability condition. 
\ep

\bdem  
There is ${g=(T,f)\in \GL}$  with ${(Z_1,\A_1)=\sigma_1=\sigma_2 g}$ and $\sigma_2=(Z_2,\A_2).$ The proof falls naturally into the following cases:

\textbf{Case 1}: If $f(0)\geq 1,$ then this follows directly from Lemma \ref{SP2G}.

%It implies that $\Hom_{\T_\A}^{\leq 1}(i_*(\A_1),j_*(\A_2))=0.$ Indeed, by definition of CP-gluing, we have that $$\Hom_{\T_\A}^{\leq 0}(i_*(\A_1),j_*(\A_2))=0.$$ We need to show that $Hom_{\T_\A}^{1}(i_*(\A_1),j_*(\A_2))=0.$ By hypothesis we have that $\mathcal{P}_1(0,1]=\mathcal{P}_2(f(0),f(1)]\subseteq \mathcal{P}_2(1,\infty),$ then if $E_1\in \A_1$ and $E_2\in \A_2$ we obtain by adjointess that $$\Hom_{\T_\A}(i_*(E_1),j_*(E_2)[1])=\Hom_{D^b(\A)}(E_1,E_2)$$ As $E_1\in \mathcal{P}_1(0,1]$ and $E_2\in \mathcal{P}_2(0,1],$ then $\phi_{\sigma_2}^{-1}(E_1)\geq \phi_{\sigma_2}^{+}(E_2),$ which implies that 
%$$\Hom_{D^b(\A)}(E_1,E_2)=0$$ 

\textbf{Case 2}: If $0\leq f(0)<1$ and $\Delta(M)\geq 0,$ then  we have the existence of real eigenvalues $\lambda_1,\lambda_2\in \R.$ If $\lambda_1,\lambda_2> 0,$ then by Lemma \ref{SP1} and the fact that the support property is stable under the $\GL$-action, we have that $\sigma$ has the support property. If $\lambda_1,\lambda_2< 0,$ then by arguments along the lines of Proposition \ref{12_3}  there is $h\in \GL,$ such that $\sigma h=\gl_{12}(\sigma_1h,\sigma_2h)$ with $\sigma_ih=(Z'_i,\A'_i)$ and 
$\Hom^{\leq 1}(i_*\A'_1,j_*\A'_2)=0,$ therefore it follows from
Lemma \ref{SP2G}. 

\textbf{Case 3:} If $0<f(0)<1$ and $\Delta(M)<0,$ then this case follows directly from Lemma \ref{SPCPNDG}.
\edem

\brem By Remark \ref{SPUTACTION}, after applying Serre duality, we have that any pre-stability condition $\sigma\in \Theta_i,$ with $i\in\{1,2,3\}$ satisfies the support property. By Theorem \ref{TE02}, we just need to prove the 
support property for $\sigma\in\Gamma.$
\erem
 
 We will now prove the support property for $\sigma\in \Gamma$ just when $g(C)=1.$ For $g(C)>1,$ we conjecture that it is also satisfied. We start by studying the $\sigma$-semistable objects in the pre-stability conditions in Lemma \ref{PrestaNG} under the assumption that $g(C)\geq 1.$

\subsubsection*{Semistability on non-gluing pre-stability conditions}

We will now study $\sigma$-semistable objects, where $\sigma=(Z,\A)$ satisfies that $i_*(\Co(x)),l_*(\Co(x))$ are $\sigma$-stable and $j_*(\Co(x))$ is $\sigma$-stable of phase one. After applying the $\GL$-action on $\sigma$ the objects $i_*(\Co(x)),$ $j_*(\Co(x))$ and $l_*(\Co(x))$ are always $\sigma$-stable.   We use the description of the hearts given in Lemma \ref{BH}, Lemma \ref{BH1} and Lemma \ref{BH2} to prove that all $\sigma$-semistable objects have a particular form.

\bl \label{SSNG} If $E=E_1\xrightarrow{\varphi} E_2\in \A$ is $\sigma$-semistable then 
$E\in\TCoh(C)$ or $E\in \HH_{23}[-1]$ or $E\in \HH_{31}[-1].$
\el 
\bdem First note that there are elements $g_1,g_2\in \GL$ with $\delta_i=\sigma g_i=(W_i,\B_i)$ for $i=1,2,$ such that $\delta_1$ satisfies that $i_*(\Co(x))$ is $\delta_1$-stable of phase one, $j_*(\Co(x)),l_*(\Co(x))$ are $\delta_1$-stable and $\delta_2$ satisfies that $l_*(\Co(x))$ is $\delta_2$-stable of phase one and $j_*(\Co(x)),i_*(\Co(x))$ are $\delta_2$-stable.
We have $E\in \A.$ Note that Lemma \ref{BH} describes the cohomology of $E.$ For an object $G=G_1\xrightarrow{\psi} G_2\in \B_1,$ Lemma \ref{BH1} describes its cohomology and if $G=G_1\xrightarrow{\psi} G_2\in \B_2,$ in Lemma \ref{BH2} we obtain the description of its cohomology.

As $E$ is $\delta_i$-semistable we have that $E\in \B_i[n],$ for $i=1,2$ and $n\in \Z,$ where the only possible cases are $n=1,0,-1,-2.$ We study the non-trivial cases.
\begin{enumerate}[leftmargin=0.5cm]
    \item If $E\in \A\cap \B_1[-2] \cap \B_2[-1],$ as the intersection is contained in $\Coh_3(C)[-1],$ then $E\in \HH_{23}[-1].$
    \item If $E\in \A\cap \B_1[-1]\cap \B_2[-1],$ then  $E\in \HH_{23}[-1].$
    \item If $E\in \A\cap \B_1[-1]\cap \B_2,$ then $E\in \HH_{31}[-1].$ 
    \item If $E\in \A\cap \B_1\cap \B_2[-1],$ as the intersection is contained in $\Coh_2(C),$ then $E\in \TCoh(C).$
    \item If $E\in \A\cap \B_1\cap \B_2,$ then  $E\in \TCoh(C).$
    \item If $E\in \A\cap \B_1\cap \B_2[1],$ as the intersection is contained in $\Coh_1(C),$ then  $E\in \TCoh(C).$
\end{enumerate}
\edem

%\textbf{Case 1:}  $E\in \A\cap \B_1[-2] \cap \B_2 [-2].$ Since the intersection is trivial, then $E=0$.

%\textbf{Case 3:}  $E\in \A\cap \B_1[-2]\cap \B_2.$
%Since the intersection is trivial, then $E=0$.

%\textbf{Case 4:}  $E\in \A\cap \B_1[-2]\cap \B_2[1].$
%Since the intersection is trivial, then $E=0$.

%\textbf{Case 5:}  $E\in \A\cap \B_1[-1]\cap \B_2[-2].$ Since the intersection is trivial, then $E=0$.

%\textbf{Case 2:} 

%\textbf{Case 3:} 

%\textbf{Case :} $E\in \A\cap \B_1[-1]\cap \B_2[1].$ Since the intersection is trivial, then $E=0$.

%\textbf{Case 9:} $E\in \A\cap \B_1\cap \B_2[-2].$ 
%Since the intersection is trivial, then $E=0$

%\textbf{Case 4:}

%\textbf{Case 5:} 

%\textbf{Case 6:} 

%\textbf{Case 13:} If $E\in \A\cap \B_1[1]\cap \B_2[i],$ with $i=-2,-1,0,1.$ Since $\A \cap \B_1[1]$ is trivial, then $E=0.$ 

\subsection{Support property for non-gluing pre-stability conditions with negative discriminant and \texorpdfstring{$g=1$}{TEXT}}
\label{SPNGND}

We consider pre-stability conditions $\sigma=(Z,\A)$ as constructed in Lemma \ref{PrestaNG} with $\Delta<0,$ i.e.\  there is $\sigma_1=(Z_1,\Coh_1^{r_1}(C))=(T,f)\in \GL$ such that $\Coh_1^{r_1}(C)\subseteq \A$ and $\restr{Z}{\Coh_1^{r_1}(C)}=Z_1,$ with $-1<f(0)<0.$ We assume $\Delta(M)<0,$ where $M=T^{-1}.$ Under the assumption that $g(C)=1,$ these pre-stability conditions satisfy the support property and as a consequence, they are Bridgeland stability conditions.  

Since $\Delta(M)<0,$  after applying the $\GL$-action we never obtain a CP-glued pre-stability condition, because $f(\theta)<\theta$ for all $\theta\in \R.$ However, under the assumption $g=1,$ the quadratic form induced by the Euler bilinear form $-\chi(E,E)=d_2r_1-d_1r_2 $ is negative definite on $\Ker(Z)$ as in Lemma \ref{SPCPNDG}.
Therefore, it is a good candidate for being a quadratic form appearing in the support property. 

We will now prove some useful statements about $\A.$

%\blue{They are not going to appear in this position in the final version, note that they do not depend on the genus.}  

%Now let  By \hyperref[Span]{Lemma \ref*{Span}} the set $$G_2=\{j_*(\Co(x))\mid x\in C \textnormal{a closed point }\},$$ is an spanning class of  $D_2\cong D^b(C).$ Moreover $G_2\subseteq \A$   and  by adjointness $\Hom(_{\T_C}(j_*(\Co(x)),F[p])=\Hom_{D^b(C)}(\Co(x), F_2[p]).$ By Serre duality, we have that $$\Hom_{D^b(C)}(\Co(x), F_2[p])\cong \Hom_{D^b(C)}(F_2[p-1],\Co(x))^{*}$$
%As $p\geq 2,$ then $F_2[p]\in D^{\leq -1}$ and $\Co(x)\in D^{\geq 0},$ where $(D^{\leq 0},D^{\geq 0})$ is the standard t-structure, as $$\Hom_{D^b(C)}(D^{\leq -1},D^{\geq 0})=0$$

%\bdem We consider the semiorthogonal decomposition $\T_C=\lin D_2,D_3\rin.$ By \hyperref[Span]{Lemma \ref*{Span}} the set $$G_3=\{l_*(\Co(x))\mid \textnormal{x a closed point in C }\}$$ is an spanning class of  $D_3\cong D^b(C).$ Moreover $G_3\subseteq \A[1]$   and 
%$\Hom(l_*(\Co(x)),F[p])=\Hom$

%As $$

%y \hyper{https://arxiv.org/pdf/1703.10839.pdf}{\cite{}[Lemm. 4.3]} 
%$\A\cap D_2$ is the heart of a bounded t-structure in $D_2,$ as $\Coh_2(C)\subset \A\cap D_2,$ then $A\cap D_2=\Coh_2(C)$

\bl \label{SND} If $F\in\Coh(C)$ is $\mu$-stable, then $i_*(F),$ $j_*(F)$ and $l_*(F)$ are $\sigma$-stable.
\el
\bdem The proof goes along the lines of Proposition \ref{2of3}. 
If $F$ is $\mu$-stable, it is either a skyscraper sheaf of a torsion-free sheaf. In Proposition \ref{2of3} we proved the statement for skyscraper sheaves, therefore  we assume that $F$ is torsion-free and $r>0.$
Let us assume that $i_*(F)$ is not $\sigma$-semistable. Then, we consider the last triangle of its Harder--Narasimhan filtration $E\la i_*(F)\la A\la E[1]$
which satisfies $\Hom_{\T_C}^{\leq n}(E,A)=0,$ with $n\leq 0.$ By Lemma \ref{GKRT} we have that $\Hom_{D^b(C)}(E_1,A_1)=0$ and $E_1,A_1\in \Coh(C).$ By Serre duality on $C$, we get  $\Hom_{D^b(C)}(A_1,E[1])=0,$ which implies that the short exact sequence $$0\la E_1\la F\la A_1\la 0$$ splits. As $F$ is $\mu$-stable, it is indecomposable. Therefore, either $E_1=0$ or $A_1=0.$ Following exactly the same steps as in  Lemma \ref{2of3} with $X=F,$ we show that $A_1=0$ and $j_*(F),l_*(F)$ are $\sigma$-stable.
As in Lemma \ref{HNX}, the last triangle of the HN-filtration of $i_*(F)$ is given by $$l_*(F)\la i_*(F)\la j_*(F)[1]\la l_*(F)[1]$$ and this implies that 
$\phi(l_*(F))> \phi(j_*(F))+1.$ As $j_*(F)\in \Coh_2(C)\subseteq \A,$
then $1<\phi(j_*(F))+1\leq 2,$ which implies that  $1<\phi(l_*(F)).$
Moreover, note that $\phi(l_*(F))<2.$ By the stability of $l_*(\Co(x))$ we have that $1<\phi_4<2$ and we also have a non-zero morphism $l_*(F)\la l_*(\Co(x)).$

By Lemma \ref{THI} we have that $\Coh^{r_3}(C)\subseteq \A,$ so that $l_*(F)[-1]\in \A.$
Moreover, by the correspondence between slope and phase, we have that 
$$\phi(j_*(F))=\frac{d}{r}<\frac{Ad+Br-d}{-Cr-Dd-r}=\phi(l_*(F)[-1]),$$
which implies $-Dd^2-(A+C)dr-Br^2<0$ and induces a contradiction, because \linebreak $\Delta(M)=(A+C)^2-4BD<0$ and $D,B<0.$ Therefore, we obtain that $i_*(F)$ is $\sigma$-semistable. 
Now we assume that $i_*(F)$ is strictly-semistable, by Remark \ref{JHX} we obtain exactly the same contradiction. Therefore, we get that $i_*(F)$ is $\sigma$-stable. Analogously, we prove that $j_*(F)$ and $l_*(F)$ are \linebreak $\sigma$-stable. 
\edem

\bl \label{SSIGMA} If $F\in \Coh(C)$ is $\mu$-semistable, then $i_*(F),$ $j_*(F)$ and $l_*(F)$ are $\sigma$-semistable.
\el 

\bdem If $F=\Co(x),$ the statement is already proved by Lemma \ref{NGSCStability}. Therefore, we assume that $F$ is torsion-free. Let us consider a JH-filtration of $F$ with respect to $\mu.$ Note that all the $\mu$-stable factors $A_i,$ for $i=0,\cdots , n,$ have the same slope $\mu(F).$ By Lemma \ref{SND}, we obtain that  $j_*(A_i)$ is $\sigma$-stable in $\A.$ As $\restr{Z}{\Coh_2(C)}=Z_{\mu},$ we  have that $\phi(j_*(A_i))=\phi(j_*(F))=\lambda,$ with $\lambda \in \R.$ Since the category $\mathcal{P}(\lambda)$ is closed under extensions, we obtain that $j_*(F)$ is $\sigma$-semistable. 
Note that since $F$ is $\mu$-semistable, then $i_*(F)$ is in $\A$ or in $\A[1]$ and $l_*(F)$ is in $\A$ or in $\A[1].$ Analogously, the same conclusion can be drawn for $i_*(F)$ and $l_*(F).$
\edem 

%\bl \label{Hinter} \hyper{https://arxiv.org/pdf/1703.10839.pdf}{\cite{}[Lemm. 4.3]} Let $\mathcal{H}\subseteq \T$ be a heart of a bounded t-structure and $\T=\lin T_1, T_2\rin.$ Assume that the spanning class $G$ of $T_2$ satisfies $G\subseteq \mathcal{H}$, and $\Hom(G, F[p]) = 0$ for all $G \in G$, $F \in \mathcal{H}$, and all $p > 1.$ Then $T_1 ∩ \mathcal{H}$ is the heart of a bounded t-structure on $T_1.$
%\el

%\bp \label{Span} \cite{}[Prop. 3.17] Let $X$ be a smooth projective variety. Then the objects of the form $\Co(x)$ with $x \in X$ a closed point span the derived category $D^b(X)$.

\brem \label{Heart2NG} Note that $\A\cap \C_2 =\Coh_2(C).$ 
\erem 

%\bdem Let  $E\in \A.$ In  \hyperref[BH]{Lemma \ref*{BH}}, we obtain the description of the cohomology of $E.$ \linebreak If $E\in D_2,$ then ${H^0(E_1)=H^1(E_1)=0,}$ By considering the long exact sequence of cohomology
%induced by the triangle $E_1\xrightarrow{\varphi} E_2\la C(\varphi)\la E_1[1],$ we obtain that  $${H^{-1}(C(\varphi))=H^1(E_2)=0},$ which implies that $E\in \Coh_2(C).$ 
%\edem 

\bl \label{HSHND} Let $E=E_1\xrightarrow{\varphi} E_2\in \TCoh(C)$ be a $\sigma$-semistable in $\A$. Then $\Hom_{\T_C}(E,E[2])=0.$
\el 
\bdem 
First note that $E\in\TCoh(C)\cap \A=\FF,$ where $\A=( \FF,\T[-1])$ as in Lemma \ref{BH}. By Serre duality $\Hom_{\T_C}(E,E[2])=\Hom_{\T_C}(E[2],\Ser_{\T_C}(E))^*=\Hom(E[1],E_2\la C(\varphi))^*.$ It suffices to prove that $\Hom(E[1],E_2\la C(\varphi))=0.$

Let us consider the triangle 
$i_*(C(\varphi))\la \Ser_{\T_C}(E)[-1] \la j_*(E_2)\la i_*(C(\varphi))[1].$ Applying $\Hom$ induces a long exact sequence and as a consequence, it is enough to prove that $ \Hom_{\T_C}(E[1],i_*(E_2))=0 \textnormal{ and } \Hom_{\T_C}(E[1],j_*(C(\varphi)))=0.$

We have $\Hom_{\T_C}(E[1],i_*(E_2))=\Hom_{D^b(C)}(E_1[1],E_2)=0$ as  $E_1\textnormal{, } E_2 \in \Coh(C).$  We will now prove that $\Hom_{\T_C}(E[1],j_*(C(\varphi)))=0.$

\textbf{Case 1:} If $\Ker(\varphi)= 0,$ we obtain that $C(\varphi)=\Coker(\varphi)$ and  by adjointness we get
$$\Hom_{\T_C}(E[1],j_*(C(\varphi)))= \Hom_{D^b(C)}(\Coker(\varphi)[1],\Coker(\varphi))=0.$$ 

\textbf{Case 2:} $\Ker(\varphi)\neq 0.$ It is enough to show that $\phi(E)+1 >\phi^{+}(j_*(C(\varphi))).$
Let us compute $\phi^{+}(j_*(C(\varphi))).$ By Remark \ref{Heart2NG} we have that $\A\cap \C_2=\Coh_2(C),$ so $j_*(C(\varphi))\notin \A.$ As a consequence, we first need to consider its filtration in the t-structure induced by $\A,$ given by 
$$
\begin{tikzpicture}[description/.style={fill=white,inner sep=1.5pt}]
    \matrix (m) [matrix of math nodes, row sep=2em,
    column sep=1.0em, text height=1.0ex, text depth=0.25ex]
    {0 & & j_*(\Ker(\varphi))[1]  & & j_*(C(\varphi)) \\
& j_*(\Ker(\varphi))[1] & & j_*(\Coker(\varphi)). \\ };
    %\draw[double,double distance=5pt] (m-1-1) Ð (m-1-3);
       \path[->]
           (m-1-3) edge node[auto] {} (m-1-5)
		   (m-1-1) edge node[auto] {} (m-1-3)
         (m-1-5)   edge node[auto] {} (m-2-4)
                  (m-1-3)   edge node[auto] {} (m-2-2)
    (m-2-2) edge[dashed] node[auto] {} (m-1-1)
(m-2-4) edge[dashed] node[auto] {} (m-1-3);
\end{tikzpicture}
$$
By definition, $\phi^{+}(j_*(C(\varphi)))=\phi^{+}(j_*(\Ker(\varphi)))+1.$
Let us consider the HN-filtration \linebreak $0\subseteq H_1\cdots \subseteq H_{n-1}\subseteq H_{n}=j_*(\Ker(\varphi))$
of $j_*(\Ker(\varphi))$ in $\A$ with respect to $\sigma.$  Note that by Lemma \ref{SSIGMA}, if $F\in \Coh(C)$ is $\mu$-semistable then $j_*(F)$ is also $\sigma$-semistable. Therefore if we consider the HN-filtration of $\Ker(\varphi)$ with respect to $\mu,$ as $\Coh_2(C)\subseteq \A$ and $\restr{Z}{\Coh_2(C)}=Z_{\mu},$ it will give us the HN-filtration of $j_*(\Ker(\varphi))$ in $\A$ with respect to $\sigma.$ By the uniqueness of the HN-filtration, we deduce that $H_i\in \Coh_2(C),$ for all $i=0,\dots, n.$ Moreover, we have that $H_1\neq 0$ is $\sigma$-semistable and  $\phi(H_1)=\phi^{+}(j_*(\Ker(\varphi)))=\phi^{+}(j_*(C(\varphi)))-1.$
 
Let $H_1=0\la F_1,$ with $F_1\in \Coh(C).$ By definition $F_1\subseteq \Ker(\varphi).$ 
As $\Ker(\varphi)\la 0$ is a subobject of $E$ in $\TCoh(C)$ and $\FF$ is closed under subobjects, we obtain that $F_1\la 0\in \FF\subseteq \A.$ By Lemma \ref{SSIGMA}, as $F_1$ is $\mu$-semistable, $i_*(F_1)$ is also $\sigma$-semistable. Moreover, we have a non-zero morphism $i_*(F_1)\la E.$ As they are both $\sigma$-semistable this implies that $\phi_{\sigma}(i_*(F_1))\leq \phi_{\sigma}(E).$

Let $d=\deg(F_1)$ and $r=\rank(F_1).$ By the definition of $\FF,$ we get that
$\Ker(\varphi)$ and therefore $F_1$ is torsion-free and  $r>0.$ As $i_*(F_1)\in \FF,$ we also have that $Cr+Dd\geq 0.$ Moreover, we have $\phi(H_1)<\phi(i_*(F_1)).$
As a consequence, we obtain $\phi^{+}(j_*(C(\varphi)))-1=\phi(j_*(F_1))<\phi(E)$ as  we wanted to prove.  
\edem

%\bdem
%We have two cases, the first one  $Cr_1+Dd_1=0$ then $\phi(i_*(F_1))=1.$ As $r_1>0,$ then $\phi(H_1)<1.$
%Therefore, we obtain $\phi(H_1)<\phi(i_*(F_1)).$
%If $Cr_1+Dd_1>0,$ then $\frac{d}{r}<\frac{-Ad-Br}{Cr+Dd}$ if and only if $Dd^2+(A+C)dr+Br^2<0.$
%Due to the fact that $\Delta(M)<0$ and $D,B<0,$ we obtain that $Dx^2+(A+C)xy+By^2<0$ for all $x,y\in \R.$  By the correspondence between slope and phase, we obtain that $\phi(H_1)<\phi(i_*(F_1)).$
%\edem

We now use Serre duality to obtain the same results for the $\sigma$-stable objects $E\in \HH_{23}[-1],$ $\HH_{31}[-1].$ 

\bc \label{SNDH}If $E$ is $\sigma$-stable, then  $\Hom_{\T_C}(E,E[2])=0.$
\ec

Now we compute the Euler form for all $\sigma$-stable objects. 

\bl \label{ECN} If $E$  is a $\sigma$-stable object, then $-\chi(E,E)=d_2r_1-d_1r_2\geq 0.$ 
\el
\bdem The proof falls naturally into two cases:

\textbf{Case 1:} $[\varphi]=0.$
It implies that either $E_1=0$ or $E_2=0.$ If not it would contradict that $E$ is $\sigma$-stable. It clearly follows that $-\chi(E,E)=0.$ 

\textbf{Case 2:} $[\varphi]\neq 0.$ As $E\in \A$ and $\A$ is the heart of a bounded t-structure, we have that $\Hom_{\T_C}(E,E[n])=0$ for all $n<0.$ By Corollary \ref{SNDH} we have that $\Hom_{\T_C}(E,E[2])=0$ and as $\TCoh(C)$ has homological dimension $2,$ we get that $\HH_{23}[-1]$ and $\HH_{31}[-1]$ also have homological dimension 2.
 Therefore, it follows that $\Hom_{\T_C}(E,E[n])=0$ for $n\geq 2.$ As a consequence, we obtain 
$-\chi(E,E)=-\hom_{\T_C}(E,E)+\hom_{\T_C}(E,E[1]).$

As $E$ is $\sigma$-stable, it follows that $-\hom_{\T_C}(E,E)=-1.$ To prove our claim, it suffices to show that $\hom_{\T_C}(E,E[1])>0.$ By Serre duality $\Hom_{\T_C}(E,E[1])=\Hom(E[1],\Ser_{\T_C}(E))^{*}$ where \linebreak $\Ser_{\T_C}(E)=E_2[1]\xrightarrow{i_E[1]} C(\varphi)[1].$ We have $\hom_{\T_C}(E,\Ser_{\T_C}(E))>0$ because there is a non-zero morphism $E\la\Ser_{\T_C}(E)[-1].$ Therefore, $-\chi(E,E)=d_2r_1-d_1r_2>0.$

%As $E$ is $\sigma$-stable, it follows that $-\hom_{\T_C}(E,E)=-1.$ To prove our claim, it suffices to show that $\hom_{\T_C}(E,E[1])>0.$ By Serre duality $\Hom_{\T_C}(E,E[1])=\Hom(E[1],\Ser_{\T_C}(E))^{*}$ where $\Ser_{\T_C}(E)=E_2[1]\xrightarrow{i_E[1]} C(\varphi)[1].$ Since there is a non-zero morphism $E\la\Ser_{\T_C}(E)[-1],$ we have $\hom_{\T_C}(E,\Ser_{\T_C}(E))>0,$ and therefore $-\chi(E,E)=d_2r_1-d_1r_2>0.$
\edem

\bp\label{SPNGND1} Let $\sigma=(Z,\A)$ be a pre-stability condition as in Lemma \ref{PrestaNG} with  $\Delta(M)<0.$ Then it satisfies the support property and therefore it is a Bridgeland stability condition. 

\bdem We claim that $\sigma$ satisfies the support property with respect to the following quadratic form 
$Q(r_1,d_1,r_2,d_2)=d_2r_1-d_1r_2.$ 
We first show that $Q$ is negative definite on $\Ker(Z)=\{(r_1,d_1,r_2,d_2)\mid d_2=Ad_1+Br_1 \textnormal{ and } r_2=-Cr_1-Dd_1\}.$ Let $(r_1,d_1,r_2,d_2)\in \Ker(Z),$ then $Q(r_1,d_1,r_2,d_2 )= Dd_1^2+(A+C)d_1r+Br_1^2<0$ as $-1< f(0)=r<0,$ we have that $D< 0.$ Moreover since $\Delta(M)=(A+C)^2-4BD<0,$ then $B<0$ and  $Dd_1^2+(A+C)d_1r_1+Br_1^2<0$ for all $(r_1,d_1)\in \R^2.$

Let $E= E_1\xrightarrow{\varphi} E_2$ be a $\sigma$-semistable object. By \cite[Lemma\ A.6]{BMS3} it is enough to show that $Q(E)\geq 0$ for $\sigma$-stable objects. By Claim \ref{ECN} we have that $d_2r_1-d_1r_2\geq 0.$
\edem
\ep

\brem To prove the support property of the pre-stability conditions with negative discriminant in the case $g>1,$  we would need to prove \hyperref[HSHND]{Lemma \ref*{HSHND}}. But this would not be enough as we cannot use the Euler form $-\chi(E,E)=d_2r_1-d_1r_2-(1-g)(r_1^2+r_2^2-r_1r_2)$ directly, because it is not negative definite on $\Ker(Z).$
\erem

\bt \label{SP12g=1} Let $g=1$ and $\sigma\in \Theta_{12}$ be a pre-stability condition. Then it satisfies the support property and therefore is a Bridgeland stability condition.
\et
\bdem If $\sigma\in \Theta_i,$ this follows directly from Proposition \ref{SPGG} and if $\sigma\in \Gamma,$ then it follows from Proposition \ref{SPNGND1}.
\edem 

\section{Topological description of \texorpdfstring{$S_{12}$}{TEXT}}
\label{TPC}

It is now our purpose to study the topology of $\Stab(\T_C),$ we proceed by defining the following sets \begin{alignat}{8} \nonumber
S_{12}&=&\{\sigma\in \Stab(\T_{C})&\mid& i_*(\Co(x)), j_*(\Co(x)), i_*(\Ot_C), j_*(\Ot_C)\: \:  \sigma\textnormal{-stable} \textnormal{ for all closed points $x\in C$}&\}&,\\\nonumber
S_{23}&=&\{\sigma\in \Stab(\T_C)&\mid& j_*(\Co(x)), l_*(\Co(x)), j_*(\Ot_C), l_*(\Ot_C)  \: \: \sigma\textnormal{-stable}\textnormal{ for all closed points $x\in C$} &\}&, \\\nonumber
S_{31}&=&\{\sigma\in \Stab(\T_C)&\mid& i_*(\Co(x)), l_*(\Co(x)), i_*(\Ot_C), l_*(\Ot_C)  \:  \:  \sigma\textnormal{-stable}  \textnormal{ for all closed points $x\in C$}&\}&,
\end{alignat}

Throughout the whole section, we assume that every pre-stability condition $\sigma\in \Gamma$ is also a Bridgeland stability condition, i.e.\ it satisfies the support property.  
The aim of this section is to prove that $S_{12}$ is an open, connected four dimensional complex manifold. It is based on the well-behaved wall and chamber decompositions of the space of stability conditions. See \linebreak \cite[Proposition\ 3.3]{BMLP} or \cite[Proposition\ 9.3]{BS8}. Let us consider a connected component $\Stab^{\dagger}(\T_C)$ of $\Stab(\T_C).$

\bl\label{open12} The set $S_{12}\cap \Stab^{\dagger}(\T_C)$ is open in $\Stab^{\dagger}(\T_C).$ 
\el
\bdem Let $S=\{i_*(\Co(x)),j_*(\Co(x)), i_*(\Ot_C), j_*(\Ot_C)\mid  x\in C\}\subseteq \T_C.$ First note that the classes of $i_*(\Co(x)),$ $j_*(\Co(x)),$ and $i_*(\Ot_C)$ and $j_*(\Ot_C)$ in $K(\T_C)$ are primitive.  By \cite[Proposition\ 3.3]{BMLP}, we have a well-behaved wall and chamber decomposition.
We consider the set $\Theta^{\dagger}$ of points $\sigma\in \Stab^{\dagger}(\T_C)$ at which all objects of $S$ are $\sigma$-stable. We will now prove that $\Theta^{\dagger}$ is open. 
Let $B\subseteq \Stab^{\dagger}(\T_C)$ be a compact set, we show that $F=\{\sigma\in B\mid \textnormal{ not every } E\in S\textnormal{ is stable in }  \sigma\}$ is a closed set. As in \cite[Proposition\ 9.4]{BS8} this follows from the fact that $F=\cup^{n}_{j=0}\bar{C_j},$ where each $C_j$ is a chamber in which some $E\in S$ is not stable. 
\edem

Let $g_1,g_2\in \GL,$ with $g_1=(T_1,f_1)$ and $g_2=(T_2,f_2).$ We denote by $$M_i=T_i^{-1}=\begin{bmatrix}
-A_i & B_i \\
-D_i & C_i 
\end{bmatrix} \textnormal{ for } i=1,2.$$
Let us consider the subset $$\mathcal{P}_{12}=\{(\sigma_1,\sigma_2)\in \GL^2\mid \phi_0<\phi_2+1, \phi_1<\phi_3+1 \textnormal{ and if } \phi_0>\phi_2 \textnormal{, then  } \det(M_1+M_2)>0 \},$$
where $f_1(0)=n+\theta_1,$ $f_2(0)=m+\theta_2$, $n,m\in\Z$ and $\theta_1,\theta_2\in [0,1),$ $\rho(\sigma_1)=(m_0,m_1,\phi_0,\phi_1)$ and $\rho(\sigma_2)=(m_2,m_3,\phi_2,\phi_3),$ as we explained in Remark \ref{DATA}.

%\bl If $\sigma=(Z,\A)\in S_{12},$ then for each object given by $F_*(\Co(x))$ or $F_*(\Ot_C),$ with $F=i,j,$ there is $n\in\Z$ such that they belong to $\A[n].$
%\el
%\bdem If $\sigma\in \Theta_{12},$ then $F_*(\Co(x))$ or $F_*(\Ot_C),$ for $F=i,j$ are all stable and it follows from the definition. If $\sigma\in \Theta_{23},$ then $j_*(\Co(x))$ and $j_*(\Ot_C)$ are stable and it follows from the definition. As $i_*(\Co(x))$ is also stable, it is enough to study $i_*(\Ot_C)$ if it $\sigma$-stable, we are done. If it is not, then $\phi_5\geq \phi_3+1,$

%then we show that $$
%and there is $n\in Z$ with $i_*(\Ot_C)\in \A[n].$
%Analogously, if $\sigma\in \Theta_{31}.$
%\edem

\brem For $\sigma\in S_{12},$ we consider
\begin{eqnarray} \nonumber
\phi_0=\phi(i_*(\Co(x))) & \textnormal{ and } &\phi_1=\phi(i_*(\Ot_C))\\
\phi_2=\phi(j_*(\Co(x))) & \textnormal{ and } &\phi_3=\phi(j_*(\Ot_C)).
\end{eqnarray}
\erem

\bl For every $\sigma\in S_{12},$ we have that $\phi_1<\phi_0<\phi_1+1\textnormal{ and } \phi_3<\phi_2<\phi_3+1.$
\el
\bdem
 %If $\sigma\in\Theta_{12},$
It follows directly from the stability of $i_*(\Co(x)), i_*(\Ot_C)$ and $j_*(\Co(x)), j_*(\Ot_C).$ 
\edem

 %If $\sigma\in \Theta_{23},$  the stability of $j_*(\Co(x)), j_*(\Ot_C)$ implies $\phi_3<\phi_2<\phi_3+1.$  Note that $i_*(\Ot_C)$ is not necessarily stable. If $i_*(\Ot_C)$ is stable it follows directly from its stability that  $$\phi_1<\phi_0<\phi_1+1$$  . If $i_*(\Ot_C)$ is not stable, then by \hyperref[JHX]{Lemma  \ref*{JHX}} we get $\phi_5\geq \phi_1\geq\phi_3+1$ . Since $\sigma\in S_{12}\cap \Theta_{23},$ then  $i_*(\Co(x)), j_*(\Co(x))$ and $l_*(\Co(x))$ are stable, it follows  $$\phi_4<\phi_0<\phi_2+1.$$ Combining all the equation above, we obtain $$\phi_1\leq\phi_5<\phi_4<\phi_0 \textnormal{ and } \phi_0\leq\phi_2+1<\phi_3+2<\phi_1+1,$$  then $$\phi_1<\phi_0<\phi_1+1.$$ 
 
 %If $\sigma\in \Theta_{31},$ from the stability of $i_*(\Co(x)), i_*(\Ot_C).$ follows $\phi_1<\phi_0<\phi_1+1.$  If $j_*(\Ot_C)$ is stable it follows directly from its stability that $\phi_3<\phi_2<\phi_3+1$ . If $j_*(\Ot_C)$ is not stable, then by \hyperref[JHX]{Lemma  \ref*{JHX}} we obtain $\phi_1-1\geq \phi_3\geq\phi_5$. Since $\sigma\in S_{12}\cap \Theta_{23,}$ then  $i_*(\Co(x)), j_*(\Co(x))$ and $l_*(\Co(x))$ are stable, it follows  $$\phi_0-1<\phi_2<\phi_4.$$ Combining all the equation above, we obtain $$\phi_3\leq\phi_1-1<\phi_0-1<\phi_2 \textnormal{ and } \phi_2\leq\phi_4<\phi_5+1\leq\phi_3+1,$$  then $$\phi_3<\phi_2<\phi_3+1.$$ 
 
Since  every $\sigma\in S_{12}$ satisfies $\phi_1<\phi_0<\phi_1+1\textnormal{ and } \phi_3<\phi_2<\phi_3+1,$ then by Remark \ref{DATA} for $(m_0,m_1,\phi_0,\phi_1),$ where 
$m_0=\abs{ Z(i_*(\Co(x)))} \textnormal{ and } m_1=\abs{ Z(i_*(\Ot_C))}$
and for $(m_2,m_3,\phi_2,\phi_3),$ where 
$m_2=\abs{Z(j_*(\Co(x)))} \textnormal{ and } m_3=\abs{ Z(j_*(\Ot_C))}.$ Therefore, we obtain two stability conditions $\sigma_1=(Z_1,\A_1)=(T_1,f_1)$ and $\sigma_2=(Z_2,\A_2)=(T_2,f_2)$ in $\Stab(C).$ 

We define the map \begin{eqnarray} 
\pi\colon  S_{12}&\la & \mathcal{P}_{12}\\\nonumber
 \sigma & \mapsto & (\sigma_1,\sigma_2).\nonumber
\end{eqnarray} Note that $\restr{Z}{	\C_1}=Z_1\textnormal{ and } \restr{Z}{	\C_2}=Z_2.$

\brem The group $\GL$ acts freely on $\mathcal{P}_{12}$  by the definition of the action.
\erem

\bl \label{m12} The map $\pi$ is well-defined, continuous, open and $\GL$-equivariant.
\el
\bdem

First note that $\GL$ also acts freely on $S_{12}.$ As $\pi$ is defined in terms of the slicing, we clearly have a $\GL$-equivariant continuous map from $S_{12}$ to $\GL\times \GL.$

We will now show that $(\sigma_1,\sigma_2)\in \mathcal{P}_{12}.$ First, we prove that $m-n\geq-1.$ Since $i_*(\Co(x)),j_*(\Co(x))$ are stable and we have a non-zero morphism $i_*(\Co(x))\la j_*(\Co(x))[1],$ it follows that $\phi_0-\phi_2<1.$ 

If $\phi_0>\phi_2,$ then by Remark \ref{JHX} we get that $l_*(\Co(x))$ is stable. We show now that in this case  $\det(M_1+M_2)>0.$ By Proposition \ref{12UTA} and the analogous propositions for $\Theta_{23}$ and $\Theta_{31},$ we obtain that  there is $g\in\GL$ such that by acting by $g$ we obtain a stability condition $\sigma^{'}=\sigma g$ such that $\pi(\sigma^{'}) =(\sigma_1 g, \sigma_{\mu}).$ Let $\sigma^{'}=\sigma_1 g=(T^{'},f^{'})$ and $M^{'}={T^{'}}^{-1}.$  By Lemma \ref{NGUTA}, we have \linebreak $\det(M^{'}+I)>0.$ Note that $M^{'}=M_2^{-1}M_1,$ therefore $$0<\det(M_2^{-1}M_1+I)=\det(M_2)\det(M_1+M_2).$$ As $\phi_3<\phi_2<\phi_3+1,$ we obtain $\det(M_2)>0.$ This implies  $\det(M_1+M_2)>0.$ Moreover, as $i_*(\Ot_C)$ and $j_*(\Ot_C)$ are $\sigma$-stable and there is a non-zero morphism $i_*(\Ot_C)\la j_*(\Ot_C)[1],$ then we immediately obtain that $\phi_1<\phi_3+1.$

Consequently, we obtain that $\pi$ is well defined.

%Alternatively, we can prove that $\phi_5<\phi_4<\phi_5+1,$ which implies $\det(M_1+M_2)>0.$ If $\sigma\in \Theta_{23}\cup \Theta_{31},$ It follows directly from the stability of $l_*(\Co)$ and $l_*(\Ot_C).$  If $\sigma\in \Theta_{12},$ we have that if $l_*(\Ot_C)$ is stable, by the same reasoning we obtain the inequality. If $l_*(\Ot_C)$ is not stable, then $\phi_3\geq\phi_5\geq\phi_1.$ As $i_*(\Co(x)), j_*(\Co(x))$ and $l_*(\Co(x))$ are stable, it follows  $$\phi_2<\phi_4<\phi_0.$$ Combining all the equation above, we obtain $$\phi_5\leq\phi_3<\phi_2<\phi_4 \textnormal{ and } \phi_4\leq\phi_0<\phi_1+1\leq\phi_5+1,$$  then $$\phi_5<\phi_4<\phi_5+1.$$ 
Since  $\pi'\colon S_{12}\la \rm{GL}^{+}(2,\R)^2$ is a local homeomorphism, where $\pi'$ maps a stability condition to its stability function, the fact that $\pi$ is a local homeomorphism follows almost directly from the fact that $p\circ \pi=\pi',$ where 
$p\colon \GL^2\la  \rm{GL}^{+}(2,\R)$ is the universal covering. 
\edem

In order to prove that the map $\pi$ is in fact a homeomorphism, we study the action of $\GL$ on $S_{12}.$
As $\GL$ acts freely on $S_{12},$ we define a section of the action
$$\VV_{12}=\{\sigma\in S_{12} \mid \sigma=\pi(\sigma_1,\sigma_2) \textnormal{ such that } \sigma_2=\sigma_{\mu} \}.$$
\bl \label{NGLV} If $\sigma\in\VV_{12}$ and $\pi(\sigma)=(\sigma_1,\sigma_{\mu})$ with $\sigma_1=(T,f)\in\GL$ then $l_*(\Co(x))$ is $\sigma$-stable if and only if $-1<f(0)<0.$
\el 
\bdem 
If $l_*(\Co(x))$ is $\sigma$-stable, since $i_*(\Co(x))$ and $j_*(\Co(x))$ are also $\sigma$-stable, we obtain that \linebreak $\phi_2<\phi_4<\phi_0<\phi_2+1$ and therefore that $n=-1$ and $-1<f(0)<0.$ 
If $l_*(\Co(x))$ is not \linebreak $\sigma$-stable, then by Remark \ref{HNX} we have $\phi_0-\phi_2\leq 0,$ which implies that $f(0)\geq 0$ and we obtain a contradiction. 
\edem

\brem \label{GLV} If $\sigma\in\VV_{12}$ and $\pi(\sigma)=(\sigma_1,\sigma_{\mu})$ and  $0\leq f(0),$ then as $l_*(\Co(x))$ is not $\sigma$-stable, by Theorem \ref{TEO1} we have $\sigma  \in \Theta_{12}.$ \erem

The image of $\mathcal{V}_{12}$ under $\pi$ is contained in $$\mathcal{L}_{12}=\{(\sigma,\sigma_{\mu})\in \GL^{2}\mid f(0)>-1\textnormal{, }3/2>f^{-1}(1/2) \textnormal{ and if} f(0)<0 \textnormal{ then }\det(M+I)>0 \}.$$
Abusing the notation we see $\mathcal{L}_{12}$ as a subset of $\GL.$

\bl The subset $\mathcal{L}_{12}\subseteq \GL$ is open and connected.
\el

\bdem By definition $\mathcal{L}_{12}$ is clearly an open subset of $\GL.$

We define $Y\subseteq \R^4$ as follows: We say that $(m_0,m_1,\phi_0,\phi_1)\in Y$ if $m_i>0, \textnormal{ } \phi_0< 2 \textnormal{, }$\linebreak  $\phi_1<\frac{3}{2} \textnormal{, } \phi_1<\phi_0< \phi_1 +1$ and if $ 1\leq\phi_0<2 \textnormal{ and } 0<\phi_1<\frac{3}{2} \textnormal{, then } \delta(m_0,m_1,\phi_0,\phi_1)>-1,$ where \begin{eqnarray} \nonumber
\delta \colon \R_{>0}\times \R_{>0} \times (1,2) \times (0,\frac{3}{2}) &\la & \R\\ \nonumber
(m_0,m_1,\phi_0,\phi_1)&\mapsto & m_0m_1\sin((\phi_0-\phi_1)\pi)-m_0\cos(\phi_0 \pi)+m_1\sin(\phi_1 \pi). \nonumber
\end{eqnarray}
Note that $Y$ is connected, because although $f$ is defined in terms of trigonometric functions, it is restricted to intervals where it behaves well.
It is easy to see that the map\begin{eqnarray}
\rho \colon \mathcal{L}_{12} & \la & Y \\ \nonumber
(T,f) &\mapsto & (m_0,m_1,f^{-1}(1),f^{-1}(\frac{1}{2})) \nonumber 
\end{eqnarray} where $m_0=|A+Di|, m_1=|B+Ci|,$ is a homeomorphism. It follows that $\mathcal{L}_{12}$ is connected.
\edem

\bp \label{SHomeo} The map \begin{eqnarray}
\pi\colon & \mathcal{V}_{12}&\la \mathcal{L}_{12}\\\nonumber
& \sigma & \mapsto \sigma_1,\nonumber
\end{eqnarray} is a homeomorphism. 
\ep
\bdem 
First, we prove that $\pi$ is injective. Let $\sigma=(Z,\A)$, $\tau=(W,\mathcal{B}) \in \VV_{12},$ such that \linebreak $\pi(\sigma)=\pi(\tau)=\sigma_1.$ 
If $0\leq f(0),$  by Remark \ref{GLV}, we obtain $\sigma, \tau\in \Theta_{12}.$ By Lemma \ref{GUTA}, it follows that $\sigma=\gl_{12}(\sigma_1,\sigma_2)=\tau.$

%If $0> f(0)>-1,$ by \hyperref[NGLV]{Claim \ref*{NGLV}}, we have that $l_*(\Co(x))$ is stable. Let $\mathcal{P},\mathcal{Q}$ be the slicing of $\sigma$ and $\tau$ respectively. In \hyperref[NGUTA]{Lemma \ref*{NGUTA}} we described this type of stability conditions. As by defintion  $Z=W$ and $d(\mathcal{P},\mathcal{Q})<1$ therefore it follows  that $\sigma=\tau.$

If $0> f(0)>-1,$ by Lemma \ref{NGLV}, we have that $l_*(\Co(x))$ is stable. In Lemma \ref{NGUTA} we described this type of stability conditions and its hearts. As by definition  $Z=W$ and $d(\mathcal{P},\mathcal{Q})<1,$ where $\mathcal{P}$ and $\mathcal{Q}$ are the slicing of $\sigma$ and $\tau$ respectively, then $\sigma=\tau$.

Thus, by \hyperref[m12]{Lemma \ref*{m12}} we already have a homeomorphism onto the image of $\VV_{12}.$ We now prove that it is in fact onto. By \hyperref[open12]{Lemma \ref*{open12}} and \hyperref[m12]{Lemma \ref*{m12}}  the image of $\VV_{12}$ is open. 
%Explain more! 
Since $\mathcal{L}_{12}$ is also connected, it is enough to prove that $\pi(\VV_{12})$ is closed.  Moreover, it contains a dense subset as the image of the discrete stability conditions constructed in \hyperref[CSCG]{Section \ref*{CSCG}} and  \hyperref[PrestaNG]{Lemma \ref*{PrestaNG}}. We prove it by contradiction. Assume that $\pi(\VV_{12})$ is not close and let us take a $\tau_1=(Z,\A)$ in the boundary of $\pi(\VV_{12})$ which does not belong to the image. 
Note that there is $\tau^{'}=(Z^{'},\A^{'})=\pi(\sigma^{'})\in \pi(\VV_{12}),$ where $\sigma'=(W',\mathcal{B'}),$ sufficiently close to $\tau_1$ such  for $W=Z(r_1,d_1)+Z_{\mu}(r_2,d_2),$ we have that $Q_{W'}$ restricted to $\Ker{W}$ is negative definite, where $Q_{W'}$ is the quadratic form given by the support property satisfied by $\sigma'.$ 
%Let $C(W')\in \R_{>0}$ be the constant appearing on the support property of $\sigma',$ i.e.\ $\vert\vert [E]\vert\vert C(W')<|W'(E)|$ for all $E$ $\sigma'$-stable for an appropiate norm in $\mathbb{R}^4$ and $\vert-\vert$ the Euclidean norm in $\Co.$ 

Consider the open subset of $ \Hom(\Z^4,\Co),$ consisting in homomorphisms whose kernel is negative definite with respect to $Q_{W'}$ and let $U$ be the connected component containing $W'.$ Then as in  \cite[Proposition A.5]{BMS3}, there is a continuos function $C\colon U\la \R_{>0}$ such that $C(Y)\in \R_{>0}$  satisfies that $\vert\vert  v \vert\vert C(Y)<|Y(v)|$ for $v\in \R^4$ with $Q_{W'}(v)\geq 0,$  an appropiate norm in $\mathbb{R}^4$ and $\vert-\vert$ the Euclidean norm in $\Co.$ 
Then there is $0<\epsilon<\frac{1}{8},$ such that $|W-W^{'}|_{\infty}\leq \sin(\pi\epsilon) C(W').$ 
As a consequence we have that $|W-W^{'}|_{\infty}<\sin(\epsilon\pi)\frac{|W'(E)|}{\vert\vert [E] \vert\vert},$ which implies $\vert\vert W-W'\vert\vert_{\sigma'}<\sin(\pi \epsilon).$ 

By Bridgeland's deformation Theorem \ref{BDT}, there is a stability condition $\sigma=(W,\mathcal{B})$ in the neighbourhood of $\sigma^{'}.$ It is possible to choose $\tau^{'}$ appropriately, such that $\sigma$ is in a desired wall. Now we assume the object $j_*(\Co(x))$ to be $\sigma$-semistable but not stable and $j_*(\Co(x))\in \mathcal{P}_{\sigma}(1).$ By Remark \ref{JHX} the Jordan-H\"older filtration is given by $i_*(\Co(x))[-1]\la j_*(\Co(x))\la l_*(\Co(x)),$ all with the same phase. 
Therefore, we obtain $i_*(\Co(x))[-1]\in \mathcal{P}_{\sigma}(1),$ which implies that $\restr{W}{\C_1}=Z=Ad+Br+iCr.$ It follows that $\tau_1$ is  a stability condition given by $(Z,\Coh(C)[n]),$ for $n\in \Z.$ By the definition of $\mathcal{L}_{12}$ we have that $f(0)=n\geq 0.$ As a consequence, the stability conditions $\tau_1$ and $\sigma_{\mu}$ satisfy the CP-gluing conditions. Moreover $\sigma_1=\gl_{12}(\tau_1,\sigma_{\mu})$ is a discrete stability condition, by Corollary \ref{CP1} and Lemma \ref{SP1} a Bridgeland stability condition in $\mathcal{V}_{12}$, satisfying 
$\pi(\sigma_1)=\tau_1.$ We obtain a contradiction. The argument goes along the same lines for the other walls. We finally obtain $\pi(\VV_{12})=\mathcal{L}_{12}.$
\edem

\bc The map \begin{eqnarray} 
\pi\colon & S_{12}&\la \mathcal{P}_{12}\\\nonumber
& \sigma & \mapsto (\sigma_1,\sigma_2)\nonumber
\end{eqnarray} is a homeomorphism.
\ec
%\bdem As $\pi$ is $\GL$-equivariant and we proved that the induced map in the sections is a homeomorphism, therefore $\pi$ is also a homeomorphism. 
%\edem

We can finally prove the Harder--Narasimhan property for non-discrete stability conditions.

%\bt\label{NDGL} Let $\sigma_1=(Z_1,\A_1)$ and $\sigma_2=(Z_2, \A_2)$  stability conditions on $D^b(C),$ such that $\Hom^{\leq 0}_{\T}(\A_1,\A_2)=0$ and $\A_i\subseteq \D_i,$ $i=1,2.$  the  pair $\sigma=(Z,\mathcal{H})$  is stability condition on $\Stab(\T_C)$ where 

%\begin{equation} \label{FunctionG}
%Z(E)=Z_1(i^*(E))+Z_2(j^{!}(E)).
%\end{equation}
%and $\mathcal{H}$ as in \hyperref[CP]{Lemma \ref*{CP}}
%\et 

\bp The pairs $\sigma=(Z,\mathcal{A})$ constructed in Corollary \ref{CP1} with the semiorthogonal decomposition $\lin \C_1,\C_2\rin$ and the pairs $\sigma=(Z_{r},\mathcal{A}_r)$ given in Lemma \ref{GNGSC} with $f^{-1}(\frac{1}{2})<\frac{3}{2}$ are Bridgeland stability conditions. 
\ep 

 \bdem Whenever $\sigma_1=(T_1,f_1)=(Z_1,\A_1)\in\Stab(C)$ and $\sigma_2$ are discrete we have already shown that the data gives us a Bridgeland stability condition. It is enough to show this whenever $\sigma_2=\sigma_{\mu}.$ 
Let $\sigma=\gl_{12}(\sigma_1,\sigma_{\mu})$ be a CP-glued pair.
As $f_1(0)\geq 0,$ we get $(\sigma_1,\sigma_{\mu})\in \Li_{12}.$ Therefore, there is a stability condition $\tau\in \VV_{12}$ such that $\pi(\tau)=(\sigma_1,\sigma_2).$ By Remark \ref{GLV}, we obtain $\tau\in \Theta_{12}$ and by Lemma \ref{GUTA}, we have that $\tau=\gl_{12}(\sigma_1,\sigma_2)=(Z,\mathcal{A}).$ As a consequence, the pair $(Z,\mathcal{A})$ gives a Bridgeland stability condition.

Let  $\sigma=(Z_{r},\mathcal{A}_r).$ We consider $\sigma_1$ as in Lemma \ref{GNGSC} with $f^{-1}(\frac{1}{2})<\frac{3}{2}.$
We get that \linebreak  $-1<f_1(0)<0$ and by hypothesis $\det(M_1+I)>0,$  we have that $(\sigma_1,\sigma_{\mu})\in \Li_{12}.$ As a consequence, there is a stability condition $\tau\in \VV_{12}$ such that $\pi(\tau)=(\sigma_1,\sigma_2).$ By Lemma \ref{NGLV}, we obtain that $l_*(\Co(x))$ is $\tau$-stable. Therefore by Lemma \ref{NGUTA}, we have that $\tau$ is given precisely by the construction in Proposition \ref{GNGSC}. Then $\tau=(Z_{r},\mathcal{A}_r)$ and as a consequence, it is a Bridgeland stability condition. 
\edem

\brem If $\sigma=(Z_{r},\mathcal{A}_r)$ is a pre-stability condition given in Lemma \ref{GNGSC} with $f^{-1}(\frac{1}{2})\geq \frac{3}{2},$ then either $i_*(\Ot_C)$ is not stable or $j_*(\Ot_C)$ is not stable. Then $\sigma$ is in  $S_{23}$ or in $S_{31}.$ Precisely, by Lemma \ref{NGSL} if $\phi_5>\frac{3}{2}$ then $j_*(\Ot_C)$ and $l_*(\Ot_C)$ are $\sigma$-stable, and if $\phi_5<1/2$ then $l_*(\Ot_C)$ and $i_*(\Ot_C)$ are $\sigma$-stable. 
As a consequence, all the already constructed pairs in Lemma \ref{GNGSC} are Bridgeland stability conditions. 
\erem

 %As $\Hom^{\leq 0}_{\T}(\A_1,\A_2)=0,$ we have that $f_1(0)\geq f_2(0),$ it implies that $n-m\geq 0.$ 

%If $\phi_0\leq\phi_2,$ then $(\sigma_1,\sigma_2)\in \mathcal{P}_{12}.$ Then there is a stability condition $\tau\in S_{12}$ such that $\pi(\tau)=(\sigma_1,\sigma_2).$  But the heart of $\tau$ has to be precisely $\mathcal{H}.$ Therefore $\tau=(\mathcal{H},Z)$ and $\sigma$ is a Bridgeland stability condition. 

%We now show that $\mathcal{P}_{12}$ is simply connected. 

%We define $$U_{0}=\{(\sigma_1,\sigma_2)\in \GL\times \GL\mid \phi_0\geq \phi_2\}$$ and $$U_1=\{(\sigma_1,\sigma_2)\in \GL\times \GL\mid\ 0\geq n-m\geq -1 \textnormal{ with } \phi_0>\phi_2, \det(M_1+M_2)>0}.$$
%Then $U_{0}$ and $U_1$ are open simply connected with path connected intersection,  as $\mathcal{P}_{12}=U_0\cup U_1,$ then $\mathcal{P}_{12}$ is simply connected.  

\bt\label{TEO2} The space of stability conditions $\Stab(\T_C)=S_{12}\cup S_{23}\cup S_{31}$ is a connected, four dimensional complex manifold.
\et 

\bdem 	Since $\VV_{12}$ is connected, this implies that $S_{12}$ is also connected. 

Moreover, $S_{12}\cap S_{23}=S_{23}\cap S_{31}=S_{12}\cap S_{31}$ is not empty. Therefore $\Stab(\T_C)$ is connected.
%As $U_0$ is simply connected, by Van-Kampen's Theorem, $\Stab(\T_C)$ is simply connected. 
\edem

\appendix
\section{Recollement}

In this appendix, we would like to stress the importance of the gluing condition \eqref{eq:glHom0} of the hearts in the definition of the CP-glued stability conditions. In {\cite[Proposition 2.8.12]{EM1}}, we showed that CP-gluing of hearts correspond to hearts constructed by classical recollement in the sense of \cite{BBD} where the gluing condition \eqref{eq:glHom0} is satisfied. We recall the definitions in \cite{BBD} below.

\begin{definition}\label{def:recollement}
Let $\mathcal{X}, \mathcal{Y}, \mathcal{D}$ be triangulated categories. $\mathcal{D}$ is said to be a \textit{recollement} of $\mathcal{X}$ and $\mathcal{Y}$ if there are six triangulated functors as in the diagram
 
\begin{center}
\recoll{\mathcal{X}}{\mathcal{D}}{\mathcal{Y}}{j_*=j_!}{j^*}{j^!}{i^*=i^!}{i_!}{i_*} 
\end{center}

such that 
\begin{enumerate}
\item $(i^*,i_*)$, $(i_!,i^!)$, $(j^*,j_*)$, $(j_!,j^!)$ are adjoint pairs;
\item $j_*$, $i_*$, $i_!$ are full embeddings;
\item $j^!\circ i_*=0$ (and thus also $i^* \circ j_*=0$ and $j^* \circ i_!=0$);
\item for every $T \in \mathcal{D}$ there are triangles
$$ \xymatrixrowsep{0.1pc}\xymatrix{ i_! i^! T \ar[r] & T \ar[r] & j_*j^* T \ar[r] & i_! i^! T [1] \\
j_! j^! T \ar[r] & T \ar[r] & i_*i^* T \ar[r] & j_! j^! T [1]. }$$
\end{enumerate}
\end{definition}

Note that the functors of the definition of recollement satisfy the following properties as a consequence of the vanishing condition (3.). 
\begin{itemize}
    \item $\prescript{\perp}{}{\mathcal{X}} = \ker(j^*)$ and $j^* \circ i_!=0$ implies that $i_!$ embeds the category $\mathcal{Y}$ as $\prescript{\perp}{}{\mathcal{X}}$.
    \item $\mathcal{X}^\perp= \ker(j^!)$ and $j^! \circ i_*=0$ implies that $i_*$ embeds the category $\mathcal{Y}$ as $\mathcal{X}^\perp$.
\end{itemize}
Hence, if $i_*$ denotes the natural embedding of $\mathcal{X}$ into $\mathcal{D}$, we have $\mathcal{Y}=\mathcal{X}^\perp$. In fact, the definition of recollement gives the two semiorthogonal decompositions $\mathcal{D} = \langle \mathcal{X}, \prescript{\perp}{}{\mathcal{X}}  \rangle$ and $\mathcal{D} = \langle \mathcal{X}^\perp ,\mathcal{X}   \rangle$ associated to an admissible full subcategory $\mathcal{X} \subset \mathcal{D}$. 

\begin{proposition}[{\cite[Proposition 2.8.2]{EM1}}]\label{prop:RecolIFFAdm}
Let $\mathcal{D}$ be a triangulated category and let $\mathcal{X} \subset \mathcal{D}$ be a full triangulated subcategory. Then, $\mathcal{D}$ is a recollement of $\mathcal{X}$ and $\mathcal{X}^\perp$ if and only if $\mathcal{X}$ is (left and right) admissible.
\end{proposition}

The following theorem shows how to construct t-structures from t-structures in the smaller subcategories.
\begin{theorem}[{\cite[Theorem 1.4.10]{BBD}}]\label{thm:BBDheart} Let $\mathcal{X}, \mathcal{Y}, \mathcal{D}$ be triangulated categories such that $\mathcal{D}$ is a recollement of $\mathcal{X}$ and $\mathcal{Y}$ and assume the notation of Definition \ref{def:recollement}. Let $(\mathcal{X}^{\leq 0}, \mathcal{X}^{\geq 0})$ and $(\mathcal{Y}^{\leq 0}, \mathcal{Y}^{\geq 0})$ be t-structures on $\mathcal{X}$ and $\mathcal{Y}$ respectively. Then there is a t-structure $(\mathcal{D}^{\leq 0}, \mathcal{D}^{\geq 0})$ in $\mathcal{D}$ defined by:
$$\begin{array}{c}
\mathcal{D}^{\leq 0}\coloneqq \{T \in \mathcal{D} \mid i^*T \in \mathcal{Y}^{\leq 0}, j^* T \in \mathcal{X}^{\leq 0} \} \\
\mathcal{D}^{\geq 0}\coloneqq \{T \in \mathcal{D} \mid i^*T \in \mathcal{Y}^{\geq 0}, j^! T \in \mathcal{X}^{\geq 0} \}.
\end{array}$$
If we write $\mathcal{A}_\mathcal{X}$ and $\mathcal{A}_\mathcal{Y}$ for the corresponding hearts in $\mathcal{X}$ and $\mathcal{Y}$ respectively, we denote by $\rec(\mathcal{A}_\mathcal{Y},\mathcal{A}_\mathcal{X}) \coloneqq \mathcal{D}^{\leq 0} \cap \mathcal{D}^{\geq 0}$.
\end{theorem}

\begin{remark}
From Remark \ref{rem:Diadmissible}, we know that all three subcategories $\mathcal{C}_i$ for $i=1,2,3$ are admissible in $\mathcal{T}_C$. Therefore, by Proposition \ref{prop:RecolIFFAdm} they define 3 recollements of $\mathcal{T}_C$. Given hearts $\mathcal{A}_i$ (resp.\ $\mathcal{A}_j$) in $\mathcal{C}_i$ (resp.\ $\mathcal{C}_j$), we will denote by $\rec_{ij}(\mathcal{A}_i,\mathcal{A}_j)$ the heart in $\mathcal{T}_C$ as recollement of $\mathcal{C}_j$ and $\mathcal{C}_i$ constructed as in Theorem \ref{prop:RecolIFFAdm} for $ij \in \{12, 23, 31\}$.
\end{remark}

In general, when a heart is constructed by CP-gluing we know it agrees with a heart constructed from recollement. See {\cite[Proposition 2.8.12]{EM1}}. On the other hand, at least in the context of holomorphic triples, a heart constructed by recollement without satisfying CP-gluing condition does not accept a stability function as shown in the example below.

\begin{lemma}[Jealousy Lemma]\label{lem:jealousy}
Let $\mathcal{A}\subset \mathcal{T}_{C}$ be a heart constructed by recollement of hearts $\mathcal{A}_i \subset \mathcal{C}_i$, $\mathcal{A}_j \subset \mathcal{C}_j$ which do not satisfy CP-gluing conditions. Then, $\mathcal{A}$ does not accept a stability function defined on $K(\mathcal{A})$, i.e.\ $Z(\mathcal{A}) \not\subset \overline{\mathbb{H}} $ for every $Z \colon K(\mathcal{A}) \rightarrow \mathbb{C}.$
\begin{proof}
We give the proof for the case of $\mathcal{T}_C$ as recollement of $\mathcal{C}_2$ and $\mathcal{C}_1$ and the other cases will follow by acting with the Serre functor $S_{\mathcal{T}_{C}}$ (or its inverse) on $\sigma$.

Let $\sigma = (Z,\mathcal{A}_{12})$ be a stability condition on $\mathcal{T}_{C}$ such that $\mathcal{A}_{12} \coloneqq\rec_{12}(\mathcal{A}_1,\mathcal{A}_2)$ is a heart in $\mathcal{T}_{C}$ defined by recollement from given hearts $\mathcal{A}_i \coloneqq \coh^{r_i}(C)$ in $\mathcal{C}_i$, for $i=1,2$ that do not satisfy the CP-gluing condition \eqref{eq:glHom0}, i.e.\ such that 
\begin{equation}\label{eq:nongluing12}
r_2-1< r_1 < r_2.
\end{equation}
 
First of all, we claim that the hearts $i_*\mathcal{A}_1$, $j_*\mathcal{A}_2$ and $l_*\mathcal{A}_1$ are in $\mathcal{A}_{12}$. Indeed, it follows from the definitions. Recall from the isomorphism of Theorem \ref{thm:stabC} that we can identify $\mathcal{A}_i = \mathcal{P}(r_i, r_i + 1]$ for $i=1,2$, where $\mathcal{P}(0,1]=\Coh(C)$. Therefore, $j_*\mathcal{P}(r_2,r_2 + 1]\subset \mathcal{A}_{12}$ since $i^* j_*=0$ and by adjunction $j^*j_*=\id$ and $j^!j_* =\id$. Also $i_*\mathcal{P}(r_1, r_1 + 1] \subset \mathcal{A}_{12}$, since $j^! i_*=0$, $i^* i_*=\id$ and using $j^* i_*= j^* j_! [1]$ and by \eqref{eq:nongluing12}, $\mathcal{P}(r_1 +1,r_1 + 2] \subset \mathcal{P}(r_2, \infty)$. Similarly, $l_*\mathcal{P}(r_1,r_1 + 1]\subset \mathcal{A}_{12}$, since $l_* = i_!$, $j^* i_!=0$, $i^* i_!=\id$ and using $j^! i_!= l^* l_*$ and by \eqref{eq:nongluing12}, $\mathcal{P}(r_1 ,r_1 + 1] \subset \mathcal{P}(-\infty, r_2+1]$.

As in Remark \ref{SCCHR}, we write each $r_i = n_i + \theta_i$ for unique $n_i \in \mathbb{Z}$ and $\theta_i \in [0, 1)$. We assume $n_2$ to be (up to shift) equal to $0$. Note that equation \eqref{eq:nongluing12} implies that either $n_1 = 0$ and $\theta_1 < \theta_2 < \theta_1 +1$ or $n_1 = -1$ and $\theta_2 < \theta_1 < \theta_2 +1$. 

If we look closely at the imaginary part of $Z$, it has the form $\Im Z(r_1,d_1,r_2,d_2)=D_1 d_1 + D_2 d_2 + C_1 r_1 + C_2 r_2$ with $C_i,D_i \in \mathbb{R}$, for $i=1,2$. The restrictions to the previous hearts are
$$ \begin{array}{l}
\Im Z \mid_{i_*\mathcal{A}_1}=D_1 d + C_1 r \\
\Im Z \mid_{j_*\mathcal{A}_2}=D_2 d + C_2 r  \\
\Im Z \mid_{l_*\mathcal{A}_1}=(D_1 + D_2 )d + (C_1 + C_2 )r
\end{array}
$$
for $d,r \in \mathbb{Z}$ with $r \geq 0$. We recall that if $Cr+Dd$ is the imaginary part of a stability function on $\Coh^\theta$, then the value $\theta$ is determined by the quotient $D / C$ which implies that $\theta_1$ is determined by the quotients $D_1 / C_1$ and $(D_2 + D_1) / (C_1 + C_2)$. But these two quotients cannot determine the same $\theta_1$ unless $\theta_1 = \theta_2$, which contradicts the assumption \eqref{eq:nongluing12}.
\end{proof}
\end{lemma}

\begin{remark}
The contradiction was a consequence of including 3 smaller hearts in the big one. The only case when this situation does not create a contradiction is for the extreme cases of CP-gluing conditions in Proposition \ref{prop: gluingcond}.
\end{remark}

\bibliographystyle{acm}
\bibliography{unified_bib}

\end{document}